\documentclass{amsart}

\usepackage{bbm,amsmath,amsfonts,amssymb}
\usepackage{enumerate}
\usepackage{pgf,tikz}
\usepackage{mathrsfs}
\usepackage{mdframed}
\usetikzlibrary{arrows}
\usepackage{float}
\usepackage{subcaption}
\usepackage[mathscr]{eucal}

\allowdisplaybreaks[1]

\newtheorem{thm}{Theorem}[section]
\newtheorem{cor}[thm]{Corollary}
\newtheorem{lem}[thm]{Lemma}

\theoremstyle{remark}
\newtheorem{rem}[thm]{\bf Remark}
\newtheorem{ex}[thm]{\bf Example}
\newtheorem{property}[thm]{\bf Property}

\theoremstyle{definition}
\newtheorem{defn}[thm]{Definition}

\numberwithin{equation}{section}

\newtheorem*{xrem}{Remark}
\newtheorem*{xex}{Example}

\def\bA{{\mathbb A}}

\def\bR{{\mathbb R}}
\def\bS{{\mathbb S}}
\def\bT{{\mathbb T}}

\def\cA{\mathcal{A}}

\def\cE{\mathcal{E}}

\def\cK{\mathcal{K}}

\def\cM{\mathcal{M}}

\def\cR{\mathcal{R}}
\def\cS{\mathcal{S}}
\def\cT{\mathcal{T}}
\def\cU{\mathcal{U}}
\def\cV{\mathcal{V}}

\def\eps{{\varepsilon}}
\def\ONE{{\mathbbm 1}}

\def\R{\mathbb{R}}

\newcommand{\supp}{\operatorname{supp}}
\newcommand{\diam}{\operatorname{diam}}
\newcommand{\dist}{\operatorname{dist}}
\newcommand{\constant}{\operatorname{constant}}

\def\tri{\triangle}
\def\tR{{\tilde R}}

\def\LL{\Lambda}
\def\tLL{{\tilde{\Lambda}}}

\def\ttLL{\tilde{{\tilde{\Lambda}}}}
\def\MM{M}
\def\cf{{c^\flat}}

\def\ct{{\tilde{c}}}
\def\ch{{\hat{c}}}
\def\cch{{\check{c}}}
\def\cd{{c_\diamond}}
\def\FF{\mathcal{F}}
\def\GG{\mathcal{N}}

\def\fTT{{\mathfrak{T}}}
\def\nn{\nu}

\def\Ga{\Gamma}
\def\tGa{{\tilde\Gamma}}

\def\fA{{\mathfrak{A}}}
\def\fT{{\mathfrak{T}}}
\def\cITt{{I_T^t}}
\def\ITt{{J_T^h}}
\def\ITh{{I_T^h}}
\def\JTh{{J_T^h}}

\def\cH{\mathfrak{N}}

\begin{document}

\title[Nonlinear spline approximation]
{Nonlinear nonnested 2-d spline approximation}

\author{M. Lind}
\address{Department of Mathematics\\Karlstad University\\
651 88 Karlstad\\ Sweden }
\email{martin.lind@kau.se}

\author{P. Petrushev}
\address{Department of Mathematics\\University of South Carolina\\
Columbia, SC 29208}
\email{pencho@math.sc.edu}




\begin{abstract}
Nonlinear approximation from regular piecewise polynomials (splines) 
supported on rings in $\R^2$ is studied.
By definition a ring is a set in $\R^2$ obtained by subtracting a compact convex set with polygonal boundary
from another such a set, but without creating uncontrollably narrow elongated subregions.
Nested structure of the rings is not assumed, however, uniform boundedness of the eccentricities
of the underlying convex sets is required.
It is also assumed that the splines have maximum smoothness.
Bernstein type inequalities for this sort of splines are proved which allow
to establish sharp inverse estimates in terms of Besov spaces.
\end{abstract}

\date{June 16, 2015}

\maketitle

\section{Introduction}\label{sec:Introduction}

Nonlinear approximation from piecewise polynomials (splines) in dimensions\\
$d>1$ is important from theoretical and practical points of view.
We are interested in characterising the rates of nonlinear spline approximation in $L^p$.
While this theory is simple and well understood in the univariate case,
it is underdeveloped and challenging in dimensions $d>1$.

In this article we focus on nonlinear approximation in $L^p(\Omega)$, $0<p<\infty$,
from regular piecewise polynomials in $\R^2$ or on compact subsets of $\R^2$ with polygonal boudaries.
Our goal is to obtain complete characterization of the rates of approximation
(the associated approximation spaces).
To describe our results we begin by introducing in more detail our

\subsection*{Setting and approximation tool}

We are interested in approximation in $L^p$, $0<p<\infty$,
from the class of regular piecewise polynomials
$\cS(n, k)$ of degree $k-1$ with $k\ge 1$ of maximum smoothness over $n$ rings.
More specifically, with $\Omega$ being a compact polygonal domain in $\R^2$ or $\Omega=\R^2$,
we denote by $\cS(n, k)$ the set of all piecewise polynomials $S$ of the form
\begin{equation}\label{def-splines-0}
S=\sum_{j=1}^n P_j\ONE_{R_j},
\quad S\in W^{k-2}(\Omega),
\quad P_j\in\Pi_k,
\end{equation}
where $R_1, \dots, R_n$ are rings with disjoint interiors.
Here $\Pi_k$ denotes the set of all algebraic polynomials of degree $k-1$ in $\R^2$
and $S\in W^{k-2}(\Omega)$ means that all partial derivatives
$\partial^\alpha S \in C(\Omega)$, $|\alpha| \le k-2$.
In the case where $k=1$, these are simply piecewise constants.

A set $R\subset \R^2$ is called a ring if $R$ is a compact convex set with polygonal boundary
or the difference of two such sets. All convex sets we consider are
with uniformly bounded eccentricity
and we do not allow uncontrollably narrow elongated subregions.
For the precise definitions, see \S\ref{subsec:setting-1} and \S \ref{subsec:setting-2}.

\subsection*{Motivation}

Our setting would simplify considerably if the rings $R_j$ in (\ref{def-splines-0})
are replaced by regular convex sets with polygonal boundaries or simply triangles.
However, this would restrict considerably the approximation power of our approximation tool.
As will be seen the piecewise polynomials as defined above with rings allow to
capture well point singularities of functions, which is not quite possible with piecewise
polynomials over convex polygonal sets.
The idea of using rings has already been utilized in \cite{CDPX}.

\smallskip

It is important to point out that our tool for approximation although regular
is highly nonlinear. In particular, we do not assume any nested structure of
the rings involved in the definition of different splines $S$ in (\ref{def-splines-0}).
The case of approximation from splines over nested (anisotropic) rings induced by
hierarchical nested triangulations is developed in \cite{DP, KP1}.

\smallskip

Denote by $S_n^k(f)_p$ the best $L^p$-approximation of a function $f\in L^p(\Omega)$
from $\cS(n, k)$.
Our goal is to completely characterize the approximation spaces
$A_q^\alpha$, $\alpha>0$, $0<q\le \infty$, defined by
the (quasi)norm
\begin{equation*}
\|f\|_{A_q^\alpha} := \|f\|_{L^p}
+ \Big(\sum_{n=1}^\infty\big(n^\alpha S_n^k(f)_p\big)^q\frac{1}{n}\Big)^{1/q}
\end{equation*}
with the $\ell^q$-norm replaced by the $\sup$-norm if $q=\infty$.
To this end we utilize the standard machinery of Jackson and Bernstein estimates.
The Besov spaces $B_\tau^{s, k}:=B_{\tau\tau}^{s, k}$ with $1/\tau=s/2+1/p$ naturally
appear in our regular setting.
The Jackson estimate takes the form: For any $f\in B_\tau^{s, k}$ 
\begin{equation}\label{Jackson-intr}
S_n^k(f)_p \le cn^{-s/2}|f|_{B_\tau^{s, k}}.
\end{equation}
For $k=1, 2$ this estimate follows readily from the results in~\cite{KP1}.
It is an open problem to establish it for $k>2$.
Estimate (\ref{Jackson-intr}) implies the direct estimate
\begin{equation}\label{dir-est}
S_n^k(f)_p \le cK(f, n^{-s/2}),
\end{equation}
where $K(f, t)=K(f, t;L^p, B_\tau^{s,k})$ is the $K$-functional induced by $L^p$ and $B_\tau^{s, k}$.

It is a major problem to establish a companion inverse estimate.
The following Bernstein estimate would imply such an estimate:
\begin{equation}\label{Bernst-est-1}
|S_1-S_2|_{B^{s,k}_{\tau}} \le cn^{s/2}\|S_1-S_2\|_{L^p}, \quad S_1, S_2\in \cS(n, k).
\end{equation}
However, as is easy to show this estimate is not valid.
The problem is that $S_1-S_2$ may have one or more uncontrollably elongated parts
such as $\ONE_{[0, \eps]\times[0, 1]}$ with small $\eps$, which create problems for the Besov norm,
see Example~\ref{example-1} below.

The main idea of this article is to replace (\ref{Bernst-est-1}) by the Bernstein type estimate:
\begin{equation}\label{Bernst-est-2}
|S_1|_{B^{s,k}_{\tau}}^\lambda \le |S_2|_{B^{s,k}_{\tau}}^\lambda + cn^{\lambda s/2}\|S_1-S_2\|_{L^p}^\lambda,
\quad \lambda:=\min\{\tau, 1\},
\end{equation}
where $0<s/2<k-1+1/p$.
This estimate leads to the needed inverse estimate:
\begin{equation}\label{inv-est}
K(f, n^{-s/2}) \le cn^{-s/2}
\Big(\sum_{\nu=1}^n \frac{1}{\nu}\big[\nu^{s/2} S_\nu(f)_p\big]^\lambda + \|f\|_p^\lambda\Big)^{1/\lambda}.
\end{equation}
In turn, this estimate and (\ref{dir-est}) yield a characterization of the associated approximation spaces
$A_q^\alpha$ in terms of real interpolation spaces:
\begin{equation}\label{character-app}
A_q^\alpha = (L^p, B_\tau^{s, k})_{\frac{\alpha}{s}, q},
\quad 0<\alpha<s, \; 0<q\le \infty.
\end{equation}
See e.g. \cite{DL,PP}.

A natural restriction on the Bernstein estimate (\ref{Bernst-est-2}) is the requirement that
the splines $S_1, S_2\in \cS(n, k)$ have maximum smoothness.
For instance, if we consider approximation from piecewise linear functions $S$ ($k=2$),
it is assumed that $S$ is continuous.
As will be shown in Example~\ref{example-maxsmooth} estimate (\ref{Bernst-est-2}) is
no longer valid for discontinuous piecewise linear functions.

The proof of estimate (\ref{Bernst-est-2}) is quite involved. To make it more understandable
we first prove it in \S\ref{sec:piece-const} in the somewhat easier case of piecewise constants
and then in \S\ref{sec:linear-splines} for smoother splines.
Our method is not restricted to splines in dimension $d=2$.
However, there is a great deal of geometric arguments involved in our proofs and to avoid
more complicated considerations we consider only spline approximation in dimension $d=2$ here.

\smallskip

\noindent
{\em Useful notation.}
Throughout this article we shall use
$|G|$ to denote the Lebesgue measure a set $G\subset \R^2$,
$G^\circ$, $\overline{G}$, and $\partial G$ will denote the interior, closure, and boundary of $G$,
$d(G)$ will stand for the diameter of $G$,
and $\ONE_G$ will denote the characteristic function of  $G$.
If $G$ is finite, then $\# G$ will stand for the number of elements of $G$.
If $\gamma$ is e polygon in $\R^2$, then $\ell(\gamma)$ will denote its length.
Positive constants will be denoted by $c_1$, $c_2$, $c'$, $\dots$ and they may vary at every occurrence.
Some important constants will be denoted by $c_0$, $N_0$, $\beta, \dots$
and they will remain unchanged throughout.
The notation $a\sim b$ will stand for $c_1\le a/b\le c_2$.

\section{Background}\label{sec:beckground}

\subsection{Besov spaces}\label{subsec:smooth-Besov-spaces}

Besov spaces naturally appear in spline approximation.

The Besov space $B^{s, k}_\tau=B_{\tau\tau}^{s, k}$, $s>0$, $k\ge 1$, $1/\tau:=s/2+1/p$
is defined as the set of all functions $f\in L^\tau(\Omega)$ such that
\begin{equation}\label{def-smooth-Besov}
|f|_{B^{s, k}_\tau}
:= \Big(\int_0^\infty\big[t^{-s}\omega_k(f, t)_\tau\big]^\tau\frac{dt}{t}\Big)^{1/\tau} <\infty,
\end{equation}
with the usual modification when $q=\infty$.
Here $\omega_k(f, t)_\tau := \sup_{|h|\le t}\|\Delta_h^kf(\cdot)\|_{L^\tau(\Omega)}$
with
$
\Delta_h^kf(x):= \sum_{\nu=0}^k (-1)^{k+\nu} \binom{k}{\nu}f(x+\nu h)
$
if the segment $[x, x+kh] \subset \Omega$
and $\Delta_h^kf(x):=0$ otherwise.

Observe that for the standard Besov spaces $B^s_{pq}$ with $s>0$ and $1\le p, q \le \infty$
the norm is independent of the index $k>s$. However, in the Besov spaces above in general $\tau <1$,
which changes the nature of the Besov space and $k$ should no longer be directly connected to $s$.
For more details, see the discussion in \cite{KP1}, pp. 202-203.

\subsection{Nonlinear spline approximation in dimension \boldmath $d=1$}\label{subsec:univar-spline-app}
For comparison, here we provide a brief account of nonlinear spline approximation in the univariate case.
Denote by $S_n^k(f)_p$ the best $L^p$-approximation of $f\in L^p(\R)$ from
the set $S(n, k)$ of all picewise polynomials $S$ of degree $<k$ with $n+1$ free knots.
Thus, $S\in S(n, k)$ if
$S=\sum_{j=1}^n P_j\ONE_{I_j}$,
where $P_j\in \Pi_k$ and $I_j$, $j=1, \dots, n$, are arbitrary compact intervals
with disjoin interiors and $\cup_j I_j$ is an interval.
No smoothness of $S$ is required.

Let $s>0$, $0<p<\infty$, and $1/\tau=s+1/p$.
The following Jackson and Bernstein estimates hold (see \cite{P}):
If $f\in L^p(\R)$ and $n\ge 1$, then
\begin{equation}\label{dir-est-d1}
S_n^k(f)_p \le cn^{-s}|f|_{B_\tau^{s, k}} 
\end{equation}
and
\begin{equation}\label{Bernst-est-d1}
|S|_{B^{s,k}_{\tau}} \le cn^{s}\|S\|_{L^p}, \quad S\in S(n, k),
\end{equation}
where $c>0$ is a constant depending only on $s$ and $p$.
These estimates imply direct and inverse estimates which allow to characterise completely
the respective approximation spaces. For more details, see \cite{P} or \cite{DL, PP}.

Several remarks are in order.
(1) Above no smoothness is imposed on the piecewise polynomials from $S(n, k)$.
The point is that the rates of approximation from smooth splines are the same as for nonsmooth splines.
A key observation is that in dimension $d=1$ the discontinuous piecewise polynomials are
infinitely smooth with respect to the Besov spaces $B_\tau^{s, k}$.
This is not the case in dimensions $d>1$ - smoothness matters.
(2)~Unlike in the multivariate case, estimates (\ref{dir-est-d1})-(\ref{Bernst-est-d1})
hold for every $s>0$.
(3) If $S_1, S_2\in S(n, k)$, then $S_1-S_2\in S(2n, k)$,
and hence (\ref{Bernst-est-d1}) is sufficient for establishing the respective inverse estimate.
This is not true in the multivariate case and one needs estimates like (\ref{Bernst-est-1}) (if valid)
or (\ref{Bernst-est-2}) (in our case).
(4) There is a great deal of geometry involved in multivariate spline approximation,
while in dimension $d=1$ there is none.

\subsection{Nonlinear nested spline approximation in dimension \boldmath $d=2$}\label{subsec:nested-spline-app}

The rates of approximation in $L^p$, $0<p<\infty$, from splines generated by
multilevel anisotropic nested triangulations in $\R^2$ are studied in \cite{DP,KP1}.
The respective approximation spaces are completely characterized in terms of Besov type spaces
(B-spaces) defined by local piecewise polynomial approximation.
The setting in \cite{DP,KP1} allows to deal with piecewise polynomials over triangulations
with arbitrarily sharp angles.
However, the nested structure of the underlying triangulations is quite restrictive.
In this article we consider nonlinear approximation from nonnested splines,
but in a regular setting. It is a setting that frequently appears in applications.

\section{Nonlinear approximation from piecewise constants}\label{sec:piece-const} 

\subsection{Setting}\label{subsec:setting-1}

Here we describe all components of our setting, including
the region $\Omega$ where the approximation will take place and the tool for approximation we consider.

\subsection*{The region \boldmath $\Omega$.}

We shall consider two scenarios for $\Omega$:
(a) $\Omega=\bR^2$ or (b) $\Omega$ is a compact polygonal domain in $\R^2$.
More explicitly, in the second case we assume that $\Omega$ can be represented as the union
of finitely many triangles with disjoint interiors obeying the minimum angle condition.
Therefore, the boundary $\partial \Omega$ of $\Omega$ is the union of finitely many polygons consisting
of finitely many segments (edges).

\subsection*{The approximation tool}

To describe our tool for approximation we first introduce {\em rings} in $\R^2$.

\begin{defn}\label{def-rings}
We say that $R\subset \R^2$ is a ring if $R$ can be represented in the form
$R=Q_1\setminus Q_2$, where $Q_2, Q_1$ satisfy the following conditions:

(a) $Q_2\subset Q_1$ or $Q_2=\emptyset$;

(b) Each of $Q_1$ and  $Q_2$ is a compact regular convex set in $\R^2$
whose boundary is a polygon consisting of no more than $N_0$ ($N_0$ fixed) line segments.
Here a compact convex set $Q\subset \R^2$ is deemed {\em regular} if
$Q$ has a {\em bounded eccentricity}, that is,
there exists balls
$B_1$, $B_2$, $B_j=B(x_j, r_j)$, such that $B_2\subset Q \subset B_1$ and
$r_1 \le c_0 r_2$, where $c_0>0$ is a universal constant.

(c) $R$ contains no uncontrollably narrow and elongated subregions, which is specified as follows:
Each edge (segment) $E$ of the boundary of $R$  can be subdivided
into the union of at most two segments $E_1$, $E_2$ ($E=E_1\cup E_2$)
with disjoint (one dimensional) interiors such that
there exist triangles $\triangle_1$ with a side $E_1$ and adjacent to $E_1$ angles of magnitude $\beta$,
and $\triangle_2$ with a side $E_2$ and adjacent to $E_2$ angles of magnitude $\beta$
such that $\triangle_j\subset R$, $j=1, 2$,
where $0<\beta \le \pi/3$ is a fixed constant.
\end{defn}

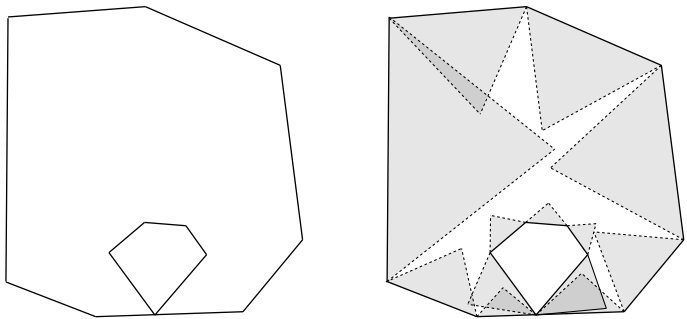
\begin{figure}[H]
\begin{center}
\begin{tikzpicture}[scale=0.55]
\clip(0.,0.) rectangle (9.,8.5);
\draw [line width=0.5pt] (4.57,7.6)-- (7.83,6.18);
\draw [line width=0.5pt] (7.83,6.18)-- (8.37,1.96);
\draw [line width=0.5pt] (8.37,1.96)-- (6.93,0.22);
\draw [line width=0.5pt] (6.93,0.22)-- (3.37,0.10);
\draw [line width=0.5pt] (3.37,0.10)-- (1.2,0.94);
\draw [line width=0.5pt] (1.2,0.95)--(1.25,7.32);
\draw [line width=0.5pt] (1.25,7.35)-- (4.57,7.6);
\draw [line width=0.5pt] (4.80,0.14)--(6.05,1.6);
\draw [line width=0.5pt] (6.05,1.6)--(5.55,2.30);
\draw [line width=0.5pt] (5.55,2.30)--(4.55,2.38);
\draw [line width=0.5pt] (4.55,2.38)--(3.69,1.65);
\draw [line width=0.5pt] (3.69,1.65)--(4.80,0.14);
\end{tikzpicture}
\begin{tikzpicture}[ scale=0.55]
\clip(0.,0.) rectangle (9.,8.5);
\draw [line width=0.5pt] (4.57,7.6)-- (7.83,6.18);
\draw [line width=0.5pt] (7.83,6.18)-- (8.37,1.96);
\draw [line width=0.5pt] (8.37,1.96)-- (6.93,0.22);
\draw [line width=0.5pt] (6.93,0.22)-- (3.37,0.10);
\draw [line width=0.5pt] (3.37,0.10)-- (1.2,0.95);
\draw [line width=0.5pt] (1.2,0.94)--(1.25,7.35);
\draw [line width=0.5pt] (1.25,7.32)-- (4.57,7.6);
\draw [line width=0.5pt] (4.80,0.14)--(6.05,1.6);
\draw [line width=0.5pt] (6.05,1.6)--(5.55,2.30);
\draw [line width=0.5pt] (5.55,2.30)--(4.55,2.38);
\draw [line width=0.5pt] (4.55,2.38)--(3.69,1.65);
\draw [line width=0.5pt] (3.69,1.65)--(4.80,0.14);
\draw (4.80,0.14)--(6.5,0.3);
\draw (6.05,1.6)--(6.5,0.3);
\fill[fill=black,fill opacity=0.1](6.05,1.6)--(4.80,0.14)--(6.5,0.3)--cycle;
\draw [dash pattern=on 1pt off 1pt] (5.55,2.30)--(6.25,2.35);
\draw [dash pattern=on 1pt off 1pt] (6.05,1.6)--(6.25,2.35);
\fill[fill=black,fill opacity=0.1](5.55,2.30)-- (6.05,1.6)--(6.25,2.35)--cycle;
\draw [dash pattern=on 1pt off 1pt] (4.55,2.38)--(5.1,2.85);
\draw [dash pattern=on 1pt off 1pt] (5.55,2.30)--(5.1,2.85);
\fill[fill=black,fill opacity=0.1](4.55,2.38)--(5.55,2.30)--(5.1,2.85);--cycle;
\draw [dash pattern=on 1pt off 1pt] (3.69,1.65)--(3.7,2.55);
\draw [dash pattern=on 1pt off 1pt] (4.55,2.38)--(3.7,2.55);
\fill[fill=black,fill opacity=0.1](3.69,1.65)--(4.55,2.38)--(3.7,2.55)--cycle;
\draw [dash pattern=on 1pt off 1pt] (4.80,0.14)--(3.15,0.4);
\draw [dash pattern=on 1pt off 1pt] (3.69,1.65)--(3.15,0.4);
\fill[fill=black,fill opacity=0.1](4.80,0.14)--(3.15,0.4)--(3.69,1.65)--cycle;
\draw [dash pattern=on 1pt off 1pt] (1.2,0.95)--(5.25,4.15);
\draw [dash pattern=on 1pt off 1pt] (1.25,7.35)--(5.25,4.15);
\fill[fill=black,fill opacity=0.1](1.2,0.95)--(1.25,7.35)--(5.25,4.15)--cycle;
\draw [dash pattern=on 1pt off 1pt] (4.57,7.6)--(3.45,5);
\draw [dash pattern=on 1pt off 1pt] (1.25,7.32)--(3.45,5);
\fill[fill=black,fill opacity=0.1] (4.57,7.6)--(1.25,7.32)--(3.45,5)--cycle;
\draw [dash pattern=on 1pt off 1pt] (1.17,0.94)--(3.0,1.75);
\draw [dash pattern=on 1pt off 1pt] (3.37,0.10)--(3.0,1.75);
\fill[fill=black,fill opacity=0.1](1.17,0.94)--(3.37,0.10)--(3.0,1.75)--cycle;
\draw [dash pattern=on 1pt off 1pt] (3.37,0.10)--(4,.8);
\draw [dash pattern=on 1pt off 1pt] (4.80,0.14)--(4,.8);
\fill[fill=black,fill opacity=0.1](3.37,0.10)--(4,.8)--(4.80,0.14)--cycle;
\draw [dash pattern=on 1pt off 1pt] (4.80,0.14)--(5.9,1.15);
\draw [dash pattern=on 1pt off 1pt] (6.93,0.22)--(5.9,1.15);
\fill[fill=black,fill opacity=0.1](4.80,0.14)-- (6.93,0.22)--(5.9,1.15);
\draw [dash pattern=on 1pt off 1pt] (6.93,0.22)--(6.2,2.15);
\draw [dash pattern=on 1pt off 1pt] (8.37,1.96)--(6.2,2.15);
\fill[fill=black,fill opacity=0.1](6.93,0.22)--(8.37,1.96)--(6.2,2.15)--cycle;
\draw [dash pattern=on 1pt off 1pt] (8.37,1.96)--(5.15,3.7);
\draw [dash pattern=on 1pt off 1pt] (7.83,6.18)--(5.15,3.7);
\fill[fill=black,fill opacity=0.1] (8.37,1.96)--(7.83,6.18)--(5.15,3.7)--cycle;
\draw [dash pattern=on 1pt off 1pt] (7.83,6.18)--(4.95,4.6);
\draw [dash pattern=on 1pt off 1pt] (4.57,7.6)--(4.95,4.6);
\fill[fill=black,fill opacity=0.1] (7.83,6.18)-- (4.57,7.6)--(4.95,4.6)--cycle;
\end{tikzpicture}
\caption{Left: a ring $R=Q_1\setminus Q_2$. Right: $R$ with the triangles associated to the segments of $\partial R$.}
\end{center}
\end{figure}

\begin{xrem}
Observe that from the above definition it readily follows that for any ring $R$ in $\R^2$
\begin{equation}\label{measure-ring}
|R|\sim d(R)^2
\end{equation}
with constants of equivalence depending only on the parameters $N_0$, $c_0$, and $\beta$.
\end{xrem}

In the case when $\Omega$ is a compact polygonal domain in $\R^2$, we assume that
there exists a constant $n_0 \ge 1$ such that $\Omega$ can be represented as the union of
$n_0$ rings $R_j$ with disjoint interiors:
$\Omega = \cup_{j=1}^{n_0} R_j$. If $\Omega=\R^2$, then we set $n_0:=1$.

We now can introduce the class of regular piecewise constants.

{\em Case 1:} $\Omega$ is a compact polygonal domain in $\R^2$.
We denote by $\cS(n, 1)$ ($n\ge n_0$) the set of all piecewise constants $S$ of the form
\begin{equation}\label{def-splines-1}
S=\sum_{j=1}^n c_j\ONE_{R_j}, \quad c_j\in\R,
\end{equation}
where $R_1, \dots, R_n$ are rings with disjoint interiors
such that $\Omega=\cup_{j=1}^n R_j$.

\smallskip

{\em Case 2:} $\Omega =\R^2$.
In this case we denote by $\cS(n, 1)$ the set of all piecewise constant functions $S$ of
the form (\ref{def-splines-1}),
where $R_1, \dots, R_n$ are rings with disjoint interiors
such that the support $R :=\cup_{j=1}^n R_j$ of $S$ is a ring in the sense of Definition~\ref{def-rings}.

\begin{xex}
A simple case of the above setting is when $\Omega=[0, 1]^2$
and the rings $R$ are of the form $R=Q_1\setminus Q_2$, where $Q_1$, $Q_2$
are dyadic squares in $\R^2$.
These kind of dyadic rings have been used in \cite{CDPX}.

A bit more general is the setting when $\Omega$ is
a regular rectangle in $\R^2$ with sides parallel to the coordinate axes
or $\Omega=\R^2$ and the rings $R$ are of the form $R=Q_1\setminus Q_2$,
where $Q_1$, $Q_2$ are regular rectangles with sides parallel to the coordinate axes,
and no narrow and elongated subregions are allowed in the sense of Definition~\ref{def-rings}~(c).
\end{xex}

Clearly the set $\cS(n, 1)$ in nonlinear since the rings $\{R_j\}$ and
the constants $\{c_j\}$ in (\ref{def-splines-1}) may vary with $S$.

We denote by $S_n^1(f)_p$ the best approximation of $f\in L^p(\Omega)$ from $\cS(n, 1)$
in $L^p(\Omega)$, $0< p<\infty$, i.e.
\begin{equation}\label{def-best-app}
S_n^1(f)_p:= \inf_{S\in \cS(n, 1)}\|f-S\|_{L^p}.
\end{equation}

\subsection*{Besov spaces}

When approximating in $L^p$, $0< p<\infty$, from piecewise constants the Besov spaces
$B^{s, 1}_{\tau}$ with 
$1/\tau=s/2+1/p$ naturally appear.
In this section, we shall use the abbreviated notation $B^s_{\tau}$ for these spaces.

\subsection{Direct and inverse estimates}\label{subsec:dir-inv-estimats}

The following {\em Jackson estimate} is quite easy to establish (see \cite{KP1}):
If $f\in B^s_\tau$, $s>0$, $1/\tau:=s/2+1/p$, $0<p<\infty$, then $f\in L^p(\Omega)$ and
\begin{equation}\label{Jackson}
S_n^1(f)_p \le cn^{-s/2}|f|_{B^s_{\tau}} \quad\hbox{for}\quad n\ge n_0,
\end{equation}
where $c>0$ is a constant depending only on $s, p$ and
the structural constants $N_0$, $c_0$, and $\beta$ of the setting.

This estimate leads immediately to the following direct estimate: If $f\in L^p(\Omega)$, then
\begin{equation}\label{direct-est}
S_n^1(f)_p \le cK(f, n^{-s/2}), \quad n\ge 1,
\end{equation}
where $K(f, t)$ is the $K$-functional induced by $L^p$ and $B^s_\tau$, namely,
\begin{equation}\label{def-K-functional}
K(f, t)=K(f, t; L^p, B^s_\tau) := \inf_{g\in B^s_\tau} \{\|f-g\|_p+t|g|_{B^s_\tau}\}, \quad t>0.
\end{equation}

The main problem here is to prove a matching inverse estimate.
Observe that the following Bernstein estimate holds: If $S\in \cS(n, 1)$, $n\ge n_0$,
and $0< p<\infty$, $0<s<2/p$, $1/\tau=s/2+1/p$, then
\begin{equation}\label{Bernstein}
|S|_{B^s_{\tau}} \le cn^{s/2}\|S\|_{L^p},
\end{equation}
where the constant $c>0$ depends only on $s, p$, and the structural constants of the setting
(see the proof of Theorem~\ref{thm:smooth-Bernstein-1}).
The point is that this estimate does not imply a companion to (\ref{direct-est}) inverse estimate.
The following estimate would imply such an estimate:
\begin{equation}\label{Bernstein-1}
|S_1-S_2|_{B^s_{\tau}} \le cn^{s/2}\|S_1-S_2\|_{L^p}, \quad S_1, S_2\in \cS(n, 1).
\end{equation}
However, as the following example shows this estimate is not valid.

\begin{ex}\label{example-1}
Consider the function $f :=\ONE_{[0, \eps]\times[0, 1]}$, where $\eps>0$ is sufficiently small.
It is easy to see that
$$
\omega_1(f, t)_\tau^\tau
\sim \left\{
\begin{array}{lll}
t \quad \hbox{if} \quad t\le \eps\\
\eps \quad \hbox{if} \quad t > \eps
\end{array}
\right.
$$
and hence for $0<s<2/p$ and $1/\tau=s/2+1/p$ we have
$$
|f|_{B^s_{\tau}} \sim \eps^{1/\tau-s} \sim \eps^{1/p-s/2} \sim \eps^{-s/2}\|f\|_{L^p},
\quad \hbox{implying} \quad
|f|_{B^s_{\tau}} \not\le c\|f\|_{L^p},
$$
since $\eps$ can be arbitrarily small.
It is easy to see that one comes to the same conclusion if $f$ is the characteristic function of
any convex elongated set in $\R^2$.
The point is that if $S_1, S_2\in \cS(n, 1)$, then $S_1-S_2$ can be a constant multiple
of the characteristic function of one or more elongated convex sets in $\R^2$
and, therefore, estimate (\ref{Bernstein-1}) is in general not possible.
\end{ex}

We overcome the problem with estimate (\ref{Bernstein-1}) by establishing the following
main result:

\begin{thm}\label{thm:Bernstein}
Let $0< p<\infty$, $0<s<2/p$, and $1/\tau=s/2+1/p$.
Then for any  $S_1, S_2\in \cS(n, 1)$, $n\ge n_0$, we have
\begin{align}
|S_1|_{B^s_{\tau}} &\le |S_2|_{B^s_{\tau}} + cn^{s/2}\|S_1-S_2\|_{L^p},
\quad\hbox{if} \;\; \tau \ge 1, \quad\hbox{and}\label{Bernstein-2}\\
|S_1|_{B^s_{\tau}}^\tau &\le |S_2|_{B^s_{\tau}}^\tau + cn^{\tau s/2}\|S_1-S_2\|_{L^p}^\tau,
\quad\hbox{if} \;\; \tau < 1, \label{Bernstein-3}
\end{align}
where the constant $c>0$ depends only on $s, p$, and the structural constants $N_0$, $c_0$, and $\beta$.
\end{thm}

In the limiting case we have this result:

\begin{thm}\label{thm:BV-Bernstein}
If  $S_1, S_2\in \cS(n, 1)$, $n\ge n_0$, then
\begin{equation}\label{BV-Bernstein}
|S_1|_{BV} \le |S_2|_{BV} + cn^{1/2}\|S_1-S_2\|_{L^2},
\end{equation}
where the constant $c>0$ depends only on the structural constants $N_0$, $c_0$, and $\beta$.
\end{thm}

We next show that estimates (\ref{Bernstein-2})-(\ref{Bernstein-3}) and (\ref{BV-Bernstein})
imply the desired inverse estimate.

\begin{thm}\label{thm:inverese-1}
Let $p$, $s$, and $\tau$ be as in Theorem~\ref{thm:Bernstein} and set $\lambda :=\min\{\tau, 1\}$.
Then for any $f\in L^p(\Omega)$ we have
\begin{equation}\label{inverse-est}
K(f, n^{-s/2}) \le cn^{-s/2}
\Big(\sum_{\ell=n_0}^n \frac{1}{\ell}\big[\ell^{s/2} S_\ell^1(f)_p\big]^\lambda + \|f\|_p^\lambda\Big)^{1/\lambda},
\quad n\ge n_0.
\end{equation}
Here $K(f, t)=K(f, t; L^p, B^s_\tau)$ is the $K$-functional defined in $(\ref{def-K-functional})$
and $c>0$ is a constant depending only on $s, p$, and the structural constants of the setting.

Furthermore, in the case when $p=2$ and $s=1$ estimated $(\ref{inverse-est})$
holds with $B^s_\tau$ replaced by $BV$ and $\lambda=1$.
\end{thm}

\begin{proof}
Let $\tau \le 1$ and $f\in L^p(\Omega)$.
We may assume that for any $n\ge n_0$ there exists $S_n\in \cS(n, 1)$ such that
$\|f-S_n\|_p = S_n^1(f)_p$.
Clearly, for any $m\ge m_0$ with $m_0:=\lceil\log_2 n_0 \rceil$ we have
\begin{equation}\label{K-est}
K(f, 2^{-ms/2}) \le \|f-S_{2^m}\|_p +2^{-ms/2}|S_{2^m}|_{B^s_\tau}.
\end{equation}
We now estimate $|S_{2^m}|_{B^s_\tau}^\tau$ using iteratively estimate (\ref{Bernstein-3}).
For $\nu\ge m_0+1$ we get
\begin{align*}
|S_{2^\nu}|_{B^s_\tau}^\tau
&\le |S_{2^{\nu-1}}|_{B^s_\tau}^\tau +c2^{\tau \nu s/2}\|S_{2^\nu}-S_{2^{\nu-1}}\|_p^\tau\\
&\le |S_{2^{\nu-1}}|_{B^s_\tau}^\tau
+c2^{\tau \nu s/2}\big(\|f-S_{2^\nu}\|_p^\tau+ \|f-S_{2^{\nu-1}}\|_p^\tau\big)\\
&\le |S_{2^{\nu-1}}|_{B^s_\tau}^\tau
+c'2^{\tau \nu s/2}S_{2^{\nu-1}}^1(f)_p^\tau.
\end{align*}
From (\ref{Bernstein}) we also have
$$
|S_{2^{m_0}}|_{B^s_\tau} \le c\|S_{2^{m_0}}\|_p
\le c\|f-S_{2^{m_0}}\|_p + c\|f\|_p = cS_{2^{m_0}}^1(f)_p + c\|f\|_p.
$$
Summing up these estimates we arrive at
$$
|S_{2^m}|_{B^s_\tau}^\tau \le c\sum_{\nu=m_0}^{m-1}2^{\tau \nu s/2}S_{2^{\nu}}^1(f)_p^\tau + c\|f\|_p^\tau.
$$
Clearly, this estimate and (\ref{K-est}) imply (\ref{inverse-est}).
The proof in the cases $\lambda>1$ or $p=2$, $s=1$, and $B^s_\tau$ replaced by $BV$ is the same.
\end{proof}

Observe that the direct and inverse estimates (\ref{direct-est}) and (\ref{Bernstein-2})-(\ref{BV-Bernstein})
imply immediately a characterization of the approximation spaces $A_q^\alpha$
associated with piecewise constant approximation from above just like in (\ref{character-app}).

\subsection{Proof of Theorems~\ref{thm:Bernstein}}\label{subsec:proof-Bernstein}

We shall only consider the case when $\Omega \subset \R^2$ is a compact polygonal domain.
The proof in the case $\Omega=\R^2$ is similar.

Assume $S_1, S_2\in \cS(n, 1)$, $n\ge n_0$.
Then $S_1, S_2$ can be represented in the form
$
S_j=\sum_{R\in \cR_j} c_R \ONE_{R},
$
where $\cR_j$ is a set of at most $n$ rings in the sense of Definition~\ref{def-rings}
with disjoint interiors and such that
$\Omega = \cup_{R\in \cR_j} R$, $j=1, 2$.

We denote by $\cU$ the set of all maximal compact connected subsets $U$ of $\Omega$ obtain by
intersecting all rings from $\cR_1$ and $\cR_2$ with the property
$\overline{U^\circ} =U$ (the closure of the interior of $U$ is $U$).
Here $U$ being maximal means that it is not contained in another such set.

Observe first that each $U\in \cU$ is obtained from the intersection of exactly
two rings $R'\in \cR_1$ and $R''\in\cR_2$,
and is a subset of $\Omega$ with polygonal boundary $\partial U$ consisting of $\le 2N_0$ line segments (edges).
Secondly, the sets in $\cU$ have disjoint interiors and
$\Omega=\cup_{U\in\cU}U$.

It is easy to see that there exists a constant $c>0$ such that
\begin{equation}\label{num-intesect}
\# \cU \le cn.
\end{equation}
Indeed, each $U\in\cU$ is obtain by intersecting two rings, say, $R'\in \cR_1$ and $R''\in \cR_2$.
If $|R'|\le |R''|$, we associate $R'$ to $U$, and
if $|R'|> |R''|$ we associate $R''$ to $U$.
However, because of condition (b) in Definition~\ref{def-rings}
every ring $R$ from $\cR_1$ or $\cR_2$ can be intersected by only finitely many,
say, $N^\star$ rings from $\cR_2$ or $\cR_1$, respectively,
of area $\ge |R|$.
Here $N^\star$ depends only on the structural constants $N_0$ and $c_0$.
Also, the intersection of any two rings may have only finitely many, say $N^{\star\star}$,
connected components.
Therefore, every ring $R\in \cR_1\cup \cR_2$ can be associated to only
$N^\star N^{\star\star}$ sets $U\in\cU$,
which implies (\ref{num-intesect}) with $c=2N^\star N^{\star\star}$.

\smallskip

Example~\ref{example-1} clearly indicates that our main problem will be in dealing with
sets $U\in \cU$ or parts of them with $\diam^2$ much larger than their area.
To overcome the problem with these sets we shall subdivide each of them using the following

\smallskip

\noindent
{\bf Construction of good triangles.}
According to Definition~\ref{def-rings}, each segment $E$ from the boundary of
every ring $R\in \cR_j$  can be subdivided into the union of at most two segments
$E_1$, $E_2$ ($E=E_1\cup E_2$) with disjoint interiors such that
there exist triangles $\triangle_1$ with a side $E_1$ and adjacent to $E_1$ angles of size $\beta >0$
and $\triangle_2$ with a side $E_2$ and adjacent to $E_2$ angles $\beta$
such that $\triangle_\ell\subset R$, $\ell=1, 2$.
We now associate with $\triangle_1$ the triangle $\tilde\tri_1 \subset \tri_1$ with one side $E_1$
and adjacent to $E_1$ angles of size $\beta/2$;
just in the same way we construct the triangle $\tilde\tri_2 \subset \tri_2$ with a side $E_2$.
We proceed in the same way for each edge $E$ from $\partial R$, $R\in \cR_j$, $j=1, 2$.
We denote by $\cT_R$ the set of all triangles $\tilde\tri_1$, $\tilde\tri_2$ associated
in the above manner with all edges $E$ from $\partial R$.
We shall call the triangles from $\cT_R$ {\em the good triangles} associated with $R$.
Observe that due to $\tri_1, \tri_2 \subset R$ for the triangles from above it readily follows that
the good triangles associated with $R$ ($R\in \cR_j$, $j=1, 2$) have disjoint interiors;
this was the purpose of the above construction.

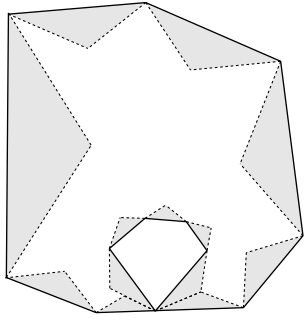
\begin{figure}[H]
\begin{center}
\begin{tikzpicture}[scale=0.55]
\clip(0.,0.) rectangle (9.,8.5);
\draw [line width=0.5pt] (4.57,7.6)-- (7.83,6.18);
\draw [line width=0.5pt] (7.83,6.18)-- (8.37,1.96);
\draw [line width=0.5pt] (8.37,1.96)-- (6.93,0.22);
\draw [line width=0.5pt] (6.93,0.22)-- (3.37,0.10);
\draw [line width=0.5pt] (3.37,0.10)-- (1.2,0.95);
\draw [line width=0.5pt] (1.2,0.94)--(1.25,7.35);
\draw [line width=0.5pt] (1.25,7.32)-- (4.57,7.6);
\draw [line width=0.5pt] (4.80,0.14)--(6.05,1.6);
\draw [line width=0.5pt] (6.05,1.6)--(5.55,2.30);
\draw [line width=0.5pt] (5.55,2.30)--(4.55,2.38);
\draw [line width=0.5pt] (4.55,2.38)--(3.69,1.65);
\draw [line width=0.5pt] (3.69,1.65)--(4.80,0.14);
\draw [dash pattern=on 1pt off 1pt] (4.80,0.14)--(5.9,0.6);
\draw [dash pattern=on 1pt off 1pt] (6.05,1.6)--(5.9,0.6);
\fill[fill=black,fill opacity=0.1](6.05,1.6)--(4.80,0.14)--(5.9,0.6)--cycle;
\draw [dash pattern=on 1pt off 1pt] (5.55,2.30)--(6.15,2.15);
\draw [dash pattern=on 1pt off 1pt] (6.05,1.6)--(6.15,2.15);
\fill[fill=black,fill opacity=0.1](5.55,2.30)-- (6.05,1.6)--(6.15,2.15)--cycle;
\draw [dash pattern=on 1pt off 1pt] (4.55,2.38)--(5.05,2.7);
\draw [dash pattern=on 1pt off 1pt] (5.55,2.30)--(5.05,2.7);
\fill[fill=black,fill opacity=0.1](4.55,2.38)--(5.55,2.30)--(5.05,2.7);--cycle;
\draw [dash pattern=on 1pt off 1pt] (3.69,1.65)--(3.95,2.4);
\draw [dash pattern=on 1pt off 1pt] (4.55,2.38)--(3.95,2.4);
\fill[fill=black,fill opacity=0.1](3.69,1.65)--(4.55,2.38)--(3.95,2.4)--cycle;
\draw [dash pattern=on 1pt off 1pt] (4.80,0.14)--(3.7,0.69);
\draw [dash pattern=on 1pt off 1pt] (3.69,1.65)--(3.7,0.69);
\fill[fill=black,fill opacity=0.1](4.80,0.14)--(3.69,1.65)--(3.7,0.69)--cycle;
\draw [dash pattern=on 1pt off 1pt] (1.2,0.95)--(3.25,4.15);
\draw [dash pattern=on 1pt off 1pt] (1.25,7.35)--(3.25,4.15);
\fill[fill=black,fill opacity=0.1](1.2,0.95)--(1.25,7.35)--(3.25,4.15)--cycle;
\draw [dash pattern=on 1pt off 1pt] (4.57,7.6)--(3.17,6.49);
\draw [dash pattern=on 1pt off 1pt] (1.25,7.32)--(3.17,6.49);
\fill[fill=black,fill opacity=0.1] (4.57,7.6)--(1.25,7.32)--(3.17,6.49)--cycle;
\draw [dash pattern=on 1pt off 1pt] (1.17,0.94)--(2.61,1.1);
\draw [dash pattern=on 1pt off 1pt] (3.37,0.10)--(2.61,1.1);
\fill[fill=black,fill opacity=0.1](1.17,0.94)--(3.37,0.10)--(2.61,1.1)--cycle;
\draw [dash pattern=on 1pt off 1pt] (3.37,0.10)--(4.05,.5);
\draw [dash pattern=on 1pt off 1pt] (4.80,0.14)--(4.05,.5);
\fill[fill=black,fill opacity=0.1](3.37,0.10)--(4.05,.5)--(4.80,0.14)--cycle;
\draw [dash pattern=on 1pt off 1pt] (4.80,0.14)--(5.9,0.6);
\draw [dash pattern=on 1pt off 1pt] (6.93,0.22)--(5.9,0.6);
\fill[fill=black,fill opacity=0.1](4.80,0.14)-- (6.93,0.22)--(5.9,0.6);
\draw [dash pattern=on 1pt off 1pt] (6.93,0.22)--(7,1.55);
\draw [dash pattern=on 1pt off 1pt] (8.37,1.96)--(7,1.55);
\fill[fill=black,fill opacity=0.1](6.93,0.22)--(8.37,1.96)--(7,1.55)--cycle;
\draw [dash pattern=on 1pt off 1pt] (8.37,1.96)--(6.85,3.7);
\draw [dash pattern=on 1pt off 1pt] (7.83,6.18)--(6.85,3.7);
\fill[fill=black,fill opacity=0.1] (8.37,1.96)--(7.83,6.18)--(6.85,3.7)--cycle;
\draw [dash pattern=on 1pt off 1pt] (7.83,6.18)--(5.65,5.97);
\draw [dash pattern=on 1pt off 1pt] (4.57,7.6)--(5.65,5.97);
\fill[fill=black,fill opacity=0.1] (7.83,6.18)-- (4.57,7.6)--(5.65,5.97)--cycle;
\end{tikzpicture}
\caption{The ring from Figure 1 with good triangles (angles $=\beta/2$).}
\end{center}
\end{figure}

From now on for every segment $E$ from $\partial R$ that has been subdivided
into $E_1$ and $E_2$ as above we shall consider $E_1$ and $E_2$ as segments from $\partial R$
in place of $E$.
We denote by $\cE_R$ the set of all (new) segments from $\partial R$.
We now associate with each $E\in \cE_R$ the good triangle which has $E$ as a side
and denote it by $\tri_E$.

To summarize, we have subdivided the boundary $\partial R$ of each ring $R \in \cR_j$, $j=1,2$,
into a set $\cE_R$ of segments with disjoint interiors
($\partial R = \cup_{E\in \cE_R} E$)
and associated with each $E\in \cE_R$ a good triangle $\tri_E \subset R$ such that
$E$ is a side of $\tri_E$ and the triangles $\{\tri_E\}_{E\in \cE_R}$
have disjoint interiors.
In addition, if $E'\subset E$ is a subsegment of $E$, then we associate with $E'$
the triangle $\tri_{E'} \subset \tri_E$ with one side $E'$ and the other two sides parallel
to the respective sides of $\tri_E$;
hence $\tri_{E'}$ is similar to $\tri_E$.
We shall call $\tri_{E'}$ a good triangle as well.

\smallskip

\noindent
{\bf Subdivision of the sets from \boldmath $\cU$.}
We next subdivide each set $U\in \cU$ by using the good triangles constructed above.
Suppose $U\in \cU$ is obtained from the intersection of rings $R'\in \cR_1$ and $R''\in \cR_2$.
Then the boundary $\partial U$ of $U$ consists of two sets of segments $\cE_U'$ and $\cE_U''$,
where each $E\in \cE_U'$ is a segment or subsegment of a segment from $\cE_{R'}$ and
each $E\in \cE_U''$ is a segment or subsegment of a segment from $\cE_{R''}$.
Clearly, $\partial U = \cup_{E\in \cE_U'\cup \cE_U''} E$ and the segments from $\cE_U'\cup \cE_U''$
have disjoint interiors.
For each $E\in \cE_U'\cup \cE_U''$ we denote by $\tri_E$ the good triangle
with a side $E$, defined above.

Consider the collection of all sets of the form
$\tri_{E_1} \cap \tri_{E_2}$ with the properties:

(a) $E_1\in \cE_{U'}$, $E_2\in \cE_{U''}$.

(b) There exists an isosceles trapezoid or an isosceles triangle $T\subset \tri_{E_1} \cap \tri_{E_2}$
such that its two legs (of equal length) are contained in $E_1$ and $E_2$, respectively,
and its height is not smaller than its larger base.
We assume that $T$ is a maximal isosceles trapezoid (or triangle) with these properties.
Observe that it may happen that there are no trapezoids like this.

We denote by $\cT_U$ the set of all trapezoids as above.
We also denote by $\cA_U$ the set of all maximal compact connected subsets $A$ of
$U \setminus \cup_{T\in\cT_U} T^\circ$.
Clearly,
$U = \cup_{T\in\cT_U} T \cup_{A\in\cA_U} A$
and the sets in $\cT_U\cup \cA_U$ have disjoint interiors.

The following lemma will be instrumental for the rest of this proof.

\begin{lem}\label{lem:two-tri}
There exist  constants $c^{\star}>1$ and $\beta^\star >0$
depending only on $N_0$, $c_0$, and $\beta$, such that
if $A\in \cA_U$ for some $U\in \cU$, then
$d(A)^2 \le c^{\star} |A|$,
and there exists a triangle $\tri\subset A$ whose minimum angle is $\ge \beta^\star$
such that
$|A| \le c^\star |\tri|$.
\end{lem}

\begin{proof}
There are several cases to be considered, depending on the shapes of $U$ and $A$.
Since in each case the argument will be geometric we shall illustrate
the geometry involved in a number of figures.

\smallskip

{\em Case 1.} Let $U$ be the closure of a connected subset of $Q_1\setminus Q_0$,
where $Q_0$, $Q_1$ are convex polygonal sets just as in Definition~\ref{def-rings}.
This may happen if rings $R_0$, $R_1$ of the form
$R_0=\tilde Q_0\setminus Q_0$ and $R_1=Q_1\setminus \tilde Q_1$
intersect as illustrated in Figure 3.

\begin{figure}[H]
\begin{center}
\begin{tikzpicture}[line cap=round,scale=1.4]
\clip(1.6,0.8) rectangle (5.5,3.9);
\fill[line width=0.5pt,fill=black,fill opacity=0.1] (3.8603243657264614,1.4381262439339069) -- (4.19842002590868,1.5863325607261127) -- (4.221354207919704,2.36609474910093) -- (3.5103945655779563,3.1114556644592146) -- (2.4095538290488014,2.710107479266294) -- (2.398086738043289,1.8500756538528889) -- (2.4869862527716182,1.7918676382569592) -- (2.486986252771618,2.612665188470064) -- (3.4219325942350336,2.9502847006651858) -- (4.097171618625278,2.3269871396895767) -- (4.097171618625278,1.651748115299332) -- cycle;
\draw [line width=0.5pt] (2.3857891683227113,0.8695381013350946)-- (4.507125588987901,0.8033365374946858);
\draw [line width=0.5pt] (4.507125588987901,0.8033365374946858)-- (5.06124069041828,1.383781359834387);
\draw [line width=0.5pt] (5.06124069041828,1.383781359834387)-- (5.048227497568886,2.2800818416329443);
\draw [line width=0.5pt] (5.048227497568886,2.2800818416329443)-- (4.598994549325088,3.4330183611270697);
\draw [line width=0.5pt] (4.598994549325088,3.4330183611270697)-- (3.1749643409969446,3.8499292083136156);
\draw [line width=0.5pt] (3.1749643409969446,3.8499292083136156)-- (1.8450630323108483,3.0397626523515076);
\draw [line width=0.5pt] (1.8450630323108483,3.0397626523515076)-- (1.7061553786768533,1.7148406920848613);
\draw [line width=0.5pt] (1.7061553786768533,1.7148406920848613)-- (2.3857891683227113,0.8695381013350946);
\draw [line width=0.5pt] (3.2016877916554765,0.8440759613964758)-- (2.4869862527716187,1.6257773835920153);
\draw [line width=0.5pt] (2.4869862527716187,1.6257773835920153)-- (2.4869862527716187,2.612665188470064);
\draw [line width=0.5pt] (2.4869862527716187,2.612665188470064)-- (3.421932594235034,2.9502847006651858);
\draw [line width=0.5pt](3.421932594235034,2.9502847006651858)-- (4.0971716186252785,2.3269871396895763);
\draw [line width=0.5pt] (4.0971716186252785,2.3269871396895763)-- (4.0971716186252785,1.6517481152993323);
\draw [line width=0.5pt] (4.0971716186252785,1.6517481152993323)-- (3.2016877916554765,0.8440759613964758);
\draw [line width=0.5pt] (2.398086738043289,1.8500756538528893)-- (2.409553829048801,2.710107479266293);
\draw [line width=0.5pt] (2.409553829048801,2.710107479266293)-- (3.5103945655779576,3.111455664459215);
\draw [line width=0.5pt] (3.5103945655779576,3.111455664459215)-- (4.221354207919704,2.3660947491009314);
\draw [line width=0.5pt] (4.221354207919704,2.3660947491009314)-- (4.19842002590868,1.586332560726112);
\draw [line width=0.5pt](4.19842002590868,1.586332560726112)-- (3.3613223825063008,1.2193856485497263);
\draw [line width=0.5pt] (3.3613223825063008,1.2193856485497263)-- (2.398086738043289,1.8500756538528893);
\draw [line width=0.5pt] (3.8603243657264614,1.4381262439339069)-- (4.19842002590868,1.5863325607261127);
\draw [line width=0.5pt] (4.19842002590868,1.5863325607261127)-- (4.221354207919704,2.36609474910093);
\draw [line width=0.5pt] (4.221354207919704,2.36609474910093)-- (3.5103945655779563,3.1114556644592146);
\draw [line width=0.5pt] (3.5103945655779563,3.1114556644592146)-- (2.4095538290488014,2.710107479266294);
\draw [line width=0.5pt] (2.4095538290488014,2.710107479266294)-- (2.398086738043289,1.8500756538528889);
\draw [line width=0.5pt] (2.398086738043289,1.8500756538528889)-- (2.4869862527716182,1.7918676382569592);
\draw [line width=0.5pt] (2.4869862527716182,1.7918676382569592)-- (2.486986252771618,2.612665188470064);
\draw [line width=0.5pt] (2.486986252771618,2.612665188470064)-- (3.4219325942350336,2.9502847006651858);
\draw [line width=0.5pt] (3.4219325942350336,2.9502847006651858)-- (4.097171618625278,2.3269871396895767);
\draw [line width=0.5pt] (4.097171618625278,2.3269871396895767)-- (4.097171618625278,1.651748115299332);
\draw [line width=0.5pt] (4.097171618625278,1.651748115299332)-- (3.8603243657264614,1.4381262439339069);
\draw [line width=0.5pt] (3.333193371161619,2.613318584940028)-- (3.0999912384532227,2.4225168399967947);
\draw [line width=0.5pt] (3.0999912384532227,2.4225168399967947)-- (3.2271924017487112,2.2741154828187238);
\draw [line width=0.5pt] (3.2271924017487112,2.2741154828187238)-- (3.4497944375158167,2.3377160644664685);
\draw [line width=0.5pt] (3.4497944375158167,2.3377160644664685)-- (3.4921948252809796,2.517917712468411);
\draw [line width=0.5pt] (3.4921948252809796,2.517917712468411)-- (3.333193371161619,2.613318584940028);

\draw (3.6,3.2) node[anchor=north west] {$Q_1$};
\draw (2.5,2) node[anchor=north west] {$Q_0$};
\draw (3.4,1.9) node[anchor=north west] {$R_1$};
\draw (3.4,2.5) node[anchor=north west] {$\tilde{Q}_1$};
\draw (4.75,3.1) node[anchor=north west] {$\tilde{Q}_0$};
\draw (4.4,2) node[anchor=north west] {$R_0$};
\end{tikzpicture}
\caption{One configuration for $U=Q_1\setminus Q_0$.}
\end{center}
\end{figure}
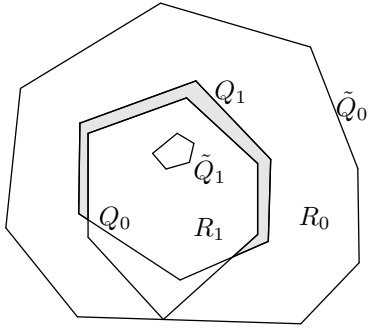

Denote
$\gamma_0:=\partial U \cap \partial Q_0$ and
$\gamma_1:=\partial U \cap \partial Q_1$.
Thus $\gamma_0$ is the ``inner" part of the boundary $\partial U$ of $U$,
which is a subset of $\partial Q_0$, and
$\gamma_1$ is the ``outer" part of $\partial U$,
which is a subset of $\partial Q_1$.
The polygons $\gamma_0$ and $\gamma_1$ may have two points of intersection, one point of intersection or none.
With no loss of generality we shall assume that $\gamma_0$ and $\gamma_1$ have two points of intersection
just as in Figure 3.

Each $\gamma_0$ and $\gamma_1$ is a polygon consisting of no more than $N_0$ segments.
For any such segment $E$ we denote by $\tri_E$ the good triangle with a side $E$
whose construction is described above.
The set $U$ with its good triangles is displayed in Figure 4.

\begin{figure}[H]
\begin{center}
\begin{tikzpicture}[scale=2.5]
\clip(5.2,1.82) rectangle (7.6,3.63);
\fill[line width=0.5pt,fill=black,fill opacity=0.1] (5.59081337827308,3.104429247297001) -- (6.287495050852281,2.9791170307528434) -- (6.691654114802241,3.5057774324899267) -- cycle;
\fill[line width=0.5pt,fill=black,fill opacity=0.1] (5.59081337827308,3.104429247297001) -- (5.846548164769672,2.6700773328146035) -- (5.579346287267565,2.244397421883602) -- cycle;
\fill[line width=0.5pt,fill=black,fill opacity=0.1] (6.691654114802241,3.5057774324899267) -- (6.830198910646261,2.933891709103345) -- (7.402613757143984,2.7604165171316466) -- cycle;
\fill[line width=0.5pt,fill=black,fill opacity=0.1] (7.402613757143984,2.7604165171316466) -- (7.169388823017498,2.3911878493093623) -- (7.379679575132958,1.9806543287568292) -- cycle;
\fill[line width=0.5pt,fill=black,fill opacity=0.1] (7.379679575132958,1.9806543287568292) -- (7.150544938996874,1.9916975080721253) -- (7.041583914950745,1.832448011964626) -- cycle;
\fill[line width=0.5pt,fill=black,fill opacity=0.1] (5.6682458019959,3.006986956500779) -- (5.4,2.6) -- (5.6682458019959,2.186189406287671) -- cycle;
\fill[line width=0.5pt,fill=black,fill opacity=0.1] (5.668245801995899,3.006986956500779) -- (6.023680674563541,3.4426765776602033) -- (6.603192143459312,3.3446064686958987) -- cycle;
\fill[line width=0.5pt,fill=black,fill opacity=0.1] (6.603192143459312,3.3446064686958987) -- (7.13170105497625,3.23916263023746) -- (7.2784311678495595,2.721308907720288) -- cycle;
line width=0.4pt,
\fill[line width=0.5pt, fill=black,fill opacity=0.1] (7.2784311678495595,2.721308907720288) -- (7.470890967347487,2.3874190725052373) -- (7.2784311678495595,2.0537338255913924) -- cycle;
\fill[line width=0.5pt,fill=black,fill opacity=0.1] (7.2784311678495595,2.0537338255913924) -- (7.222151698275247,1.8710966503401292) -- (7.041583914950745,1.832448011964626) -- cycle;

\draw  [line width=0.7pt](5.668245801995899,3.006986956500779)-- (6.60319214345931,3.3446064686958987);
\draw [line width=0.7pt](6.60319214345931,3.3446064686958987)-- (7.278431167849559,2.7213089077202883);
\draw  [line width=0.7pt](7.278431167849559,2.7213089077202883)-- (7.278431167849559,2.0537338255913906);
\draw [line width=0.7pt](7.278431167849559,2.0537338255913906)-- (7.04158391495074,1.8324480119646207);
\draw [line width=0.7pt](7.04158391495074,1.8324480119646207)-- (7.379679575132957,1.980654328756826);
\draw [line width=0.7pt](7.379679575132957,1.980654328756826)-- (7.402613757143983,2.760416517131645);
\draw [line width=0.7pt](7.402613757143983,2.760416517131645)-- (6.6916541148022395,3.5057774324899267);
\draw [line width=0.7pt](6.6916541148022395,3.5057774324899267)-- (5.590813378273079,3.104429247297002);
\draw [line width=0.7pt](5.590813378273079,3.104429247297002)--(5.579346287267566,2.2443974218836065);
\draw [line width=0.7pt](5.579346287267566,2.2443974218836065)-- (5.668245801995899,2.1861894062876766);
\draw  [line width=0.7pt](5.6682458019959,2.186189406287671)-- (5.6682458019959,3.006986956500779);

\draw [dash pattern=on 1pt off 1pt] (5.59081337827308,3.104429247297001)-- (6.287495050852281,2.9791170307528434);
\draw [dash pattern=on 1pt off 1pt] (6.287495050852281,2.9791170307528434)-- (6.691654114802241,3.5057774324899267);
\draw [dash pattern=on 1pt off 1pt](5.59081337827308,3.104429247297001)-- (5.846548164769672,2.6700773328146035);
\draw [dash pattern=on 1pt off 1pt](5.846548164769672,2.6700773328146035)-- (5.579346287267565,2.244397421883602);
\draw [dash pattern=on 1pt off 1pt] (5.579346287267565,2.244397421883602)-- (5.59081337827308,3.104429247297001);
\draw [dash pattern=on 1pt off 1pt](6.691654114802241,3.5057774324899267)-- (6.830198910646261,2.933891709103345);
\draw [dash pattern=on 1pt off 1pt](6.830198910646261,2.933891709103345)-- (7.402613757143984,2.7604165171316466);
\draw [dash pattern=on 1pt off 1pt] (7.402613757143984,2.7604165171316466)-- (7.169388823017498,2.3911878493093623);
\draw [dash pattern=on 1pt off 1pt] (7.169388823017498,2.3911878493093623)-- (7.379679575132958,1.9806543287568292);
\draw [dash pattern=on 1pt off 1pt](7.379679575132958,1.9806543287568292)-- (7.150544938996874,1.9916975080721253);
\draw [dash pattern=on 1pt off 1pt](7.150544938996874,1.9916975080721253)-- (7.041583914950745,1.832448011964626);
\draw [dash pattern=on 1pt off 1pt](5.6682458019959,3.006986956500779)-- (5.4,2.6);
\draw [dash pattern=on 1pt off 1pt](5.4,2.6)-- (5.6682458019959,2.186189406287671);

\draw [dash pattern=on 1pt off 1pt](5.668245801995899,3.006986956500779)-- (6.023680674563541,3.4426765776602033);
\draw [dash pattern=on 1pt off 1pt](6.023680674563541,3.4426765776602033)-- (6.603192143459312,3.3446064686958987);
\draw [dash pattern=on 1pt off 1pt] (6.603192143459312,3.3446064686958987)-- (7.13170105497625,3.23916263023746);
\draw [dash pattern=on 1pt off 1pt] (7.13170105497625,3.23916263023746)-- (7.2784311678495595,2.721308907720288);
\draw [dash pattern=on 1pt off 1pt](7.2784311678495595,2.721308907720288)-- (7.470890967347487,2.3874190725052373);
\draw [dash pattern=on 1pt off 1pt](7.470890967347487,2.3874190725052373)-- (7.2784311678495595,2.0537338255913924);
\draw [dash pattern=on 1pt off 1pt](7.2784311678495595,2.0537338255913924)-- (7.222151698275247,1.8710966503401292);
\draw [dash pattern=on 1pt off 1pt](7.222151698275247,1.8710966503401292)-- (7.041583914950745,1.832448011964626);
\end{tikzpicture}
\caption{The set $U$ with the good triangles associated to it.}
\end{center}
\end{figure}
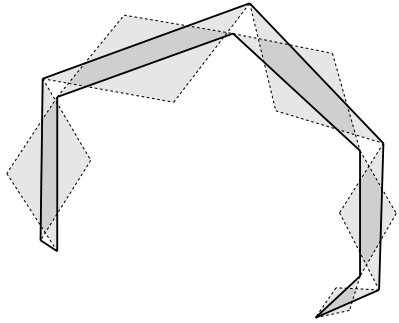

Let $E_0\subset \gamma_0$ and $E_1\subset \gamma_0$ be two edges of $\gamma_0$
such that $\tri_{E_0}\not\subset U$, $\tri_{E_1}\not\subset U$,
and either $E_0$ and $E_1$ have a common end point, say $v$ or
$E_0$ and $E_1$ are connected by a chain of segments $I_1, \dots, I_m$, $I_j\subset \gamma_0$,
such that $\tri_{I_j} \subset U$, $j=1, \dots, m$.
Denote by $v_0$ the common end point of $E_0$ and $I_1$,
and by $v_1$ the common end point of $E_1$ and $I_m$.
See Figure 5 below

\vskip 5pt

\begin{figure}[H]
\begin{center}
\begin{tikzpicture}[scale=3]
\clip(7.2,3.9) rectangle (10.6,5.95);

\fill[fill=black,fill opacity=0.1] (7.882049619464438,5.522791533950487) -- (8.01834106880721,5.313689289321217) -- (8.366535567552761,5.566486829076689) -- (8.25267278282658,5.740235415478823) -- cycle;
\fill[fill=black,fill opacity=0.1] (10.168538897165062,4.76211439753165) -- (10.30465950269715,4.863338645446963) -- (9.90636326892103,5.388804835364125) -- (9.777746742629366,5.292450388390757) -- cycle;
\draw (7.456922227024441,4.073844301181575)-- (7.459733263175691,5.275019357729001);
\draw (7.459733263175691,5.275019357729001)-- (8.604284605038165,5.9465253805186835);
\draw (8.604284605038165,5.9465253805186835)-- (9.630896607414522,5.752223830564537);
\draw (9.630896607414522,5.752223830564537)-- (10.48569991967419,4.624494764780467);
\draw (10.48569991967419,4.624494764780467)-- (10.484402213290705,3.8882271350045743);
\draw (10.484402213290705,3.8882271350045743)-- (10.3925114521977,4.458165852199138);
\draw (10.3925114521977,4.458165852199138)-- (9.60845460144276,5.52219326935398);
\draw (9.60845460144276,5.52219326935398)-- (9.462581562626308,5.640013800705728);
\draw (9.462581562626308,5.640013800705728)-- (9.307831302586154,5.6978801793444385);
\draw (9.307831302586154,5.6978801793444385)-- (8.984978159532638,5.779961486900417);
\draw (8.984978159532638,5.779961486900417)-- (8.6744373172509,5.790030868659508);
\draw (8.6744373172509,5.790030868659508)-- (7.662882677802508,5.055617989932039);
\draw (7.662882677802508,5.055617989932039)-- (7.456922227024441,4.073844301181575);
\draw [dash pattern=on 1pt off 1pt] (8.6744373172509,5.790030868659508)-- (7.956653260911887,5.714835100988549);
\draw [dash pattern=on 1pt off 1pt] (7.956653260911887,5.714835100988549)-- (7.662882677802508,5.055617989932039);
\draw [dash pattern=on 1pt off 1pt] (8.984978159532638,5.779961486900417)-- (8.83261451852636,5.874641597222815);
\draw [dash pattern=on 1pt off 1pt] (8.83261451852636,5.874641597222815)-- (8.6744373172509,5.790030868659508);
\draw [dash pattern=on 1pt off 1pt] (9.307831302586154,5.6978801793444385)-- (9.170099563565833,5.832120507647755);
\draw [dash pattern=on 1pt off 1pt] (8.984978159532638,5.779961486900417)-- (9.146404731059395,5.738920833122428);
\draw [dash pattern=on 1pt off 1pt] (9.146404731059395,5.738920833122428)-- (8.984978159532638,5.779961486900417);
\draw [dash pattern=on 1pt off 1pt] (9.170099563565833,5.832120507647755)-- (8.984978159532638,5.779961486900417);
\draw [dash pattern=on 1pt off 1pt] (9.462581562626308,5.640013800705728)-- (9.401911017248281,5.71361954217077);
\draw [dash pattern=on 1pt off 1pt] (9.401911017248281,5.71361954217077)-- (9.307831302586154,5.6978801793444385);
\draw [dash pattern=on 1pt off 1pt] (9.60845460144276,5.52219326935398)-- (9.569529939780557,5.623213454143975);
\draw [dash pattern=on 1pt off 1pt] (9.569529939780557,5.623213454143975)-- (9.462581562626308,5.640013800705728);
\draw [dash pattern=on 1pt off 1pt] (10.3925114521977,4.458165852199138)-- (10.30764128467998,5.216517277698228);
\draw [dash pattern=on 1pt off 1pt] (10.30764128467998,5.216517277698228)-- (9.60845460144276,5.52219326935398);
\draw [dash pattern=on 1pt off 1pt] (7.459733263175691,5.275019357729001)-- (8.225856025616968,5.280368856461018);
\draw [dash pattern=on 1pt off 1pt] (8.225856025616968,5.280368856461018)-- (8.604284605038165,5.9465253805186835);
\draw [dash pattern=on 1pt off 1pt] (9.630896607414522,5.752223830564537)-- (9.732750923692654,4.941598836453851);
\draw [dash pattern=on 1pt off 1pt] (9.732750923692654,4.941598836453851)-- (10.48569991967419,4.624494764780467);
\draw (7.882049619464438,5.522791533950487)-- (8.01834106880721,5.313689289321217);
\draw (8.01834106880721,5.313689289321217)-- (8.366535567552761,5.566486829076689);
\draw (8.366535567552761,5.566486829076689)-- (8.25267278282658,5.740235415478823);
\draw (8.25267278282658,5.740235415478823)-- (7.882049619464438,5.522791533950487);
\draw (10.168538897165062,4.76211439753165)-- (10.30465950269715,4.863338645446963);
\draw (10.30465950269715,4.863338645446963)-- (9.90636326892103,5.388804835364125);
\draw (9.90636326892103,5.388804835364125)-- (9.777746742629366,5.292450388390757);
\draw (9.777746742629366,5.292450388390757)-- (10.168538897165062,4.76211439753165);
\draw (8.4,5.6) node[anchor=north west] {$p_1$};
\draw (8.1,5.9) node[anchor=north west] {$p_2$};
\draw (7.7,5.7) node[anchor=north west] {$p_3$};
\draw (7.9,5.3) node[anchor=north west] {$p_4$};
\draw (8.6,5.8) node[anchor=north west] {$v_0$};
\draw (9,5.7) node[anchor=north west] {$I_j$};
\draw (9.45,5.55) node[anchor=north west] {$v_1$};
\draw (9.67,5.31) node[anchor=north west] {$q_1$};
\draw (9.85,5.55) node[anchor=north west] {$q_2$};
\draw (10.2,5.05) node[anchor=north west] {$q_3$};
\draw (9.95,4.8) node[anchor=north west] {$q_4$};
\draw (10.1,4.65) node[anchor=north west] {$E_1$};
\draw (10.35,4.9) node[anchor=north west] {$\tilde{E}_1$};
\draw (7.69,5.15) node[anchor=north west] {$E_0$};
\draw (7.5,5.55) node[anchor=north west] {$\tilde{E}_0$};
\draw (7.25,4.5) node[anchor=north west] {$\gamma_1$};
\draw (7.55,4.6) node[anchor=north west] {$\gamma_0$};
\draw (8.05,5.6) node[anchor=north west] {$T_0$};
\draw (9.9,5.2) node[anchor=north west] {$T_1$};
\end{tikzpicture}
\caption{The case $m\ge1$}
\end{center}
\end{figure}
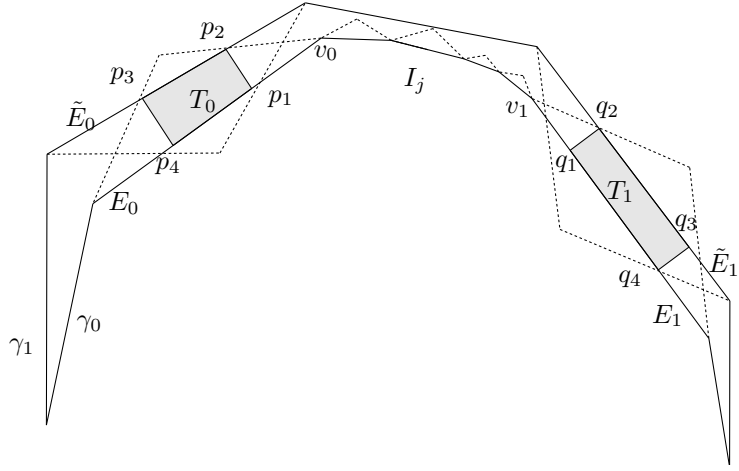

Let $\tilde E_0 \subset \gamma_1$ and $\tilde E_1 \subset \gamma_1$
be edges of $\gamma_1$ such that
$\tri_{\tilde E_0}\not\subset U$, $\tri_{\tilde E_1}\not\subset U$,
and $\tri_{\tilde E_0}\cap \tri_{E_0} \ne \emptyset$,
$\tri_{\tilde E_1}\cap \tri_{E_1} \ne \emptyset$.
Assume that there exist isosceles trapezoids
$T_0\subset \tri_{E_0} \cap \tri_{\tilde E_0}$,
$T_1\subset \tri_{E_1} \cap \tri_{\tilde E_1}$,
and $T_0$, $T_1$ are maximal.
Let $p_1, p_2, p_3, p_4$ be the vertices of $T_0$ and
$q_1, q_2, q_3, q_4$ be the vertices of $T_1$
as shown in Figure 5.

Let $\eta_0$ be the part of $\gamma_0$ enclosed by the points $p_1$ and $q_1$,
and let $\eta_1$ be the part of $\gamma_1$ between the points $p_2$ and $q_2$.

Consider now the polygonal set $A \subset U$ bounded by
$\eta_0$, $\eta_1$ and the segments $[p_1, p_2]$ and $[q_1, q_2]$.
We next show that
\begin{equation}\label{dA-A}
d(A)^2 \le c|A|
\quad\hbox{for some constant $c>0$.}
\end{equation}

Indeed,
since $Q_0$, $Q_1$ are convex sets with uniformly bounded eccentricities it is easy to see
that $\ell(\eta_1) \le c\ell(\eta_0)$.
Consider the case when $E_0$ and $E_1$ are connected by segments $I_1, \dots, I_m$.
Denote $I_0:= [p_1, v_0]$ and $\tri_{I_0}:= [p_1, v_0, p_2]$ the triangle with vertices $p_1, v_0, p_2$.
Also, denote $I_{m+1}:= [v_1, q_1]$ and set $\tri_{I_{m+1}}:= [v_1, q_1, q_2]$.
Now, let $j_{\max}:= \arg\max_{0\le j\le m+1}|\tri_{I_j}|$.
Then
$$
d(A)^2 \le c\max\{\ell(\eta_0), \ell(\eta_1)\}^2
\le c\ell(\eta_0)^2
\le c|I_{j_{\max}}|^2
\le c|\tri_{j_{\max}}| \le c|A|
$$
as claimed.
Here $\tri_{j_{\max}}$ is the triangle whose existence is claimed in Lemma~\ref{lem:two-tri}.

Just as above we establish estimate (\ref{dA-A}) for a set $A$ as above where the roles of
$\gamma_0$ and $\gamma_1$ are interchanged.

\smallskip

{\em Case 2.} Let $U$ be the closure of the $Q_0\cap Q_1$, where $Q_0,Q_1$ are convex polygonal sets as in Figure 6.
\begin{figure}[H]
\begin{center}
\begin{tikzpicture}[scale=0.45]
\clip(0.,7.5) rectangle (13,15.);
\fill[fill=black,fill opacity=0.13] (6.066849286368884,12.921681069076524) -- (6.368819823139635,12.468725263920396) -- (6.415470598933781,10.975900438507727) -- (6.210254884199399,9.418016885917249) -- (5.797044849582665,10.515160770934093) -- (6.001778101213133,12.794659901987165) -- cycle;
\draw (2.89052206733474,13.731394997192027)-- (1.7946200346838819,12.56402109284655);
\draw (1.7946200346838819,12.56402109284655)-- (1.5802044196000182,10.896344086638726);
\draw (1.5802044196000182,10.896344086638726)-- (2.104331478693907,9.32396290935706);
\draw (2.104331478693907,9.32396290935706)-- (3.009641853492442,8.347180662863906);
\draw (3.009641853492442,8.347180662863906)-- (5.082326132636456,8.299532748400825);
\draw (5.082326132636456,8.299532748400825)-- (6.1782281652873134,9.181019165967818);
\draw (6.1782281652873134,9.181019165967818)-- (6.416467737602717,10.943992001101805);
\draw (6.416467737602717,10.943992001101805)-- (6.368819823139636,12.468725263920389);
\draw (6.368819823139636,12.468725263920389)-- (5.606453191730345,13.612275211034326);
\draw (5.606453191730345,13.612275211034326)-- (2.89052206733474,13.731394997192027);
\draw (7.607665599179738,14.255522056285916)-- (6.011460464666531,12.873732536856574);
\draw (6.011460464666531,12.873732536856574)-- (5.868516721277289,11.706358632511098);
\draw (5.868516721277289,11.706358632511098)-- (5.797044849582668,10.515160770934079);
\draw (5.797044849582668,10.515160770934079)-- (6.487939609297338,8.680716064105471);
\draw (6.487939609297338,8.680716064105471)-- (7.583841641948197,8.180412962243125);
\draw (7.583841641948197,8.180412962243125)-- (9.108574904766781,8.180412962243125);
\draw (9.108574904766781,8.180412962243125)-- (10.77625191097461,8.895131679189335);
\draw (10.77625191097461,8.895131679189335)-- (11.371850841763118,11.30135135957491);
\draw (11.371850841763118,11.30135135957491)-- (10.919195654363852,13.25491585256122);
\draw (10.919195654363852,13.25491585256122)-- (9.251518648156024,13.993458526738971);
\draw (9.251518648156024,13.993458526738971)-- (7.607665599179738,14.255522056285916);
\draw (2.89052206733474,11.015463872796428)-- (3.3431772547340084,11.039287830027966);
\draw (3.3431772547340084,11.039287830027966)-- (3.4603974704633673,11.477150503296652);
\draw (3.4603974704633673,11.477150503296652)-- (3.0801883605534375,11.723940560550052);
\draw (3.0801883605534375,11.723940560550052)-- (2.7279859920673974,11.438602530749499);
\draw (2.7279859920673974,11.438602530749499)-- (2.89052206733474,11.015463872796428);
\draw (9.299166562619105,11.515766974658774)-- (9.823293621712994,11.515766974658774);
\draw (9.823293621712994,11.515766974658774)-- (10.085357151259938,11.96967432264491);
\draw (10.085357151259938,11.96967432264491)-- (9.823293621712994,12.423581670631044);
\draw (9.823293621712994,12.423581670631044)-- (9.299166562619105,12.423581670631044);
\draw (9.299166562619105,12.423581670631044)-- (9.03710303307216,11.96967432264491);
\draw (9.03710303307216,11.96967432264491)-- (9.299166562619105,11.515766974658774);
\draw (6.066849286368884,12.921681069076524)-- (6.368819823139635,12.468725263920396);
\draw (6.368819823139635,12.468725263920396)-- (6.415470598933781,10.975900438507727);
\draw (6.415470598933781,10.975900438507727)-- (6.210254884199399,9.418016885917249);
\draw (6.210254884199399,9.418016885917249)-- (5.797044849582665,10.515160770934093);
\draw (5.797044849582665,10.515160770934093)-- (6.001778101213133,12.794659901987165);
\draw (6.001778101213133,12.794659901987165)-- (6.066849286368884,12.921681069076524);
\draw (8.4,11.6) node[anchor=north west] {$\tilde{Q}_0$};
\draw (11.3,12.5) node[anchor=north west] {$Q_0$};
\draw (0.5,12.5) node[anchor=north west] {$Q_1$};
\draw (2.5,11) node[anchor=north west] {$\tilde{Q}_1$};
\end{tikzpicture}
\caption{One configuration for $U=Q_0\cap Q_1$.}
\end{center}
\end{figure}
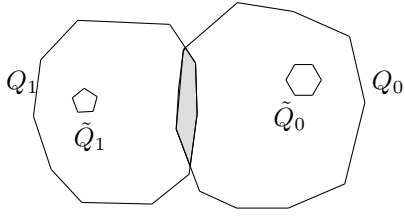

Then $\partial U$ consists of two polygonal curves $\gamma_0$ and $\gamma_1$ with
two points of intersection, each having no more than $N_0$ segments.
The argument is now simpler than the one in Case 1.

 \smallskip

{\em Case 3.}
It may also happen that we have a situation just as in Case 1,
where in addition the set $\tilde Q_1$ intersects $Q_1\setminus Q_0$ (see Figure 3)
or
the situation is as in Case~2, where $\tilde Q_0$ or $\tilde Q_1$
or both $\tilde Q_0$ and $\tilde Q_1$
intersect $Q_0\cap Q_1$ (see Figure 6).
We only consider in detail the first scenario, the second one is similar.

With the notation from Case 1, let $Q_1\setminus Q_0\ne \emptyset$
and assume that $\tilde Q_1$ intersects $Q_1\setminus Q_0$.
Let $U$ be the closure of a connected subset of $Q_1\setminus (Q_0\cup \tilde Q_1)$.
Then $U$ is subdivided by applying the procedure described above.

Several subcases are to be considered here.

 \smallskip

{\em Case 3 (a).} If $\tilde Q_1$ and the good triangles attached to $\tilde Q_1$
are contained in some set $A \in \cA_U$ from Case 1, then apparently
$|A| \le c |A\setminus\tilde Q_1|$ and hence
$$d(A)^2 \le c|A| \le c|A\setminus\tilde Q_1|.$$

 \smallskip

\begin{figure}[H]
\begin{center}
\begin{tikzpicture}[scale=7.5]
\clip(1.8,4) rectangle (3.2,4.6);
\fill[fill=black,fill opacity=0.1] (2.2053868980110396,4.242215300711261) -- (2.0875301039546255,4.361955859924884) -- (1.8925162770447033,4.1442219755108916) -- (1.9861195699217138,4.050357702207476) -- cycle;
\fill[fill=black,fill opacity=0.1] (2.8623983857331425,4.234372012694204) -- (2.9755779994259393,4.368192334770638) -- (3.174920920998595,4.17651644864308) -- (3.0857149558536947,4.070687147119672) -- cycle;

\fill[fill=black,fill opacity=0.1] (2.634646343811656,4.401307956346387) -- (2.6394972346532226,4.409328336919531) -- (2.7234294897318367,4.34650185358583) -- (2.71761429479809,4.340494744425402) -- cycle;
\draw (1.4119982405834608,2.845927033212807)-- (1.6575001483528458,3.881825326971437);
\draw (1.6575001483528458,3.881825326971437)-- (2.274248843480813,4.570428239007521);
\draw (2.274248843480813,4.570428239007521)-- (2.7652526590195827,4.570428239007521);
\draw (2.7652526590195827,4.570428239007521)-- (3.387989205556559,3.9716430981065782);
\draw (3.387989205556559,3.9716430981065782)-- (3.6275032619169343,2.7680849648956847);
\draw (3.6275032619169343,2.7680849648956847)-- (3.2861957316033994,3.9237402868345033);
\draw (3.2861957316033994,3.9237402868345033)-- (2.5,4.5);
\draw (2.5,4.5)-- (1.8071964335780806,3.893801029789456);
\draw (1.8071964335780806,3.893801029789456)-- (1.4119982405834608,2.845927033212807);
\draw [dash pattern=on 1pt off 1pt] (2.1646570342312277,4.048086770412238)-- (2.274248843480813,4.570428239007521);
\draw [dash pattern=on 1pt off 1pt] (1.6575001483528458,3.881825326971437)-- (2.1646570342312277,4.048086770412238);
\draw [dash pattern=on 1pt off 1pt] (1.978603647472274,4.396895677679342)-- (1.8071964335780806,3.893801029789456);
\draw [dash pattern=on 1pt off 1pt] (1.978603647472274,4.396895677679342)-- (2.5,4.5);
\draw [dash pattern=on 1pt off 1pt] (2.9037665511451185,4.091267112168379)-- (3.387989205556559,3.9716430981065782);
\draw [dash pattern=on 1pt off 1pt] (2.9037665511451185,4.091267112168379)-- (2.7652526590195827,4.570428239007521);
\draw [dash pattern=on 1pt off 1pt] (3.0594497160613168,4.438825302055729)-- (2.5,4.5);
\draw [dash pattern=on 1pt off 1pt] (3.0594497160613168,4.438825302055729)-- (3.2861957316033994,3.9237402868345033);
\draw (2.2053868980110396,4.242215300711261)-- (2.0875301039546255,4.361955859924884);
\draw (2.0875301039546255,4.361955859924884)-- (1.8925162770447033,4.1442219755108916);
\draw (1.8925162770447033,4.1442219755108916)-- (1.9861195699217138,4.050357702207476);
\draw (1.9861195699217138,4.050357702207476)-- (2.2053868980110396,4.242215300711261);
\draw (2.8623983857331425,4.234372012694204)-- (2.9755779994259393,4.368192334770638);
\draw (3.174920920998595,4.17651644864308)-- (3.0857149558536947,4.070687147119672);
\draw [dash pattern=on 1pt off 1pt] (2.5197507512501978,4.428687646470299)-- (2.7652526590195827,4.570428239007521);
\draw [dash pattern=on 1pt off 1pt] (2.5197507512501978,4.428687646470299)-- (2.274248843480813,4.570428239007521);
\draw (2.7652526590195827,4.570428239007521)-- (2.6499157271122504,4.483607017646278);
\draw (2.6499157271122504,4.483607017646278)-- (2.6273020577281545,4.418456890407301);
\draw (2.6273020577281545,4.418456890407301)-- (2.7340846904567173,4.338526030821008);
\draw (2.7340846904567173,4.338526030821008)-- (2.780040515320631,4.4312590849996845);
\draw (2.780040515320631,4.4312590849996845)-- (2.7652526590195827,4.570428239007521);
\draw [dash pattern=on 1pt off 1pt] (2.6273020577281545,4.418456890407301)-- (2.657619322443086,4.347665969738851);
\draw [dash pattern=on 1pt off 1pt] (2.657619322443086,4.347665969738851)-- (2.7340846904567173,4.338526030821008);
\draw (2.71761429479809,4.340494744425402)-- (2.722244027462879,4.347389217155577);
\draw (2.634646343811656,4.401307956346387)-- (2.639722624828114,4.409159623806161);

\draw (2.55,4.36) node[anchor=north west] {$T'$};
\draw (2.11,4.182934731486053) node[anchor=north west] {$T_1$};
\draw (2.9359677571572225,4.15) node[anchor=north west] {$T_2$};
\draw (2.65,4.47) node[anchor=north west] {$\tilde{Q}_1$};
\draw (2.75,4.25) node[anchor=north west] {$\gamma_0$};
\draw [->] (2.613463300052983,4.316225429964607) -- (2.67,4.37);
\end{tikzpicture}
\end{center}
\caption{The case when $\tilde{Q}_1\subset U$ and $\tilde{Q}_1$ is close to $\gamma_0$.}
\end{figure}
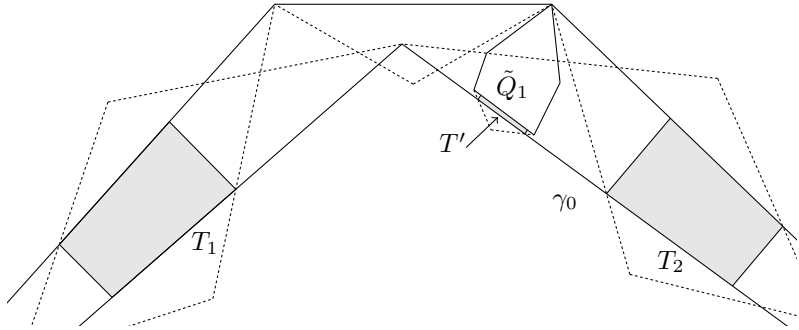
{\em Case 3 (b).} The most dangerous situation is when $\tilde Q_1$ is contained in $U$
and an edge of $\tilde Q_1$ is located close to the inner part $\gamma_0$ of $\partial U$
as shown in Figure 7.
However, in this situation a good triangle attached to $\tilde Q_1$ would intersect $\gamma_0$
(see Figure 7) and would create a trapezoid in $\cT_U$.

The set $\tilde Q_1$ may intersect $Q_1\setminus Q_0$ in various other ways.
The point is that after subtracting from $U$ the trapezoids $T\in \cT_U$ constructed above
the remaining connected components $A\in\cA_U$ cannot be uncontrollably elongated.
We omit the further details.

Also, an important point is that by construction $\tilde Q_1$ cannot intersect any trapezoid from $\cT_U$.
\end{proof}

In what follows we shall need the following obvious property of the trapezoids from $\cT$.

\begin{property}\label{property-tr}
There exists a constant $0< \ch <1$
such that if $L=[v_1,v_2]$ is one of the legs of a trapezoid $T\in\cT$
and $T\subset \tri_{E_1}\cap \tri_{E_2}$ (see the construction of trapezoids),
then for any $x\in L$ with $|x-v_j| \ge \rho$, $j=1, 2$, for some $\rho>0$
we have $B(x, \ch \rho) \subset \tri_{E_1}\cup \tri_{E_2}$.
Moreover, if $D=[v_1,v_2]$ is one of the bases of the trapezoid $T$,
then for any $x\in D$ with $|x-v_j| \ge \rho$, $j=1, 2$, for some $\rho>0$
we have $B(x, \ch \rho) \subset \tri_{E_1}\cap \tri_{E_2}$.
\end{property}

Let $\cA:= \cup_{U\in\cU} \cA_U$ and $\cT:= \cup_{U\in\cU} \cT_U$.
We have $\Omega=\cup_{A\in \cA} A \cup_{T\in\cT}T$
and, clearly, the sets in $\cA\cup \cT$ have disjoint interiors.
From these we obtain the following representation of $S_1(x)-S_2(x)$
for $x\in \Omega$ which is not on any of the edges:
\begin{equation}\label{rep-S1-S2}
S_1(x)-S_2(x) = \sum_{A\in \cA} c_A\ONE_A(x) + \sum_{T\in \cT} c_T\ONE_T(x),
\end{equation}
where $c_A$ and $c_T$ are constants.

For future reference, we note that
\begin{equation}\label{number-sets}
\# \cA \le cn \quad\hbox{and}\quad \#\cT \le cn.
\end{equation}
These estimates follow readily by (\ref{num-intesect}) and
the fact that the number of edges of each $U\in\cU$ is $\le 2N_0$.

Let $0<s/2<1/p$ and assume $\tau \le 1$.
Fix $t>0$ and
let $h\in \R^2$ with norm $|h|\le t$.
Write $\nn:=|h|^{-1}h$ and assume $\nn=:(\cos \theta, \sin\theta)$, $-\pi<\theta\le \pi$.

We shall frequently use the following obvious identities:
If $S$ is a constant on a~measurable set $G \subset \R^2$ and $H\subset G$ ($H$ measurable), then
\begin{equation}\label{triv-est-1}
\|S\|_{L^\tau(G)} = |G|^{1/\tau-1/p}\|S\|_{L^p(G)} = |G|^{s/2}\|S\|_{L^p(G)}
\end{equation}
and
\begin{equation}\label{triv-est-2}
\|S\|_{L^\tau(H)} = (|H|/|G|)^{1/\tau}\|S\|_{L^\tau(G)}.
\end{equation}

We next estimate
$\|\Delta_h S_1\|_{L^\tau(G)}^\tau - \|\Delta_h S_2\|_{L^\tau(G)}^\tau$
for different subsets $G$ of $\Omega$.

\subsection*{Case 1}

Let $T \in \cT$ be such that $d(T)> 2t/\ch$ with $\ch$ the constant from Property~\ref{property-tr}.
Denote
$$
T_h:=\{x\in\Omega: [x, x+h] \subset \Omega\;\;\hbox{and}\;\;[x, x+h]\cap T\ne \emptyset\}.
$$
We now estimate
$\|\Delta_h S_1\|_{L^\tau(T_h)}^\tau - \|\Delta_h S_2\|_{L^\tau(T_h)}^\tau$.

We may assume that $T$ is an isosceles trapezoid contained in $\tri_{E_1}\cap\tri_{E_2}$,
where $\tri_{E_j}$ ($j=1, 2$) is a good triangle for a ring $R_j\in\cR_j$,
and
$T$ is positioned so that its vertices are the points:
$$
v_1:=(-\delta_1/2, 0),\;\; v_2:=(\delta_1/2, 0),\;\; v_3:=(\delta_2/2, H),\;\; v_4:=(-\delta_2/2, H),
$$
where $0\le\delta_2\le\delta_1$ and $H >\delta_1$.
Let $L_1:=[v_1, v_4]$ and $L_2:=[v_2, v_3]$ be the two equal (long) legs of $T$.
We assume that $L_1\subset E_1$ and $L_2\subset E_2$.
We denote by $D_1:=[v_1, v_2]$ and $D_2:=[v_3, v_4]$ the two bases of $T$.
Set $\cV_T:=\{v_1, v_2, v_3, v_4\}$.
See Figure 8 below.

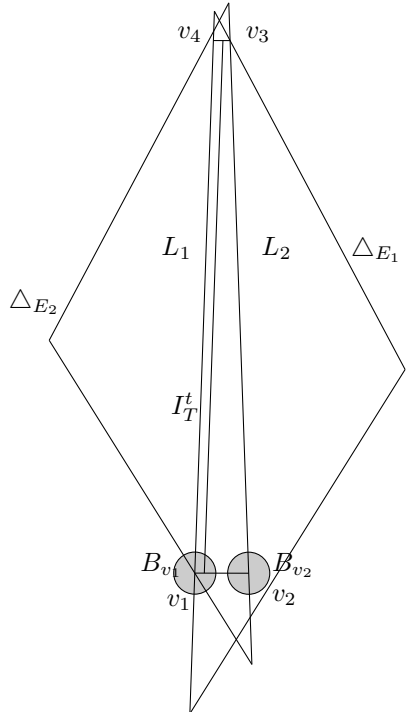
\begin{figure}[H]
\begin{center}
\begin{tikzpicture}[scale=0.7]
\clip(3,1.8) rectangle (11,16.);
\draw [fill=black,fill opacity=0.2] (7.732097903256681,4.435386181174541) circle (0.4cm);
\draw [fill=black,fill opacity=0.2] (6.707978841433643,4.43758386156472) circle (0.4cm);
\draw  (7.353526672640425,15.27624782672674)-- (3.943142175290054,8.86224744555305);
\draw  (3.943142175290054,8.86224744555305)-- (7.792637196594354,2.7017676435881626);
\draw (6.614670041512313,1.7655694334926233)-- (10.702555632195383,8.307553895032116);
\draw (7.080973692680657,15.118758895097757)-- (10.702555632195383,8.307553895032116);
\draw (7.080973692680657,15.118758895097757)-- (6.614670041512313,1.7655694334926233);
\draw (7.353526672640425,15.27624782672674)-- (7.792637196594354,2.7017676435881626);
\draw (7.0614523020116975,14.559739407466656)-- (7.378571462033556,14.559058894247723);
\draw (6.707978841433643,4.43758386156472)-- (7.732097903256681,4.435386181174541);
\draw (6.889174255170433,4.437195030204769)-- (7.242647715748484,14.559350576106675);
\draw (6.082932336361152,8.03746921885388) node[anchor=north west] {$I_T^t$};
\draw (6,4.2) node[anchor=north west] {$v_1$};
\draw (8,4.3) node[anchor=north west] {$v_2$};
\draw (7.5,15) node[anchor=north west] {$v_3$};
\draw (6.2,15) node[anchor=north west] {$v_4$};
\draw (5.5,5) node[anchor=north west] {$B_{v_1}$};
\draw (8,5) node[anchor=north west] {$B_{v_2}$};
\draw (5.9,11) node[anchor=north west] {$L_1$};
\draw (7.8,11) node[anchor=north west] {$L_2$};
\draw (3.,10) node[anchor=north west] {$\tri_{E_2}$};
\draw (9.5,11) node[anchor=north west] {$\tri_{E_1}$};

\end{tikzpicture}
\caption{A trapezoid $T$.}
\end{center}
\end{figure}

Furthermore, let $\gamma \le \pi/2$ be the angle between $D_1$ and $L_1$
and assume that  $\nn=:(\cos \theta, \sin\theta)$ with $\theta \in [\gamma, \pi]$.
The case $\theta \in [-\gamma, 0]$ is just the same.
The case when $\theta \in [0, \gamma]\cup [-\pi, -\gamma]$
is considered similarly.

Denote $B_v:=B(v, 2t/\ch)$, $v\in \cV_T$,
\begin{align*}
\cA_T^t &:=\big\{A\in\cA: d(A) > t \quad\hbox{and}\quad A\cap(T+B(0,t))\ne \emptyset \big\},\\
\fA_T^t &:=\big\{A\in\cA: d(A) \le t \quad\hbox{and}\quad A\cap(T+B(0,t))\ne \emptyset \big\}
\end{align*}
and
\begin{align*}
\cT_T^t:=\big\{T'\in\cT: d(T') > 2t/\ch\quad\hbox{and}\quad T'\cap(T+B(0,t))\ne \emptyset \big\},\\
\fT_T^t:=\big\{T'\in\cT: d(T') \le 2t/\ch\quad\hbox{and}\quad T'\cap(T+B(0,t))\ne \emptyset \big\}.
\end{align*}

\smallskip

{\em Case 1 (a).}
If $[x, x+h]\in \tri_{E_1}^\circ$, then $\Delta_h S_1(x)=0$
because $S_1$ is a constant on $\tri_{E_1}$.
Hence no estimate is needed.

\smallskip

{\em Case 1 (b).}
If $[x, x+h] \subset \cup_{v\in\cV_{T}} B_v$,
we estimate $|\Delta_h S_1(x)|$ using the obvious inequality
\begin{equation}\label{est-Del-S1-triv}
|\Delta_h S_1(x)| \le |\Delta_h S_2(x)| + |S_1(x)-S_2(x)| + |S_1(x+h)-S_2(x+h)|.
\end{equation}

Clearly, the contribution of this case to estimating
$\|\Delta_h S_1\|_{L^\tau(T_h)}^\tau - \|\Delta_h S_2\|_{L^\tau(T_h)}^\tau$~is
\begin{align*}
&\le c \sum_{v\in \cV_T}\sum_{A\in \cA_T^t}\|S_1-S_2\|_{L^\tau(B_v\cap A)}^\tau
+ c\sum_{v\in\cV_T}\sum_{T'\in\cT_T^t}\|S_1-S_2\|_{L^\tau(B_v\cap T')}^\tau\\
& + c \sum_{v\in \cV_T}\sum_{A\in \fA_T^t}\|S_1-S_2\|_{L^\tau(B_v\cap A)}^\tau
+ c\sum_{v\in\cV_T}\sum_{T'\in\fT_T^t}\|S_1-S_2\|_{L^\tau(B_v\cap T')}^\tau\\
&\le \sum_{A\in \cA_T^t}ct^2d(A)^{\tau s-2}\|S_1-S_2\|_{L^p(A)}^\tau
 + \sum_{T'\in\cT_T^t} ct^{1+\tau s/2}d(T')^{\tau s/2-1}\|S_1-S_2\|_{L^p(T')}^\tau\\
& + \sum_{A\in \fA_T^t} cd(A)^{\tau s}\|S_1-S_2\|_{L^p(A)}^\tau
+ \sum_{T'\in\fT_T^t} cd(T')^{\tau s}\|S_1-S_2\|_{L^p(T')}^\tau.
\end{align*}
Here we used these estimates, obtained using Lemma~\ref{lem:two-tri} and
(\ref{triv-est-1}) or/and (\ref{triv-est-2}):
(1) If $A\in \cA_T^t$ and $v\in \cV_T$, then
\begin{align*}
\|S_1-S_2\|_{L^\tau(B_v\cap A)}^\tau
&=(|B_v|/|A|)\|S_1-S_2\|_{L^\tau(A)}^\tau\\
&\le ct^2d(A)^{-2}\|S_1-S_2\|_{L^\tau(A)}^\tau
\le ct^2d(A)^{\tau s-2}\|S_1-S_2\|_{L^p(A)}^\tau.
\end{align*}
(2) If $T'\in\cT_T^t$ and $\delta_1(T') > 2t/\ch$
with $\delta_1(T')$ being the maximal base of $T'$,
then for any $v\in\cV_{T}$ we have
\begin{align*}
\|S_1-S_2\|_{L^\tau(B_v\cap T')}^\tau
&=(|B_v|/|T'|)\|S_1-S_2\|_{L^\tau(T')}^\tau
\le ct^2|T'|^{\tau s/2-1}\|S_1-S_2\|_{L^p(T')}^\tau\\
&\le ct^2\delta_1(T')^{\tau s/2-1}d(T')^{\tau s/2-1}\|S_1-S_2\|_{L^p(T')}^\tau\\
&\le ct^{1+\tau s/2}d(T')^{\tau s/2-1}\|S_1-S_2\|_{L^p(T')}^\tau,
\end{align*}
where we used that $\tau s/2 <1$, which is equivalent to $s<s+2/p$.

\noindent
(3) If $T'\in\cT_T^t$ and $\delta_1(T') \le 2t/\ch$,
then for any $v\in\cV_{T}$
\begin{align*}
\|S_1-S_2\|_{L^\tau(B_v\cap T')}^\tau
&=(|B_v\cap T'|/|T'|)\|S_1-S_2\|_{L^\tau(T')}^\tau\\
&= |B_v\cap T'||T'|^{\tau s/2-1}\|S_1-S_2\|_{L^p(T')}^\tau\\
&\le ct\delta_1(T')[\delta_1(T')d(T')]^{\tau s/2-1}\|S_1-S_2\|_{L^p(T')}^\tau\\
&= ct\delta_1(T')^{\tau s/2}d(T')^{\tau s/2-1}\|S_1-S_2\|_{L^p(T')}^\tau\\
&\le ct^{1+\tau s/2}d(T')^{\tau s/2-1}\|S_1-S_2\|_{L^p(T')}^\tau.
\end{align*}
(4) If $A\in \fA_T^t$, then
\begin{align*}
\|S_1-S_2\|_{L^\tau(B_v\cap A)}^\tau
&\le \|S_1-S_2\|_{L^\tau(A)}^\tau
\le c|A|^{\tau s/2}\|S_1-S_2\|_{L^p(A)}^\tau\\
&\le cd(A)^{\tau s}\|S_1-S_2\|_{L^p(A)}^\tau.
\end{align*}
(5) If $T'\in \fT_T^t$, then
\begin{align}\label{est-5}
\|S_1-S_2\|_{L^\tau(B_v\cap T')}^\tau
&\le \|S_1-S_2\|_{L^\tau(T')}^\tau
\le c|T'|^{\tau s/2}\|S_1-S_2\|_{L^p(T')}^\tau\\
&\le cd(T')^{\tau s}\|S_1-S_2\|_{L^p(T')}^\tau.\notag
\end{align}

{\em Case 1 (c).}
If $[x, x+h]\not\subset \cup_{v\in\cV_{T}} B_v$
and $[x, x+h]$ intersects $D_1$ or $D_2$, then
$\delta_1 > 2t/\ch >2t$ or $\delta_2 >2t$ and hence $[x, x+h]\subset \tri_{E_1}\cap\tri_{E_2}$,
which implies $\Delta_h S_1(x)=0$.
No estimate is needed.

\smallskip

{\em Case 1 (d).}
Let $\cITt$ be the set defined by
\begin{equation}\label{def-ITt}
\cITt:= \{x\in T: x \;\; \hbox{is between $L_1$ and $L_1+\eps e_1$}\}
\setminus \big(B(v_1, t/\ch)\cup B(v_4, t/\ch)\big),
\end{equation}
where
$\eps:= (\delta_1-\delta_2)M^{-1}t$,
$e_1:= \langle 1, 0 \rangle$, and $M:= |L_1|=|L_2|$.
Set $\ITt:=\cITt+ [0, h]$.
See Figure 8.

In this case we use again (\ref{est-Del-S1-triv}) to estimate $|\Delta_h S_1(x)|$.
We obtain
\begin{align*}
\|\Delta_h S_1\|_{L^\tau(\cITt)}^\tau
&\le \|\Delta_h S_2\|_{L^\tau(\cITt)}^\tau
+ \|S_1-S_2\|_{L^\tau(\cITt)}^\tau
\\
&+ \sum_{A\in \cA_T^t}\|S_1-S_2\|_{L^\tau(\ITt\cap A)}^\tau
+ \sum_{A\in \fA_T^t}\|S_1-S_2\|_{L^\tau(\ITt\cap A)}^\tau.
\end{align*}
Clearly, $|\cITt| \le ct\delta_1(T)$ and $|T| \sim \delta_1(T)d(T)$.
Then using (\ref{triv-est-1})-(\ref{triv-est-2}) we infer
\begin{align*}
\|S_1-S_2\|_{L^\tau(\cITt)}^\tau
&=(|\cITt|/|T|)\|S_1-S_2\|_{L^\tau(T)}^\tau
\le ct d(T)^{-1}\|S_1-S_2\|_{L^\tau(T)}^\tau\\
&= ct d(T)^{-1}|T|^{\tau s/2}\|S_1-S_2\|_{L^p(T)}^\tau
\le ct d(T)^{\tau s-1}\|S_1-S_2\|_{L^p(T)}^\tau.
\end{align*}
Similarly, for $A\in \cA_T^t$ we use that $|\ITt\cap A| \le ctd(A)$ and $|A|\sim d(A)^2$
to obtain
\begin{align*}
\|S_1-S_2\|_{L^\tau(\ITt\cap A)}^\tau
&\le ct d(A)\|S_1-S_2\|_{L^\infty(A)}^\tau
= ct d(A)|A|^{-\tau/p}\|S_1-S_2\|_{L^p(A)}^\tau\\
&\le ct d(A)^{1-2\tau/p}\|S_1-S_2\|_{L^p(A)}^\tau
\le ct d(A)^{\tau s-1}\|S_1-S_2\|_{L^p(A)}^\tau.
\end{align*}
For $A\in \fA_T^t$, we have
\begin{align*}
\|S_1-S_2\|_{L^\tau(\ITt\cap A)}^\tau
&\le \|S_1-S_2\|_{L^\tau(A)}^\tau
= |A|^{\tau s/2}\|S_1-S_2\|_{L^p(A)}^\tau\\
&\le cd(A)^{\tau s}\|S_1-S_2\|_{L^p(A)}^\tau.
\end{align*}
Putting the above estimates together we get
\begin{align*}
\|\Delta_h S_1\|_{L^\tau(\cITt)}^\tau
\le \|\Delta_h S_2\|_{L^\tau(\cITt)}^\tau
&+ ct d(T)^{\tau s-1}\|S_1-S_2\|_{L^p(T)}^\tau\\
+ \sum_{A\in \cA_T^t} ct d(A)^{\tau s-1}\|S_1-S_2\|_{L^\infty(A)}^\tau
&+ \sum_{A\in \fA_T^t}cd(A)^{\tau s}\|S_1-S_2\|_{L^p(A)}^\tau.
\end{align*}

\smallskip

{\em Case 1 (e) (Main).}
Let $T_h^\star \subset T_h$ be defined by
\begin{equation}\label{def-T-st}
T_h^\star := \{x\in T_h: [x, x+h]\cap L_1 \ne \emptyset,\;
x\not\in \cITt,\; [x,x+ h] \not\subset \bigcup_{v\in\cV_{T}} B_v\}.
\end{equation}
We next estimate
$\|\Delta_h^k S_1\|_{L^\tau(T_h^\star)}^\tau$.

Recall that by assumption
$h=|h|\nu$ with $\nn=:(\cos \theta, \sin\theta)$ and $\theta \in [\gamma, \pi]$,
where $\gamma \le \pi/2$ is the angle between $D_1$ and $L_1$.

Let $x\in T_h^\star$.
With the notation $x=(x_1, x_2)$ we let $(-a, x_2)\in L_1$ and $(a, x_2)\in L_2$, $a>0$,
be the points of intersection of the horizontal line through $x$ with $L_1$ and $L_2$.
Set $b:=2a-\eps$ with $\eps:= (\delta_1-\delta_2)M^{-1}t$, see (\ref{def-ITt}).

We associate the points $x+be_1$ and $x+be_1+h$ to $x$ and $x+h$.
A~simple geometric argument shows that
$x+be_1\in \tri_{E_1}\setminus T$, while $x+be_1+h \in T^\circ$.

Now, using that
$S_1 = \constant$ on $\tri_{E_1}^\circ$
we have
$S_1(x)=S_1(x+be_1)$
and since $S_2 = \constant$ on $\tri_{E_2}^\circ$
we have
$S_2(x+h)=S_2(x+be_1+h)$.
We use these two identities to obtain
\begin{align*}
S_1(x+h)-S_1(x)&= S_2(x+be_1+h)-S_2(x+be_1)\\
&+[S_1(x+h)-S_2(x+h)] - [S_1(x+be_1)-S_2(x+be_1)]
\end{align*}
and, therefore,
\begin{align}\label{Delta-shift}
|\Delta_h S_1(x)|
&\le |\Delta_h S_2(x+be_1)|\\
&+|S_1(x+h)-S_2(x+h)| + |S_1(x+be_1)-S_2(x+be_1)|.\notag
\end{align}
Some words of explanation are in order here.
The purpose of the set $\cITt$ is that
there is one-to-one correspondence between pairs of points
$x\in T^\circ\setminus \cITt$, $x+h\in \tri_{E_2}\setminus T$
and $x+be_1 \in \tri_{E_1}\setminus T$, $x+be_1+h\in T^\circ$.
Due to $\delta_2<\delta_1$, this would not be true if $\cITt$ was not removed from $T^\circ$.
Thus there is one-to-one correspondence between the differences
$|\Delta_h S_1(x)|$ and $|\Delta_h S_2(x+be_1)|$ in the case under consideration.
Also, it is important that $\Delta_h S_1(x+be_1)=0$
and hence $|\Delta_h S_2(x+be_1)|$ need not be used to estimate $|\Delta_h S_1(x+be_1)|$.

Another important point here is that $x+h\not\in T^\circ$ and $x+be_1 \not\in T^\circ$.
Therefore, no quantities $|S_1(x)-S_2(x)|$ with $x\in T^\circ\setminus \cITt$
are involved in (\ref{Delta-shift}), which is critical.

Observe that for $x\in T_h^\star$ we have
$$
[x,x+h] \not\subset \bigcup_{v\in\cV_{T}} B_v
\quad\hbox{and hence}\quad
[x+be_1,x+be_1+h] \not\subset \bigcup_{v\in\cV_{T}} B_v.
$$
Therefore, by Property~\ref{property-tr} it follows that $[x, x+h]$ and $[x+be_1,x+be_1+h]$ do not intersect
any trapezoid $T'\in\cT$, $T'\ne T$.

Let $T_h^{\star\star} :=\{x+be_1: x\in T_h^\star\}$.
For any $A\in\cA$ and $t>0$ define
\begin{equation}\label{def-At}
A_t:=\{x\in A: \dist(x, \partial A) \le t\}.
\end{equation}

From all of the above we get
\begin{align*}
\|\Delta_h S_1\|_{L^\tau(T_h^\star)}^\tau
\le \|\Delta_h S_2\|_{L^\tau(T_h^{\star\star})}^\tau
+ \sum_{A\in \cA_T^t}\|S_1-S_2\|_{L^\tau(A_t)}^\tau
+ \sum_{A\in \fA_T^t}\|S_1-S_2\|_{L^\tau(A)}^\tau.
\end{align*}
Now, using that $|A_t| \le ctd(A)$ and $|A|\sim d(A)^2$
for $A\in \cA_T^t$ we obtain
\begin{align}\label{est-norm-At}
\|S_1-S_2\|_{L^\tau(A_t)}^\tau
&=(|A_t|/|A|)|A|^{\tau s/2}\|S_1-S_2\|_{L^p(A)}^\tau\\
&\le ct d(A)^{\tau s-1}\|S_1-S_2\|_{L^p(A)}^\tau.\notag
\end{align}
For $A\in \fA_T^t$ we use that $|A|\sim d(A)^2$ and obtain
\begin{align}\label{est-norm-A}
\|S_1-S_2\|_{L^\tau(A)}^\tau
=|A|^{\tau s/2}\|S_1-S_2\|_{L^p(A)}^\tau
\le cd(A)^{\tau s}\|S_1-S_2\|_{L^p(A)}^\tau.
\end{align}
Inserting these estimates above we get
\begin{align}\label{est-Th-star}
\|\Delta_h S_1\|_{L^\tau(T_h^\star)}^\tau
\le \|\Delta_h S_2\|_{L^\tau(T_h^{\star\star})}^\tau
&+ \sum_{A\in \cA_T^t} ct d(A)^{\tau s-1}\|S_1-S_2\|_{L^p(A)}^\tau\\
&+ \sum_{A\in \fA_T^t}cd(A)^{\tau s}\|S_1-S_2\|_{L^p(A)}^\tau.\notag
\end{align}

\subsection*{Case 2}

Let $\Omega_h^\star$ be the set of all $x\in \Omega$ such that
$[x, x+h] \subset \Omega$
and
$[x, x+h] \cap T = \emptyset$ for all $T\in\cT$ with $d(T) \ge 2t/\ch$.
To estimate $|\Delta_h S_1(x)|$ we use again (\ref{est-Del-S1-triv}).
With the notation from (\ref{def-At}) we get
\begin{align*}
\|\Delta_hS_1\|_{L^\tau(\Omega_h^\star)}^\tau
\le \|\Delta_h S_2\|_{L^\tau(\Omega_h^\star)}^\tau
&+ \sum_{T\in\cT: d(T) \le 2t/\ch} \|S_1-S_2\|_{L^\tau(T)}^\tau\\
+ \sum_{A\in\cA: d(A) >t} \|S_1-S_2\|_{L^\tau(A_t)}^\tau
&+ \sum_{A\in\cA: d(A) \le t} \|S_1-S_2\|_{L^\tau(A)}^\tau\\
\end{align*}
For the first sum above we have just as in (\ref{est-5})
\begin{align*}
\sum_{T\in\cT: d(T) \le 2t/\ch} \|S_1-S_2\|_{L^\tau(T)}^\tau
\le \sum_{T\in\cT: d(T) \le 2t/\ch}
c d(T)^{\tau s}\|S_1-S_2\|_{L^p(T)}^\tau.
\end{align*}
We estimate the other two sums as in (\ref{est-norm-At}) and (\ref{est-norm-A}).
We obtain
\begin{align*}
\|\Delta_hS_1\|_{L^\tau(\Omega_h^\star)}^\tau
\le \|\Delta_h S_2\|_{L^\tau(\Omega_h^\star)}^\tau
&+ \sum_{T\in\cT: d(T) \le 2t/\ch}c d(T)^{\tau s}\|S_1-S_2\|_{L^p(T)}^\tau\\
+ \sum_{A\in\cA: d(A) >t}ct d(A)^{\tau s-1}\|S_1-S_2\|_{L^p(A)}^\tau
&+ \sum_{A\in\cA: d(A) \le t}cd(A)^{\tau s}\|S_1-S_2\|_{L^p(A)}^\tau.
\end{align*}

It is an important observation that each trapezoid $T\in\cT$ with $d(T) > 2t/\ch$
may share trapezoids $T'\in \fT_T^t$ and sets $A\in \fA_T^t$ with only finitely many
trapezoids with the same properties.
Also, for every such trapezoid $T$ we have $\# \cT_T^t \le c$ and $\# \cA_T^t \le c$
with $c>0$ a constant depending only on the structural constants of the setting.
Therefore, in the above estimates only finitely many norms may overlap at a time.
Putting all of them together we obtain
\begin{align*}
\omega_1(S_1, t)_\tau^\tau
\le \omega_1(S_2, t)_\tau^\tau + Y_1 + Y_2,
\end{align*}
where
\begin{align*}
Y_1
&= \sum_{A\in\cA: d(A) >t}ct d(A)^{\tau s-1}\|S_1-S_2\|_{L^p(A)}^\tau\\
&+ \sum_{A\in\cA: d(A) >t}ct^2d(A)^{\tau s-2}\|S_1-S_2\|_{L^p(A)}^\tau\\
&+ \sum_{A\in\cA: d(A) \le t}cd(A)^{\tau s}\|S_1-S_2\|_{L^p(A)}^\tau
\end{align*}
and
\begin{align*}
Y_2
&= \sum_{T\in\cT: d(T) > 2t/\ch}ct d(T)^{\tau s-1}\|S_1-S_2\|_{L^p(T)}^\tau\\
&+ \sum_{T\in\cT: d(T) > 2t/\ch}ct^{1+\tau s/2}d(T)^{\tau s/2-1}\|S_1-S_2\|_{L^p(T)}^\tau\\
&+ \sum_{T\in\cT: d(T) \le 2t/\ch}c d(T)^{\tau s}\|S_1-S_2\|_{L^p(T)}^\tau.
\end{align*}
%
%
We now turn to the estimation of $|S_1|_{B^s_\tau}$.
Using the above and interchanging the order of integration and summation we get
\begin{align*}
|S_1|_{B^s_\tau}^\tau
=\int_0^\infty t^{-s\tau-1}\omega_1(S_1, t)_\tau^\tau dt
\le |S_2|_{B^s_\tau}^\tau + Z_1 + Z_2,
\end{align*}
where
\begin{align*}
Z_1 &= \sum_{A\in\cA}cd(A)^{\tau s-1}\|S_1-S_2\|_{L^p(A)}^\tau \int_0^{d(A)}t^{-\tau s} dt\\
&+ \sum_{A\in\cA}cd(A)^{\tau s-2}\|S_1-S_2\|_{L^p(A)}^\tau \int_0^{d(A)}t^{-\tau s+1} dt\\
&+ \sum_{A\in\cA}cd(A)^{\tau s}\|S_1-S_2\|_{L^p(A)}^\tau \int_{d(A)}^\infty t^{-\tau s-1} dt
\end{align*}
and
\begin{align*}
Z_1 &=  \sum_{T\in\cT}cd(T)^{\tau s-1}\|S_1-S_2\|_{L^p(T)}^\tau \int_0^{\ch d(T)/2}t^{-\tau s} dt\\
&+ \sum_{T\in\cT}cd(T)^{\tau s/2-1}\|S_1-S_2\|_{L^p(T)}^\tau \int_0^{\ch d(T)/2}t^{-\tau s/2} dt\\
&+ \sum_{T\in\cT}cd(T)^{\tau s}\|S_1-S_2\|_{L^p(T)}^\tau \int_{\ch d(T)/2}^\infty t^{-\tau s-1} dt.
\end{align*}
%
Observe that $-\tau s>-1$ is equivalent to $s/2<1/p$, which is one of the assumptions,
and $-\tau s/2>-1$ is equivalent to $s<s+2/p$, which is obvious.
Therefore, all of the above integrals are convergent, and we obtain
\begin{align*}
|S_1|_{B^s_\tau}^\tau
\le |S_2|_{B^s_\tau}^\tau
+ \sum_{A\in\cA}c\|S_1-S_2\|_{L^p(A)}^\tau
+ \sum_{T\in\cT}c\|S_1-S_2\|_{L^p(T)}^\tau.
\end{align*}
Finally, applying H\"{o}lder's inequality and using (\ref{number-sets}) we arrive at
\begin{align*}
|S_1|_{B^s_\tau}^\tau
\le |S_2|_{B^s_\tau}^\tau
&+ c\big(\#\cA\big)^{\tau(1/\tau-1/p)}\Big(\sum_{A\in\cA} \|S_1-S_2\|_{L^p(A)}^p\Big)^{\tau/p}\\
&+ c\big(\#\cT\big)^{\tau(1/\tau-1/p)}\Big(\sum_{T\in\cT} \|S_1-S_2\|_{L^p(T)}^p\Big)^{\tau/p}\\
& \le cn^{\tau(1/\tau-1/p)}\|S_1-S_2\|_{L^p(\Omega)}^{\tau}
= cn^{\tau s/2}\|S_1-S_2\|_{L^p(\Omega)}^{\tau}.
\end{align*}
This confirms estimate (\ref{Bernstein-3}).
The proof in the case when $\tau >1$ is the same.
\qed

The proof of Theorem~\ref{thm:BV-Bernstein} is easier than the above proof. We omit it.

\section{Nonlinear approximation from smooth splines}\label{sec:linear-splines}

In this section we focus on Bernstein estimates in nonlinear approximation
in $L^p$, $0<p< \infty$, from regular nonnested smooth piecewise polynomial functions in~$\R^2$.

\subsection{Setting and approximation tool}\label{subsec:setting-2}

We first introduce the class of regular piecewise polynomials
$\cS(n, k)$ of degree $k-1$ with $k\ge 2$ over $n$ rings of maximum smoothness.
As in \S\ref{sec:piece-const} we introduce two versions of this class depending on
whether $\Omega$ is compact or $\Omega=\R^2$.

{\em Case 1:} $\Omega$ is a compact polygonal domain in $\R^2$.
We denote by $\cS(n, k)$ ($n\ge n_0$) the set of all piecewise polynomials $S$ of the form
\begin{equation}\label{def-smooth-splines-1}
S=\sum_{j=1}^n P_j\ONE_{R_j},
\quad S\in W^{k-2}(\Omega),
\quad P_j\in\Pi_k,
\end{equation}
where $R_1, \dots, R_n$ are rings in the sense of Definition~\ref{def-rings} with disjoint interiors
such that $\Omega=\cup_{j=1}^n R_j$.
Here $\Pi_k$ stands for the set of all polynomials of degree $<k$ in two variables
and $S\in W^{k-2}(\Omega)$ means that all partial derivatives
$\partial^\alpha S \in C(\Omega)$, $|\alpha| \le k-2$.

\smallskip

{\em Case 2:} $\Omega =\R^2$.
In this case we denote by $\cS(n, k)$ the set of all piecewise polynomials $S$ of degree $k-1$ on $\R^2$
of the form (\ref{def-splines-1}),
where $R_1, \dots, R_n$ are rings with disjoint interiors
such that the support $\Lambda=\cup_{j=1}^n R_j$ of $S$ is a ring in the sense of Definition~\ref{def-rings}.

\smallskip

We denote by $S_n^k(f)_p$ the best approximation of $f\in L^p(\Omega)$ from $\cS(n,k)$
in $L^p(\Omega)$, $0 < p<\infty$, i.e.
\begin{equation}\label{def-best-smooth-app}
S_n^{k}(f)_p:= \inf_{S\in \cS(n, k)}\|f-S\|_{L^p}.
\end{equation}

\begin{xrem}
Observe that in our setting the splines are of maximum smoothness and this is critical
for our development.
As will be shown in Example~\ref{example-maxsmooth} below in the nonnested case our Bernstein type inequality
is not valid in the case when the smoothness of the splines is not maximal.
\end{xrem}

We next consider several scenarios for constructing of regular piecewise polynomials
of maximum smoothness:

\smallskip

{\em Example 1.} Suppose that $\cT_0$ is an initial subdivision of $\Omega$ into triangles
which obey the minimum angle condition and is with no hanging vertices in the interior of $\Omega$.
In the case of $\Omega=\R^2$ we assume for simplicity that the triangles $\tri\in\cT_0$
are of similar areas, i.e. $c_1 \le |\tri_1|/|\tri_2| \le c_2$ for any $\tri_1, \tri_2\in \cT_0$.
Next we subdivide each triangle $\tri\in\cT_0$ into 4 triangles by introducing the mid points on the sides
of $\tri$. The result is a triangulation $\cT_1$ of $\Omega$.
In the same way we define the triangulations $\cT_2$, $\cT_3$, etc.
Each triangulation $\cT_j$ supports Courant hat functions (linear finite elements)
$\varphi_\theta$, each of them supported on the union $\theta$ of all triangles from $\cT_j$ which
have a common vertex, say, $v$.
Thus $\varphi_\theta(v)=1$, $\varphi_\theta$ takes values zero
at all other vertices of triangles from $\cT_j$, and $\varphi_\theta$ is continuous
and piecewise linear over the triangles from $\cT_j$.
Clearly, each piecewise liner function over the triangles from $\cT_j$
can be represented as a linear combination of Courant hat functions like these.

Denote by $\Theta_j$ the set of all supports $\theta$ of Courant elements supported by $\cT_j$
and set $\Theta:= \cup_{j\ge 0} \Theta_j$.
Consider the nonlinear set $\bS_n$ of all piecewise linear functions $S$ of the form
$$
S=\sum_{\theta\subset \cM_n} c_\theta \varphi_\theta,
$$
where $\cM_n\subset \Theta$ and $\# \cM_n \le n$; the elements $\theta\in\cM_n$
may come from different levels and locations.
It is not hard to see that $\bS_n \subset S(cn, 2)$, see \cite{KP1}.

\smallskip

{\em Example 2.} More generally, one can consider piecewise linear functions $S$ of the form
$$
S=\sum_{\theta\subset \cM_n} c_\theta \varphi_\theta,
$$
where $\{\varphi_\theta\}$ are Courant hat functions as above,
$\# \cM_n \le n$, and
$\cM_n$ consists of cells $\theta$ as above that are not necessarily induced by a hierarchical
collection of triangulations of $\Omega$, however, there exists a underlying subdivision of $\Omega$
into rings obeying the conditions from \S\ref{subsec:setting-1}.

\smallskip

{\em Example 3.} The $C^1$ quadratic box-splines on the four-directional mesh
(the so called ``Powell-Zwart finite elements")
and
the piecewise cubics in $\R^2$ or on a rectangular domain,
endowed with the Powell--Sabin triangulation generated by a uniform 6-direction mesh
provide examples of quadratic and cubic splines of maximum smoothness.

Other examples are to be identified or developed.

\subsection*{Splines with defect}

To make the difference between approximation from nonnested and nested splines more transparent
and for future references
we now introduce the splines with arbitrary smoothness.
Given a set $\Omega\subset \R^2$ with polygonal boundary or $\Omega:=\R^2$, $k\ge 2$, and $0\le r\le k-2$,
we denote by $\cS(n, k, r)$ ($n\ge n_0$) the set of all piecewise polynomials $S$ of the form
\begin{equation}\label{def-smooth-splines-2}
S=\sum_{j=1}^n P_j\ONE_{R_j},
\quad S\in W^{r}(\Omega),
\quad P_j\in\Pi_k,
\end{equation}
where $R_1, \dots, R_n$ are rings with disjoint interiors
such that $\Omega=\cup_{j=1}^n R_j$.
We set
\begin{equation}\label{def-best-smooth-app-2}
S_n^{k, r}(f)_p:= \inf_{S\in \cS(n, k, r)}\|f-S\|_{L^p}.
\end{equation}

\subsection{Jackson estimate}\label{subsec:smooth-Jackson-est}

Jackson estimates in 
spline approximation are relatively easy to prove.
Such estimates (also in anisotropic settings) are established in \cite{DP,KP1}.
For example the Jackson estimate we need in the case of approximation
from piecewise linear functions ($k=2$) follows by \cite[Theorem 3.6]{KP1}
and takes the form:

\begin{thm}\label{thm:smooth-Jackson}
Let $0<p<\infty$, $s>0$, and $1/\tau=s/2+1/p$.
Assume $\Omega=\R^2$ or $\Omega\subset \R^2$ is a compact set with polygonal boundary
and initial triangulation consisting of $n_0$ triangles with no hanging interior vertices
and obeying the minimum angle condition.
Then for any $f\in B^{s,2}_{\tau}$ we have $f\in L^p(\Omega)$ and
\begin{equation}\label{smooth-Jackson-1}
S_n^2(f)_p  \le cn^{-s/2}|f|_{B^{s,2}_{\tau}} , \quad n\ge n_0.
\end{equation}
Consequently, for any $f\in L^p(\Omega)$
\begin{equation}\label{smooth-direct-est}
S_n^2(f)_p  \le cK(f, n^{-s/2}) , \quad n\ge n_0.
\end{equation}
Here $K(f, t)=K(f, t; L^p, B^s_\tau)$ is the $K$-functional defined in $(\ref{def-K-functional})$
and $c>0$ is a constant depending only on $s,p$, and the structural constants of the setting.
\end{thm}

\smallskip

Similar Jackson and direct estimates for nonlinear approximation from splines of degrees 2 and higher
and of maximum smoothness do not follow automatically from the results in \cite{DP}.
The reason being the fact that the basis functions for splines of degree 2 and 3 that
we are familiar with are not stable. The stability is required in \cite{DP}.
The problem for establishing Jackson estimates for approximation from splines
of degree 2 and higher of maximum smoothness remains open.

\subsection{Bernstein estimate in the nonnested case}\label{subsec:smooth-Bern-est-2}

We come now to one of the main result of this article.
Here we operate in the setting described above in \S\ref{subsec:setting-2}.

\begin{thm}\label{thm:smooth-Bernstein-2}
Let $0< p<\infty$, $k\ge 1$, $0<s/2<k-1+1/p$, and $1/\tau=s/2+1/p$.
Then for any  $S_1, S_2\in \cS(n, k)$, $n\ge n_0$, we have
\begin{align}
|S_1|_{B^{s,k}_{\tau}} &\le |S_2|_{B^{s, k}_{\tau}} + cn^{s/2}\|S_1-S_2\|_{L^p},
\quad\hbox{if} \;\; \tau \ge 1, \quad\hbox{and}\label{smooth-Bernstein-2}\\
|S_1|_{B^{s, k}_{\tau}}^\tau &\le |S_2|_{B^{s, 2}_{\tau}}^\tau + cn^{\tau s/2}\|S_1-S_2\|_{L^p}^\tau,
\quad\hbox{if} \;\; \tau < 1. \label{smooth-Bernstein-3}
\end{align}
where the constant $c>0$ depends only on $s, p, k$, and the structural constants of the setting.
\end{thm}

An immediate consequence of this theorem is the inverse estimate given in

\begin{cor}\label{cor:inverese-2}
Let $0< p<\infty$, $k\ge 1$, $0<s/2<k-1+1/p$, and $1/\tau=s/2+1/p$. Set $\lambda :=\min\{\tau, 1\}$.
Then for any $f\in L^p(\Omega)$ we have
\begin{equation}\label{inverse-est-2}
K(f, n^{-s/2}) \le cn^{-s/2}
\Big(\sum_{\ell=n_0}^n \frac{1}{\ell}\big[\ell^{s/2} S_\ell^k(f)_p\big]^\lambda + \|f\|_p^\lambda\Big)^{1/\lambda},
\quad n\ge n_0.
\end{equation}
Here $K(f, t)=K(f, t; L^p, B^s_\tau)$ is the $K$-functional defined just as in $(\ref{def-K-functional})$
and $c>0$ is a constant depending only on $s, p, k$, and the structural constants of the setting.
\end{cor}

The proof of this corollary is just a repetition of the proof of Theorem~\ref{thm:inverese-1}.
We omit it.

In turn, estimates (\ref{smooth-direct-est}) and (\ref{inverse-est-2}) imply a characterization
of the approximation spaces associated with nonlinear nonnested piecewise linear approximation,
see (\ref{character-app}).

The proof of Theorem~\ref{thm:smooth-Bernstein-2} relies on the idea we used in the proof of
Theorem~\ref{thm:Bernstein}. However, there is an important complication to overcome.
The fact that many rings with relatively small supports can be located next to a large ring
is a major obstacle in implementing this idea in the case of smooth splines.
An additional construction is needed.
To make the proof more accessible, we shall proceed in two steps.
We first develop the needed additional construction and implement it in \S\ref{subsec:smooth-Bern-est}
to prove the respective Bernstein estimate in the nested case and then
we present the proof of Theorem~\ref{thm:smooth-Bernstein-2} in \S\ref{subsec:proof-smooth-Bern-est-2}.

Before we proceed with the proofs of the Bernstein estimates we show in the next example
that the assumption that in our setting the splines are of maximum smoothness is essential.

\begin{ex}\label{example-maxsmooth}
We now show that estimates (\ref{smooth-Bernstein-2})-(\ref{smooth-Bernstein-3}) 
fail without the assumption that $S_1,S_2\in W^{k-2}(\Omega)$
(i.e., both splines have maximum smoothness).
We shall only consider the case when $k=2$ and $\tau \le 1$. 
Let $\Omega=[-1,1]\times[0,1]$ and $0<\varepsilon<1/4$.
Set
$$
S_1(x):=x_1\ONE_{[0,1]^2}(x),\quad S_2(x):=x_1\ONE_{[\varepsilon,1]\times[0,1]}(x), \quad x=(x_1, x_2).
$$
Clearly, $S_1$ is continuous on $\Omega$, while $S_2$ is discontinuous along $x_1=\varepsilon$.
A straightforward calculation shows that
\begin{equation}\label{smallt}
\omega_2(S_1,t)_\tau^\tau=\frac{2t^{\tau+1}}{\tau+1}
\quad\hbox{and}\quad\omega_2(S_2,t)_\tau^\tau=\int_{-t}^t|w+\varepsilon|^\tau dw
\quad\hbox{for}\quad 0\le t\le 1/4.
\end{equation}
Further,
\begin{equation}\label{int-tt}
\int_{-t}^t|w+\varepsilon|^\tau dw=\frac{1}{\tau+1}\big[(t+\varepsilon)^{\tau+1}
+{\rm sign}(t-\varepsilon)|t-\varepsilon|^{\tau+1}\big].
\end{equation}
On the other hand, obviously
$\omega_2(S_1-S_2, t)_\tau^\tau \le 4\|S_1-S_2\|_{L^\tau}^\tau \le 4\varepsilon^{\tau+1}$
yielding
\begin{equation}\label{larget}
\omega_2(S_2,t)_\tau^\tau\ge\omega_2(S_1,t)_\tau^\tau-4\varepsilon^{\tau+1}.
\end{equation}
We shall use this estimate for $t>1/4$.
From (\ref{def-smooth-Besov}) and (\ref{smallt})-(\ref{larget}) we obtain
\begin{align*}
|S_2|_{B_\tau^{s,2}}^\tau &-|S_1|_{B_\tau^{s,2}}^\tau
\ge\frac{1}{\tau+1}\Big[\int_0^\varepsilon t^{-s\tau-1}[(t+\varepsilon)^{\tau+1}-(\varepsilon-t)^{\tau+1}-2t^{\tau+1}]dt\\
&+\int_\varepsilon^{1/4} t^{-s\tau-1}[(\varepsilon+t)^{\tau+1}+(t-\varepsilon)^{\tau+1}-2t^{\tau+1}]dt\Big]
-4\varepsilon^{\tau+1}\int_{1/4}^\infty t^{-s\tau-1}dt\\
&=: I_1+I_2-(4^{s\tau+1}/s\tau)\varepsilon^{\tau+1}.
\end{align*}
Substituting $t=\varepsilon u$ in $I_1$ and $I_2$, we get
$$
I_1+I_2=\frac{\varepsilon^{\tau-s\tau+1}}{\tau+1}
\Big[\int_0^1u^{-s\tau-1}\phi_1(u)du
+\int_1^{1/4\varepsilon}u^{-s\tau-1}\phi_2(u)du\Big],
$$
where
$$
\phi_1(u)=(1+u)^{\tau+1}-(1-u)^{\tau+1}-2u^{\tau+1}
$$
and
$$
\phi_2(u)=(1+u)^{\tau+1}+(u-1)^{\tau+1}-2u^{\tau+1}.
$$
We clearly have $\phi_1\ge0$ on $[0,1]$ and $\phi_2\ge 0$ on $[1,\infty)$.
Therefore,
$$
|S_2|_{B_\tau^{s,2}}^\tau-|S_1|_{B_\tau^{s,2}}^\tau\ge c_1\varepsilon^{\tau-s\tau+1}-c_0\varepsilon^{\tau+1}
=\varepsilon^{\tau-s\tau+1}(c_1-c_0\varepsilon^{s\tau}),
$$
where
$$
c_1:=\frac{1}{\tau+1}\int_0^1t^{-s\tau-1}\phi_1(u)du>0
\quad\hbox{and}\quad
c_0:= 4^{s\tau+1}/s\tau.
$$
%
By taking $\varepsilon$ sufficiently small, we get
\begin{equation}
\label{norm1}
|S_2|_{B_\tau^{s,2}}^\tau-|S_1|_{B_\tau^{s,2}}^\tau\ge (c_1/2)\varepsilon^{\tau-s\tau+1}.
\end{equation}
Evidently,
\begin{equation}
\label{norm2}
\|S_2-S_1\|_{L^p}\le \varepsilon^{1+1/p}.
\end{equation}
By (\ref{norm1}) and (\ref{norm2}),
$$
\frac{|S_2|_{B_\tau^{s,2}}^\tau-|S_1|_{B_\tau^{s,2}}^\tau}{\|S_2-S_1\|_{L^p}^\tau}
\ge (c_1/2)\varepsilon^{1-s\tau-\tau/p}=(c_1/2)\varepsilon^{-s\tau/2}.
$$
Since $\varepsilon^{-s\tau/2}\rightarrow\infty$ as $\varepsilon\rightarrow0$,
estimate (\ref{smooth-Bernstein-3}) 
cannot hold.
\end{ex}

\subsection{Bernstein estimate in the nested case}\label{subsec:smooth-Bern-est}

We next prove a Bernstein estimate which yields an inverse estimate in the case of
nested spline approximation.

\begin{thm}\label{thm:smooth-Bernstein-1}
Let $0< p<\infty$, $k\ge 2$, $0\le r\le k-2$, $0<s/2<r+1/p$, and $1/\tau=s/2+1/p$.
Then for any  $S\in \cS(n, k, r)$, $n\ge n_0$, we have
\begin{equation}\label{smooth-Bern-1}
|S|_{B^{s, k}_{\tau}} \le cn^{s/2}\|S\|_{L^p},
\end{equation}
where the constant $c>0$ depends only on $s,p, k, r$, and the structural constant of our setting.
\end{thm}

\subsection*{Additional subdivision of \boldmath $\Omega$}
Situations where there are many small rings located next to a large ring create problems.
To be able to deal with such cases we shall additionally subdivide $\Omega$
in two steps.

\smallskip

\noindent
{\em Subdivision of all rings $R\in \cR_n$ into nested hierarchies of rings.}

\begin{lem}\label{lem:hierarchies-of-rings}
There exists a subdivision of every ring $R\in \cR_n$ into a nested multilevel collection of rings
$$
\cK^R = \cup_{m=m_R}^\infty\cK_m^R
$$
with the following properties, where we use the abbreviated notation $\cK_m:=\cK_m^R$:

$(a)$ Every level $\cK_m$ defines a partition of $R$ into rings with
disjoint interiors such that
$R = \cup_{K\in \cK_m} K$.

$(b)$ The levels $\{\cK_m\}_{m\ge m_R}$ are nested, i.e. $\cK_{m+1}$ is a refinement of $\cK_m$,
and each $K\in \cK_m$ has at least 4 and at most $\MM$ children in $\cK_{m+1}$,
where $\MM \ge 4$ is a constant. 

$(c)$ $|R| \le c_1|K|$ for all $K\in \cK_{m_R}$.

$(d)$ We have
\begin{equation}\label{areas}
c_2^{-1}4^{-m} \le |K| \le c_2 4^{-m}, \quad \forall K\in \cK_m, \quad \forall m\ge m_R.
\end{equation}
%
As a consequence we have
$c_3^{-1}4^{-m_R} \le |R| \le c_3 4^{-m_R}$
and
\begin{equation}\label{perimeter}
c_4^{-1} 2^{-m} \le d(K) \le c_4 2^{-m}, \quad \forall K\in \cK_m, \quad \forall m\ge m_R.
\end{equation}

$(e)$ All rings $K\in \cK^R$ are rings without a hole,
except for finitely many of them in the case when $R=Q_1\setminus Q_2$
and $Q_2$ is small relative to $Q_1$.
Then the rings with a hole form a chain $R\supset K_1 \supset K_2 \supset\cdots\supset K_\ell \supset Q_2$.
All sets $K\in \cK^R$ are rings in the sense of Definition~\ref{def-rings}
with structural constants (parameters)
$N_0^*$, $c_0^\star$, and $\beta^\star$.
These and the constants $\MM$ and $c_1, c_2, c_3, c_4>0$ from above depend only on
the initial structural constants $N_0$, $c_0$, and $\beta$.
\end{lem}

\begin{proof}
Observe first that if we are in a setting as the one described in Scenario 1 from \S\ref{subsec:setting-2},
then the needed subdivision is given by the hierarchy of triangulations described there.

In the general case,
let $R=Q_1\setminus Q_2$ be a ring in the sense of Definition~\ref{def-rings},
and assume that $Q_2 \ne \emptyset$.
We subdivide the polygonal convex set $Q_1$ into subrings by
connecting the center of eccentricity of $Q_1$ with, say, 6 points
from the boundary $\partial R$ of $R$, preferably end points of segments on the boundary,
so that the minimum angle condition is obeyed.
After that we subdivide the resulting rings using mid points and connecting them with segments.
Necessary adjustments are made around $Q_2$ depending on the size and location of $Q_2$.
\end{proof}

\noindent
{\em Subdivision of all rings from $\cR_n$ into subrings with disjoint interiors.}
We first pick up all rings from each $\cK^R$, $R\in\cR_n$, see Lemma~\ref{lem:hierarchies-of-rings},
that are needed to handle situations where many small rings are located next to a large ring.

We shall only need the rings in $\cK^R$ that intersect the boundary $\partial R$ of $R$.
Denote the set all such rings by $\Ga^R$ and set $\Ga^R_m:= \Ga^R \cap \cK^R_m$.
We shall make use of the tree structure in $\Ga^R$. More precisely,
we shall use the parent-child relation in $\Ga^R$ induced by the inclusion relation:
Each ring $K \in \Ga^R_m$ has (contains) at least $1$ and at most $M$ children in $\Ga^R_{m+1}$
and has a single parent in $\Ga^R_{m-1}$ or no parent.

We now construct a set $\LL^R$ of rings from $\Ga^R$ which will help prevent
situations where a ring may have many small neighbors.

Given $R\in\cR_n$, we denote by $\cR_n^R$ the set of all rings $\tR\in \cR_n$, $\tR \ne R$,
such that $\tR \cap R \ne \emptyset$ and $d(\tR) \le d(R)$.
These are all rings from $\cR_n$ that are small relative to $R$ and intersect $R$
(are neighbors of $R$).

It will be convenient to introduce the following somewhat geometric terminology:
{\em We say that a ring $K\in \Ga^R$ can see $\tR\in \cR_n^R$
or that $\tR$ is in the range of $K$ if $d(K) \ge d(\tR)$ and $K\cap \tR\ne \emptyset$.}

We now construct $\LL^R$ by applying the following

{\bf Rule}:
{\em We place $K\in \Ga^R$ in $\LL^R$ if $K$ can see some (at least one) rings from $\cR_n^R$
but neither of the children of $K$ in $\Ga^R$ can see all of them.}

We now extend $\LL^R$ to $\tLL^R$ by adding to $\LL^R$ all same level neighbors
of all $K\in \LL^R$, i.e. if $K\in \LL^R$ and $K\in \Ga^R_m$,
then we add to $\LL^R$ each $K'\in \Ga^R_m$ such that $K'\cap K \ne \emptyset$.

The next step is to construct a subdivision of each $R\in\cR_n$ into rings by using $\tLL^R$.
We fix $R\in \cR_n$ and shall suppress the superscript $R$ for the new sets that
will be introduced next and depend on $R$.

Let $\tilde\Ga \subset \Ga^R$ be the minimal subtree of $\Ga^R$ that contains $\tLL^R$,
i.e.
$\tGa$ is the set of all $K \in \Ga^R$ such that $K \supset K'$ for some $K' \in \tLL^R$.
We denote by $\tGa_b$ the set of all {\it branching rings} in $\tGa$
(rings with more than one child in $\tGa$) and by $\tGa_b'$ the set of all
{\it children in $\tGa$ of branching rings} (each of them may or may not belong to $\tGa$).
Furthermore, we let $\tGa_\ell$ denote the set of all {\em leaves} in $\tGa$
(rings in $\tGa$ containing no other rings from $\tGa$).

Evidently, $\tGa_\ell\subset \tLL^R$.
However, rings from $\tGa_b$ and $\tGa_b'$ may or may not belong to $\tLL^R$.
%
We extend $\tLL^R$ to $\ttLL^R := \tLL^R \cup \tGa_b \cup \tGa_b'$.
In addition, we add to $\ttLL^R$ all rings from $\cK^R_{m_R}$, if they are not there yet.

It is readily seen that each ring $\tR \in \cR_n^R$ can be in the range of
only finitely many $K\in \tGa_\ell$
and each ring $\tR\in\cR_n$ may have only finitely many neighbors $R\in\cR_n$
such that $d(R) \ge d(\tR)$.
Therefore,
$$
\sum_{R\in\cR_n} \# \tGa_\ell^R \le cn.
$$
Obviously $\# \tGa_b \le \# \tGa_\ell$,
$\# \tGa_b' \le \MM \# \tGa_b \le M\# \tGa_\ell$,
implying $\# \tLL^R \le \#\tGa_\ell + \#\tGa_b \le c\# \tGa_\ell$,
and hence $\# \ttLL^R \le c'\# \tGa_\ell$.
Putting these estimates together implies
\begin{equation}\label{num-L}
\sum_{R\in\cR_n} \# \ttLL^R \le cn.
\end{equation}

Observe that, with the exception of all branching rings in $\tLL^R$,
by construction every other ring $K\in \tilde\LL^R$ is
either a leaf, and hence contains no other rings from $\ttLL^R$,
or contains only one ring $K'\in\ttLL^R$ of minimum level,
i.e. $K$ has one descendent $K'$ in $\ttLL^R$.

We now make the {\em final step} in our construction:
We denote by $\FF^R$ the set of all rings from $\tGa_\ell^R$
along with all new rings of the form $\overline{K\setminus K'}$,
where $K\in \tGa_b'$, $K'\in \ttLL^R$, $K' \subset K$ and $K'$ is of minimum level with these properties.
Set $\FF := \cup_{R\in\cR_n} \FF^R$.

The purpose of the above construction becomes clear from the the following

\begin{lem}\label{lem:summary}
The set $\FF$ consists of rings in the sense of Definition~\ref{def-rings}
with parameters depending only on the structural constants $N_0$, $c_0$ and $\beta$.
Also, for any $R\in\cR_n$ the rings in $\FF^R$ have disjoint interiors,
$R= \cup_{K\in \FF^R} K$, and $\# \FF^R \le c\#\ttLL^R$.
Hence,
\begin{equation}\label{summary}
\Omega = \bigcup_{R\in\cR_n} \bigcup_{K\in \FF^R} K
\quad\hbox{and}\quad
\sum_{R\in\cR_n} \#\FF^R \le cn.
\end{equation}
Most importantly, each ring $K\in \FF$ has only finitely many neighbors in $\FF$,
that is, there exists a constant $N_1$  such that for any $K\in \FF$ there are
at most $N_1$ rings in $\FF$ intersecting~$K$.
\end{lem}

\begin{proof}
All properties of the newly constructed rings but the last one given in this lemma
follow readily from their construction.

To show that each ring $K\in \FF$ has only finitely many neighbors in $\FF$
we shall need the following technical

\begin{lem}\label{lem:Lam-R}
Suppose $K \supset K_1 \supset K_2$,
$K\in \Ga^R$, $K_1, K_2\in \tLL^R$,
and both $K_1$ and $K_2$ share parts of an edge $E$ of $K$ located in the interior of $R$.
Then there exists
$K^\star\in \tLL^R$ such that $K^\star\cap K^\circ =\emptyset$, $K^\star\cap E\ne \emptyset$,
and $K^\star$ is either a neighbor of $K_1$ or $K_2$,
or $K^\star$ is a neighbor of the parent of $K_1$ in $\Ga^R$.
\end{lem}

\begin{proof}
If $K_1\in \LL^R$, then by construction all same level neighbors of $K_1$ belong $\tLL^R$
and hence the one that shares the edge of $K_1$ contained in $E$ will be in $\tLL^R$.
We denote this ring by $K^\star$ and apparently it has the claimed properties.
By the same token, if $K_2\in \LL^R$, then one of his neighbors will do the job.

Suppose $K_1, K_2\in \tLL^R\setminus\LL^R$. Then $K_1$ has a neighbor, say, $\hat K_1$
that belongs to $\LL^R$ and $\hat K_1$ is at the level of $K_1$.
If $\hat K_1$ has an edge contained in $E$, then $K^\star:=\hat K_1$ has the claimed property.
Similarly, $K_2$ has a neighbor $\hat K_2\in\LL^R$ at the level of $K_2$.
If $\hat K_2$ has an edge contained in $E$, then $K^\star:=\hat K_2$ will do the job.

Assume that neither of the above is true. Then since $K_1, \hat K_1 \in \Ga^R$
they must have the same parent in $\Ga^R$ that has an edge contained in $E$.
Denote this common parent by $K^\sharp$.
For the same reason, $K_2, \hat K_2 \in \Ga^R$ have a common parent,
say, $K^{\sharp\sharp}$ in $\Ga^R$.
Clearly, $K^\sharp$ and $K^{\sharp\sharp}$ have some edges contained in $E$.
Also, $\hat K_1 \subset K^\sharp$, $\hat K_2 \subset K^\sharp$,
and $\hat K_1^\circ \cap \hat K_2^\circ =\emptyset$.

We claim that $K^\sharp$ belongs to $\LL^R$.
Indeed, the rings from $\cR_n$ that are in the range of $\hat K_1$ are also
in the range of $K^\sharp$.
Also, the rings from $\cR_n$ that are in the range of $\hat K_2$ are also
in the range of $K^\sharp$.
However, obviously neither of the children of $K^\sharp$ can have the range of $K^\sharp$.
Therefore, $K^\sharp$ belongs to $\LL^R$.
Now, just as above we conclude that one of the neighbors of $K^\sharp$
has the claimed property.
\end{proof}

We are now prepared to show that each ring $K\in \FF$ has only finitely many neighbors in $\FF$.
By the construction any $K\in \FF^R$, $R\in \cR_n$, has only finitely many neighbors
that do not belong to $\FF^R$.
Thus, it remains to show that it cannot happen that there exist rings
$K_1\subset K_2 \subset \cdots \subset K_J$, $K_j\in \tLL^R$, with $J$ uncontrollably large
that have edges contained in an edge of a single ring $K\in\tLL^R$ whose interior does not intersect
$K_j$, $j=1, \dots, J$.
But this assertion readily follows by Lemma~\ref{lem:Lam-R}.
\end{proof}

The following lemma will be instrumental in the proof of this theorem.

\begin{lem}\label{lem:Nikolski-Markov}
Assume $0< p, q\le \infty$, $k\ge 1$, $r\ge 0$, and $\nn\in \R^2$ with $|\nn|=1$.
Let the sets $G, H\subset \R^2$ be measurable, $G\subset H$,
and such that there exist balls $B_1, B_2, B_3, B_4$, $B_j=B(x_j, r_j)$,
with the properties:
$B_2\subset G \subset B_1$,
$r_1\le \cf r_2$, and
$B_4\subset H \subset B_3$,
$r_3\le \cf r_4$,
where $\cf\ge 1$ is a constant.
Then for any $P\in \Pi_k$
\begin{equation}\label{Nikolski}
\|P\|_{L^p(G)} \le c|G|^{1/p-1/q}\|P\|_{L^q(G)},
\end{equation}
\begin{equation}\label{Markov}
\|D_\nn^r P\|_{L^p(G)} \le c d(G)^{-r}\|P\|_{L^p(G)},
\end{equation}
and
\begin{equation}\label{norms}
\|P\|_{L^p(G)} \le c(|G|/|H|)^{1/p}\|P\|_{L^p(H)},
\end{equation}
where $c>0$ is a constant depending on $p, q, k, r, \cf$, and
the parameters $N_0$, $c_0$, and $\beta$
from Definition~\ref{def-rings}.
Here $D_\nn^r S$ is the $r$th directional derivative of $S$ in the direction of $\nn$.

Furthermore, inequality $(\ref{norms})$ holds with $Q$ and $H$ replaced by
their images $L(G)$ and $L(H)$, where $L$ is a nonsingular linear transformation of $\R^2$.
\end{lem}

\begin{proof}
Inequality (\ref{Nikolski}) holds whenever $B_2=B(0, 1)$ and $B_1=B(0, \cd)$ with $\cd={\rm constant}$
by the fact that any two (quasi)norms on $\Pi_k$ are equivalent.
This implies that (\ref{Nikolski}) is valid in the case when
$B_2=B(0, 1)$ and $B_2\subset B_1$, where $B_1=B(x_2, \cd/2)$.
Then (\ref{Nikolski}), in general, follows by rescaling.
Inequality (\ref{norms}) is obvious when $p=\infty$.
In general, it follows from the case $p=\infty$ and application of (\ref{Nikolski})
to $G$ with $p$ and $q=\infty$ and to $H$ with $p=\infty$, $q=p$.
Inequality (\ref{Markov}) is an easy consequence of the Markov inequality for univariate polynomials
whenever $G$ is a square.  Then in general it follows by inscribing $B_1$ in
a smallest possible cube and then applying it for the cube and using (\ref{norms}).
The last claim in the lemma is obvious.
\end{proof}

\medskip

\noindent
{\em Proof of Theorem~\ref{thm:smooth-Bernstein-1}.}
We shall only consider the case when $\Omega \subset \R^2$ is a compact polygonal domain.
Let $S\in \cS(n, k, r)$ and suppose $S$ is represented as in (\ref{def-smooth-splines-1}),
that is,
\begin{equation}\label{rep-linear-1}
S=\sum_{R\in\cR_n} P_R\ONE_{R},
\quad S\in W^r(\Omega),
\quad P_R\in\Pi_k,
\end{equation}
where $\cR_n$ is a collection of $\le n$ rings with disjoint interiors such that
$\Omega=\cup_{R\in\cR_n} R$.

We are now prepared to complete the proof of Theorem~\ref{thm:smooth-Bernstein-1}.
From (\ref{rep-linear-1}) and because $\FF$ is a refinement of $\cR_n$ it follows that
$S$ can be represented in the form
\begin{equation}\label{rep-linear-2}
S=\sum_{K\in\FF} P_K\ONE_{K},
\quad \quad S\in W^r(\Omega),
\quad P_K\in\Pi_k.
\end{equation}
Here $\FF$ is the collection of at most $cn$ rings from above with disjoint interiors such that
$\Omega=\cup_{K\in\FF} K$.

We next introduce some convenient notation.
For any ring $K\in \FF$ we denote by $\GG_K$ the set of all rings $K'\in \FF$ such that
$K\cap K' \ne \emptyset$,
$\cE_K$ will denote the set of all segments (edges)
from the boundary $\partial K$ of $K$,
and $\cV_K$ will be the set of all vertices of the polygon $\partial K$ (end points of edges from $\cE_K$).

The fact that $\FF$ consists of rings in the sense of Definition~\ref{def-rings} implies the following

\begin{property}\label{property-1}
There exists a constant $0< \cch <1$
such that if $E=[v_1,v_2]$ is an edge shared by two rings $K, K'\in \FF$
then for any $x\in E$ with $|x-v_j| \ge \rho$, $j=1, 2$ for some $\rho>0$
we have $B(x, \cch \rho) \subset K\cup K'$.
\end{property}

Fix $t>0$. For each ring $K\in \FF$  we define
$$
K_t:= \{x\in K: \dist(x, \partial K) \le kt\}.
$$
Write
$\Omega_t:=\cup_{K\in\FF} K_t$.

Let $h\in \R^2$ with norm $|h|\le t$ and set $\nn:=|h|^{-1}h$.
For $S$ is a polynomial of degree $\le k-1$ on each $K\in\FF$ we have
$\Delta^k_hS(x) =0$ for $x\in \cup_{K\in\FF} K\setminus K_t$.
Therefore,
\begin{equation}\label{Delta-Delta-K}
\|\Delta^k_hS\|_{L^\tau(\Omega)} = \|\Delta^k_hS\|_{L^\tau(\Omega_t)}.
\end{equation}

Let $K\in\FF$ and assume $d(K) > 2kt/\cch$  with $0<\cch<1$ the constant from Property~\ref{property-1}.
Denote $\GG_K^t:= \{K'\in \GG_K: d(K) > 2kt/\cch\}$,
$B_v:=B(v, 2kt/\cch)$, $v\in \cV_K$,
and
$$
\cH_K^t:= \{K'\in \FF: d(K') > 2kt/\cch \;\;\hbox{and} \;\; K'\cap (K+B(0, 2kt/\cch))\ne \emptyset\}.
$$
Observe that because $d(K) > 2kt/\cch$ the number of rings in $\cH_K^t$ is uniformly bounded.

Let $x\in \Omega_t$ be such that $[x, x+kh]\cap K \ne \emptyset$.
Two cases are to be considered here.

(a) Let $[x, x+kh] \not\subset \cup_{v\in\cV_K} B_v$.
Then $[x, x+kh]$ intersects some edge $E\in \cE_K$ such that $\ell(E) \ge 2kt/\cch$,
and $[x, x+kh]$ cannot intersect another edge $E'\in \cE_K$ with this property
or an edge $E'\in \cE_K$ with $\ell(E') < 2kt/\cch$.

Suppose that the edge $E=:[v_1, v_2]$ is shared with $K'\in\FF$
and $y:= E\cap[x, x+kh]$.
Evidently, $|y-v_j| > kt/\cch$, $j=1,2$, and in light of Property~\ref{property-1} we have
$[x, x+kh] \subset B(y, kt) \subset K\cup K'$.
Clearly,
\begin{equation}\label{est-Delta-S1}
|\Delta_h^k S(x)|
\le ct^r\|D_\nn^rS\|_{L^\infty([x, x+kh])}
\le ct^r\|D_\nn^rS\|_{L^\infty(K)} + ct^r\|D_\nn^rS\|_{L^\infty(K')}.
\end{equation}

(b) Let $[x, x+kh] \subset \cup_{v\in\cV_K} B_v$.
Then we estimate $|\Delta_h^k S(x)|$ trivially:
\begin{equation}\label{est-Delta-S2}
|\Delta_h^k S(x)|
\le 2^k\sum_{\ell=0}^k|S(x+\ell h)|.
\end{equation}
Using (\ref{est-Delta-S1}) - (\ref{est-Delta-S2}) we obtain
\begin{align}\label{est-Delta-S3}
\|\Delta_h^k &S\|_{L_\tau(K_t)}^\tau
\le c\sum_{K'\in\GG_K^t} t d(K') t^{r\tau}\|D_\nn^rS\|_{L^\infty(K')}^\tau\notag\\
&+ c\sum_{K'\in\cH_K^t}\sum_{v\in \cV_{K}} \|S\|_{L^\tau(B_v\cap K')}^\tau
+ c\sum_{K''\in\FF: d(K'') \le 2kt/\cch}\|S\|_{L^\tau(K''\cap(K+[0, kh]))}^\tau.
\end{align}
Note that the number of rings $K'\in\cH_K^t$ such that $K'\cap B_v\ne \emptyset$
for some $v\in\cV_K$ is uniformly bounded.

By Lemma~\ref{lem:Nikolski-Markov} it follows that
$\|D_\nn^rS\|_{L^\infty(K')} \le c d(K')^{-r-2/p}\|S\|_{L^p(K')}$
and if the ring $K'\in \cH_K^t$ and $v\in \cV_{K}$, then
\begin{align*}
\|S\|_{L^\tau(B_v\cap K')}^\tau
&\le c(|B_v|/|K'|)\|S\|_{L^\tau(K')}^\tau
\le ct^2|K'|^{-1}\|S\|_{L^\tau(K')}^\tau\\
&\le ct^2|K'|^{-1+\tau(1/\tau-1/p)}\|S\|_{L^p(K')}^\tau
\le ct^2d(K')^{\tau s-2}\|S\|_{L^p(K')}^\tau.
\end{align*}
We use the above estimates in (\ref{est-Delta-S3}) to obtain
\begin{align}\label{est-Delta-S4}
\|\Delta_h^k S&\|_{L_\tau(K_t)}^\tau
\le c\sum_{K'\in\GG_K^t} t^{1+r\tau} d(K')^{1-r\tau-2\tau/p}\|S\|_{L^p(K')}^\tau\notag\\
&+ c\sum_{K'\in\cH_K^t} t^{2} d(K')^{\tau s-2}\|S\|_{L^p(K')}^\tau
+ c\sum_{K''\in\FF: d(K'') \le 2kt/\cch}\|S\|_{L^\tau(K''\cap(K+[0, kh]))}^\tau.
\end{align}

Denote by $\Omega_t^\star$ the set of all $x\in \Omega_t$ such that
$[x, x+kh] \subset \Omega$ and
$$
[x, x+kh] \subset \cup \{K\in\FF: d(K) \le 2kt/\cch\}.
$$
In this case we shall use the obvious estimate
$$
\|\Delta_h^kS\|_{L^\tau(\Omega_t^\star)}^\tau
\le c \sum_{K\in\FF: d(K) \le 2kt/\cch}\|S\|_{L^\tau(K)}^\tau.
$$
This estimate along with (\ref{est-Delta-S4}) yields
\begin{align*}
\omega_k(S, t)_\tau^\tau
&\le c \sum_{K\in\FF: d(K) \ge 2kt/\cch}  t^{1+r\tau}d(K)^{1-r\tau-2\tau/p}\|S\|_{L^p(K)}^\tau \notag\\
&
+ c\sum_{K\in\FF: d(K) \ge 2kt/\cch} t^2 d(K)^{s\tau-2}\|S\|_{L^p(K)}^\tau
+ c\sum_{K\in\FF: d(K) \le 2kt/\cch} \|S\|_{L^\tau(K)}^\tau.
\end{align*}
Here we used the fact that only finitely many (uniformly bounded number)
of the rings involved in the above estimates may overlap at a time
due to Lemma~\ref{lem:summary}.
For the norms involved in the last sum we use the estimate
$
\|S\|_{L^\tau(K)}^\tau \le cd(K)^{s\tau}\|S\|_{L^p(K)}^\tau,
$
which follows by Lemma~\ref{lem:Nikolski-Markov},
to obtain
\begin{align*}
&\omega_k(S, t)_\tau^\tau
\le c \sum_{K\in\FF: d(K) \ge 2kt/\cch}  t^{1+r\tau}d(K')^{1-r\tau-2\tau/p}\|S\|_{L^p(K')}^\tau \notag\\
\qquad &
+ c\sum_{K\in\FF: d(K) \ge 2kt/\cch} t^2 d(K)^{s\tau-2}\|S\|_{L^p(K)}^\tau
+ c\sum_{K\in\FF: d(K) \le 2kt/\cch} d(K)^{s\tau}\|S\|_{L^p(K)}^\tau.
\end{align*}
We insert this estimate in (\ref{def-smooth-Besov}) and
interchange the order of integration and summation to obtain
\begin{align*}
|S|_{B^{s,k}_\tau}^\tau
&=\int_0^\infty t^{-s\tau-1}\omega_k(S, t)_\tau^\tau dt
\le c \sum_{K\in\FF}  d(K)^{1-r\tau-2\tau/p}\|S\|_{L^p(K)}^\tau \int_0^{\cch d(K)/2k} t^{-s\tau+r\tau} dt \\
\qquad &
+ c\sum_{K\in\FF} d(K)^{s\tau-2}\|S\|_{L^p(K)}^\tau \int_0^{\cch d(K)/2k} t^{-s\tau+1} dt\\
&+ c\sum_{K\in\FF} d(K)^{s\tau}\|S\|_{L^p(K)}^\tau \int_{\cch d(K)/2k}^\infty t^{-s\tau-1} dt.
\end{align*}
Observe that $-s\tau + r\tau > -1$ is equivalent to $s/2<r+1/p$ and
$-s\tau + 1 > -1$ is equivalent to  $s<2/\tau = s+2/p$.
Therefore, the above integrals are convergent and taking into account that
$2-2\tau/p-s\tau = 2\tau(1/\tau-1/p-s/2) = 0$ we obtain
\begin{align*}
|S|_{B^{s,k}_\tau}^\tau
\le c\sum_{K\in\FF} \|S\|_{L^p(K)}^\tau
\le cn^{\tau(1/\tau-1/p)}\Big(\sum_{K\in\FF} \|S\|_{L^p(K)}^\tau\Big)^{\tau/p}
=cn^{\tau s/2}\|S\|_{L^p(\Omega)}^\tau,
\end{align*}
where we used H\"{o}lder's inequality.
This completes the proof.
\qed

\subsection{Proof of the Bernstein estimate (Theorem~\ref{thm:smooth-Bernstein-2})
in the nonnested case}\label{subsec:proof-smooth-Bern-est-2}

For the proof of Theorem~\ref{thm:smooth-Bernstein-2}
we combine ideas from the proofs of
Theorem~\ref{thm:Bernstein} and Theorem~\ref{thm:smooth-Bernstein-1}.
We shall adhere to a large extent to the notation introduced in
the proof of Theorem~\ref{thm:Bernstein} in \S\ref{subsec:proof-Bernstein}.
An important distinction between this proof and the proof of Theorem~\ref{thm:Bernstein}
is that the directional derivatives $D_\nn^{k-1}S$ of any $S\in \cS(n, k)$
are piecewise constants along the respective straight lines rather than
$S$ being a piecewise constant.

We consider the case when $\Omega \subset \R^2$ is a compact polygonal domain.
Assume $S_1, S_2\in \cS(n, k)$, $n\ge n_0$.
Then each $S_j$ ($j=1, 2$) can be represented in the form
$
S_j=\sum_{R\in \cR_j} P_R \ONE_{R},
$
where $\cR_j$ is a set of at most $n$ rings in the sense of Definition~\ref{def-rings}
with disjoint interiors and such that
$\Omega = \cup_{R\in \cR_j} R$, $P_R \in \Pi_k$, and
$S_j\in W^{k-2}(\Omega)$.

Just as in the proof of Theorem~\ref{thm:smooth-Bernstein-1} there exist
subdivisions $\FF_1$, $\FF_2$ of the rings from $\cR_1$, $\cR_2$ with
the following properties, for $j=1, 2$:

(a) $\FF_j$ consists of rings in the sense of Definition~\ref{def-rings}
with parameters $N_0^\star$, $c_0^\star$, and $\beta^\star$
depending only on the structural constants $N_0$, $c_0$, and $\beta$.

(b) $\cup_{R\in\FF_j} R = \Omega$ and $\# F_j \le cn$.

(c) There exists a constant $N_1$  such that for any $R\in \FF_j$ there are
at most $N_1$ rings in $\FF_j$ intersecting~$R$ ($R$ has $\le N_1$ neighbors in $\FF_j$).

(d) $S_j$ can be represented in the form
$S_j = \sum_{R\in \FF_j} P_R\ONE_{R}$ with $P_R\in \Pi_k$.

Now, just as in the proof of Theorem~\ref{thm:Bernstein} we denote by $\cU$
the collection of all maximal connected sets obtained by intersecting
rings from $\FF_1$ and $\FF_2$.
By (\ref{num-intesect}) there exists a constant $c>0$ such that
$ 
\# \cU \le cn.
$ 

We claim that there exists a constant $N_2$ such that
for any $U\in\cU$ there are no more than $N_2$ sets $U'\in\cU$ which intersect $U$,
i.e. $U$ has at most $N_2$ neighbors in $\cU$.
Indeed, let $U\in \cU$ be a maximal connected component of $R_1\cap R_2$
with $R_1\in \FF_1$ and $R_2\in \FF_2$.
Then using the fact that the ring $R_1$ has finitely many neighbors in $\FF_1$
and $R_2$ has finitely many neighbors in $\FF_2$ we conclude that
$U$ has finitely many neighbors in $\cU$.

Further, we introduce the sets $\cA$ and $\cT$
just as in the proof of Theorem~\ref{thm:Bernstein}.

\subsection*{Trapezoids}
Our main concern will be in dealing with the trapezoids $T\in\cT$.
We next use the fact that any ring from $\FF_j$, $j=1, 2$, has at most $N_1$ neighbors in $\FF_j$
to additionally subdivide the trapezoids from $\cT$ into trapezoids whose long sides
are sides of good triangles for rings in $\FF_1$ or $\FF_2$.

Consider an arbitrary trapezoid $T\in\cT$.
Just as in \S\ref{subsec:proof-Bernstein} we may assume that
$T$ is a maximal isosceles trapezoid contained in $\tri_{E_1}\cap\tri_{E_2}$,
where $\tri_{E_j}$ ($j=1, 2$) is a good triangle for a ring $R_j\in\FF_j$,
and
$T$ is positioned so that its vertices are the points:
$$
v_1:=(-\delta_1/2, 0),\;\; v_2:=(\delta_1/2, 0),\;\; v_3:=(\delta_2/2, H),\;\; v_4:=(-\delta_2/2, H),
$$
where $0\le\delta_2\le\delta_1$ and $H >\delta_1$.
Let $L_1:=[v_1, v_4]$ and $L_2:=[v_2, v_3]$ be the two equal (long) legs of $T$.
We assume that $L_1\subset E_1$ and $L_2\subset E_2$.
See Figure 8.

By Lemma~\ref{lem:summary} it follows that
there exist less than $N_1$ rings $K_\ell'\in \FF_1$, $\ell=1, \dots,\nu'$,
each of them with an edge or part of an edge contained in $L_1$.
By Definition~\ref{def-rings}, each of them can be subdivided into at most two segments
so that each of these is a side of a good triangle.
Denote by $I_\ell'$, $\ell=1, \dots,m'$, these segments,
and by $\tri_{I_\ell'}$, $\ell=1, \dots,m'$, the respective good triangles attached to them.
More precisely, $I_\ell'$ is a side of $\tri_{I_\ell'} \subset K_\ell'$ and
$\tri_{I_\ell'}$ is a good triangle for $K_\ell'$.
Thus we have
$L_1=\cup_{\ell=1}^{m'} I_\ell'$,
where the segments $I_\ell'$, $\ell=1, \dots,m'$, are with disjoint interiors.

Similarly, there exist segments $I_\ell''$, $\ell=1, \dots,m''$,
and attached to them good triangles $\tri_{I_\ell''}$, $\ell=1, \dots,m''$, for rings from $\FF_2$,
so that $L_2=\cup_{\ell=1}^{m''} I_\ell''$.

Denote by $v_\ell'$, $\ell=1, \dots,m'+1$, the vertices of
the triangles $\tri_{I_\ell'}$, $\ell=1, \dots,m'$, on $L'$ so that
$I_\ell'=[v_\ell', v_{\ell+1}']$
and assume that their orthogonal projections onto the $x_2$-axis
$p_\ell'$, $\ell=1, \dots,m'+1$, are ordered so that
$0=p_1' < p_2' <\cdots <p_{m'+1}'=H$.
Exactly in the same way we define the vertices $v_\ell''$, $\ell=1, \dots,m''+1$,
of the triangles $\tri_{I_\ell''}$ and their projections onto the $x_2$-axis
$0=p_1'' < p_2'' <\cdots <p_{m''+1}''=H$.

For any $q\in [0, H]$ we let $\delta(q)$ be the distance between the points
where the line with equation $x_2=q$ intersects $L_1$ and $L_2$.
Thus $\delta(0) = \delta_1$ and $\delta(H) = \delta_2$,
and $\delta(q)$ is linear.

Inductively, starting from $q_1=0$ one can easily subdivide the interval $[0, H]$ by means of points
$$
0=q_1 < q_2 <\cdots < q_{m+1}=H, \quad\hbox{$m \le m'+m'' \le 2N_1$}
$$
with the following properties, for $k=1, \dots, m$, either

(a) $\delta(q_k)\le q_{k+1}-q_k <2\delta(q_k)$

\noindent
or

(b) $q_{k+1}-q_k > \delta(q_k)$ and
$(q_k, q_{k+1})$ contains no points $p_\ell'$ or $p_\ell''$.

We use the above points to subdivide the trapezoid $T$.
Let $T_k$, $k=1, \dots, m$, be the trapezoid bounded by $L_1$, $L_2$, and
the lines with equations $x_2=q_k$ and $x_2=q_{k+1}$.

We now separate the ``bad" from the ``good" trapezoids $T_k$.
Namely, if property (a) from above is valid then
$T_k$ is a ring and we place $T_k$ in $\cA$;
if property (b) is valid, then $T_k$ is a ``bad" trapezoid and
we place $T_k$ in $\cT$.
We apply the above procedure to all trapezoids.

\subsection*{Properties of New Trapezoids}

We now consider an arbitrary trapezoid $T$ from the above defined $\cT$ (the set of bad trapezoids).
We next summarise the properties of $T$.
It will be convenient to us to use the same notation as above
as well as in the proof of Theorem~\ref{thm:Bernstein}.
We assume that $T$ is an isosceles trapezoid contained in $\tri_{E_1}\cap\tri_{E_2}$,
where $\tri_{E_j}$, $j=1, 2$, is a good triangle for a ring $R_j\in\FF_j$,
and
$T$ is positioned so that its vertices are the points:
$$
v_1:=(-\delta_1/2, 0),\;\; v_2:=(\delta_1/2, 0),\;\; v_3:=(\delta_2/2, H),\;\; v_4:=(-\delta_2/2, H),
$$
where $0\le\delta_2\le\delta_1$ and $H >\delta_1$.
Let $L_1:=[v_1, v_4]$ and $L_2:=[v_2, v_3]$ be the two equal (long) sides of $T$.
We assume that $L_1\subset E_1$ and $L_2\subset E_2$.
See Figure 8.

As a result of the above subdivision procedure, there exists a triangle
$\tri_{L_1}$ with a side $L_1$ such that $\tri_{L_1}$ is a good triangle for some
ring $\tR_1\in \FF_1$
and $\tri_{L_1}^\circ\cap \tri_{E_1}^\circ =\emptyset$.
For the same reason,
there exists a triangle
$\tri_{L_2}$ with a side $L_2$ such that $\tri_{L_2}$ is a good triangle for some
ring $\tR_2\in \FF_2$
and $\tri_{L_2}^\circ\cap \tri_{E_2}^\circ =\emptyset$.

Observe that $\tri_{E_1}$ and $\tri_{E_2}$ are good triangles and hence
the angles of $\tri_{E_j}$ adjacent to $E_j$ are of size $\beta^\star/2$, $j=1,2$.
Likewise, $\tri_{L_1}$ and $\tri_{L_2}$ are good triangles and hence
the angles of $\tri_{L_j}$ adjacent to $L_j$ are of size $\beta^\star/2$, $j=1,2$.
Therefore, we may assume that $\tri_{L_1} \subset \tri_{E_2}$ and $\tri_{L_1} \subset \tri_{E_2}$.
Consequently, $S_1$ is a polynomial of degree $<k$ on $\tri_{L_1}$
and another polynomial of degree $<k$ on $\tri_{L_2}$.
By the same token, $S_2$ is a polynomial of degree $<k$ on $\tri_{L_1}$
and another polynomial of degree $<k$ on $\tri_{L_2}$.
We shall assume that $\tri_{L_1} \subset A_1$ and $\tri_{L_2} \subset A_2$,
where $A_1, A_2\in \cA$.

Further, denote by $D_1$ and $D_2$ the bottom and top sides of $T$.
We shall denote by $\cV_T=\{v_1, v_2, v_3, v_4\}$ the vertices of $T$,
where $v_1$ is the point of intersection of $L_1$ and $D_1$
and the other vertices are indexed counter clockwise.

We shall use the notation $\delta_1(T) :=\delta_1$ and $\delta_2(T) :=\delta_2$.
We always assume that $\delta_1(T) \ge \delta_2(T)$.
Clearly, $d(T)\sim H$;
more precisely $H < d(T) < H +\delta_1+ \delta_2$.

\smallskip

Observe that by the construction of the sets $\cT$, $\cA$,  and (\ref{num-intesect}) it follows that
$\cA \cup \cT$ consists of polygonal sets with disjoint interiors,
$\cup_{A\in \cA} A \cup_{T\in \cT} T=\Omega$,
there exists a constant $c>0$ such that
\begin{equation}\label{card-UAT}
\# \cA \le cn, \quad \# \cT \le cn,
\end{equation}
and there exists a constant $N_3$ such that each set from $\cA\cup \cT$
has at most $N_3$ neighbors in $\cA\cup\cT$.

We summarize the most important properties of the sets from $\cT$ and $\cA$
in the following

\begin{lem}\label{lem:property-2}
The following properties hold for some constant $0< \ct <1$
depending only on the structural constants $N_0$, $c_0$ and $\beta$ of the setting:


$(a)$ Let $T\in \cT$ and assume the notation related to $T$ from above.
If $x\in L_1$ with $|x-v_j| \ge \rho$, $j=1, 4$,
then $B(x, \ct \rho) \subset \tri_{L_1}\cup\tri_{L_2}$.
Also, if $x\in L_2$ with $|x-v_j| \ge \rho$, $j=2, 3$,
then $B(x, \ct \rho) \subset \tri_{L_1}\cup\tri_{L_2}$.
Furthermore, if $x\in D_1$ with $|x-v_j| \ge \rho$, $j=1, 2$,
then $B(x, \ct \rho) \subset \tri_{E_1}\cap\tri_{E_2}$,
and similarly for $x\in D_2$.

$(b)$ Assume that $E=[w_1,w_2]$ is an edge shared by two sets $A, A'\in \cA$.
Let $\cV_A$ be the set of all vertices on $\partial A$ (end points of edges)
and let $\cV_{A'}$ be the set of all vertices on $\partial A'$.
If $x\in E$ with $|x-w_j| \ge \rho$, $j=1, 2$, for some $\rho>0$, then
\begin{equation}\label{cover-ball}
B(x, \ct \rho) \subset A\cup A' \cup_{v\in \cV_A\cup\cV_{A'}} B(v, \rho).
\end{equation}
\end{lem}

\begin{proof}
Part (a) of this lemma follows readily from the properties of the trapezoids.
Part (b) needs clarification.
Suppose that for some
$x\in E$ with $|x-w_j| \ge \rho$, $j=1, 2$, $\rho>0$,
the inclusion (\ref{cover-ball}) is not valid.
Then exists a point $y$ from an edge $\tilde{E}=[u_1, u_2]$ of, say, $\partial A$ such that
$|y-x|<\rho$ and $|y-u_j|\ge \rho$, $j=1, 2$.
A~simple geometric argument shows that if the constant $\ct$ is sufficiently small
(depending only on the parameter $\beta$ of the setting),
then there exists an isosceles trapezoid $\check{T} \subset \tri_{E}\cap \tri_{\tilde{E}}$
with two legs contained in $E$ and $\tilde{E}$ such that each leg is longer than its larger base.
But then the subdivision of the sets from $\cU$ (see the proof of Theorem~\ref{thm:Bernstein})
would have created a trapezoid in $\cT$ that contains part of $A$.
This is a contradiction which shows that Part (b) holds true.
\end{proof}

We have the representation
\begin{equation}\label{rep-S1-S2-smooth}
S_1(x)-S_2(x) = \sum_{A\in \cA} P_A\ONE_A(x)
+ \sum_{T\in \cT} P_T\ONE_T(x),
\end{equation}
where $P_A, P_T\in \Pi_k$.
Note that $S_1-S_2\in W^{k-2}(\Omega)$.

Let $0<s/2<k-1+1/p$ and $\tau \le 1$.
Fix $t>0$ and
let $h\in \R^2$ with norm $|h|\le t$.
Write $\nn:=|h|^{-1}h$ and assume $\nn=:(\cos \theta, \sin\theta)$, $-\pi<\theta\le \pi$.

Since $S_1, S_2\in W^{k-2}(\Omega)$ we have the following representation of $\Delta^{k-1}_hS_j(x)$:
$$ 
\Delta^{k-1}_hS_j(x)
= |h|^{k-1}\int_\R D_\nn^{k-1} S_j\big(x+u\nu\big)M_{k-1}(u) du,
$$ 
where $M_{k-1}(u)$ is the B-spline with knots $u_0, u_1, \dots, u_{k-1}$,
$u_\ell:= \ell|h|$.
In fact,
$
M_{k-1}(u)=(k-1)[u_0, \dots, u_{k-1}](\cdot - u)_+^{k-2}
$
is the divided difference.
As is well known, $0\le M_{k-1} \le c|h|^{-1}$,
$\supp M_{k-1} \subset [0, (k-1)|h|]$, and $\int_\R M_{k-1}(u) du =1$.
Therefore, by
$\Delta^k_hS_j(x)=\Delta^{k-1}_hS_j(x+h)-\Delta^{k-1}_hS_j(x)$,
whenever $[x, x+kh]\subset \Omega$,
we arrive at the representation
\begin{equation}\label{rep-Delta-k}
\Delta^k_hS_j(x)
= |h|^{k-1}\int_0^{k|h|} D_\nn^{k-1} S_j\big(x+uv\big)M^*_k(u) du,
\end{equation}
where $M^*_k(u):= M_{k-1}(u-|h|)-M_{k-1}(u)$.

\smallskip

In what follows we estimate
$\|\Delta_h^k S_1\|_{L^\tau(G)}^\tau - \|\Delta_h^k S_2\|_{L^\tau(G)}^\tau$
for different subsets $G$ of $\Omega$.


\subsection*{Case 1}

Let $T \in \cT$ be such that $d(T)> 2kt/\ct$ with $\ct$ the constant from Lemma~\ref{lem:property-2}.
Denote
$$
T_h:=\{x\in\Omega: [x, x+kh] \subset \Omega\;\;\hbox{and}\;\;[x, x+kh]\cap T\ne \emptyset\}.
$$
We next estimate
$\|\Delta_h^k S_1\|_{L^\tau(T_h)}^\tau - \|\Delta_h^k S_2\|_{L^\tau(T_h)}^\tau$.

Assume that $T \in \cT$ is a trapezoid positioned as described above in
Properties of New Trapezoids.
We adhere to the notation introduced there.

In addition, let $v_4-v_1 =: |v_4-v_1| (\cos \gamma, \sin\gamma)$ with $\gamma \le \pi/2$,
i.e. $\gamma$ is the angle between $D_1$ and $L_1$.
Assume that  $\nn=:(\cos \theta, \sin\theta)$ with $\theta \in [\gamma, \pi]$.
The case $\theta \in [-\gamma, 0]$ is just the same.
The case when $\theta \in [0, \gamma]\cup [-\pi, -\gamma]$
is considered similarly.

We set $B_v:=B(v, 2kt/\ct)$, $v\in \cV_T$.
Also, denote
\begin{align*}
\cA_T^t &:=\{A\in\cA: d(A) > 2kt/\ct\quad\hbox{and}\quad A\cap(T+B(0, kt))\ne \emptyset\},\\
\fA_T^t &:=\{A\in\cA: d(A) \le 2kt/\ct\quad\hbox{and}\quad A\cap(T+B(0, kt))\ne \emptyset\}
\end{align*}
and
\begin{align*}
\cT_T^t:=\{T'\in\cT: d(T') > 2kt/\ct\quad\hbox{and}\quad T'\cap(T+B(0, kt))\ne \emptyset\},\\
\fTT_T^t:=\{T'\in\cT: d(T') \le 2kt/\ct\quad\hbox{and}\quad T'\cap(T+B(0, kt))\ne \emptyset\}.
\end{align*}
Clearly, $\# \cA_T^t \le c$ and $\# \cT_T^t \le c$ for some constant $c>0$.

\smallskip

{\em Case 1 (a).}
If $[x, x+kh]\subset \tri_{E_1}$, then $\Delta_h^k S_1(x)=0$
because $S_1$ is a polynomial of degree $<k$ on $\tri_{E_1}$.
Hence no estimate is needed.

\smallskip

{\em Case 1 (b).}
If $[x, x+kh] \subset \cup_{v\in\cV_{T}} B_v$,
we estimate $|\Delta_h^k S_1(x)|$ trivially:
\begin{equation}\label{est-Delta-S1-triv}
|\Delta_h^k S_1(x)|
\le |\Delta_h^k S_2(x)| + 2^k\sum_{\ell=0}^k|S_1(x+\ell h)-S_2(x+\ell h)|.
\end{equation}
Clearly, the contribution of this case to estimating
$\|\Delta_h^k S_1\|_{L^\tau(T_h)}^\tau - \|\Delta_h^k S_2\|_{L^\tau(T_h)}^\tau$~is
\begin{align*}
&\le c \sum_{v\in \cV_T}\sum_{A\in \cA_T^t}\|S_1-S_2\|_{L^\tau(B_v\cap A)}^\tau
+ c\sum_{v\in\cV_T}\sum_{T'\in\cT_T^t}\|S_1-S_2\|_{L^\tau(B_v\cap T')}^\tau\\
& + c \sum_{v\in \cV_T}\sum_{A\in \fA_T^t}\|S_1-S_2\|_{L^\tau(B_v\cap A)}^\tau
+ c\sum_{v\in\cV_T}\sum_{T'\in\fTT_T^t}\|S_1-S_2\|_{L^\tau(B_v\cap T')}^\tau\\
&\le \sum_{A\in \cA_T^t}ct^2d(A)^{\tau s-2}\|S_1-S_2\|_{L^p(A)}^\tau
 + \sum_{T'\in\cT_T^t} ct^{1+\tau s/2}d(T')^{\tau s/2-1}\|S_1-S_2\|_{L^p(T')}^\tau\\
& + \sum_{A\in \fA_T^t} cd(A)^{\tau s}\|S_1-S_2\|_{L^p(A)}^\tau
+ \sum_{T'\in\fTT_T^t} cd(T')^{\tau s}\|S_1-S_2\|_{L^p(T')}^\tau.
\end{align*}
Here we used the following estimates, which are a consequence of Lemma~\ref{lem:Nikolski-Markov}:\\
(1) If $A\in \cA_T^t$, then
\begin{align*}
\|S_1-S_2\|_{L^\tau(B_v\cap A)}^\tau
&\le c(|B_v|/|A|)\|S_1-S_2\|_{L^\tau(A)}^\tau\\
&\le ct^2d(A)^{-2}\|S_1-S_2\|_{L^\tau(A)}^\tau
\le ct^2d(A)^{\tau s-2}\|S_1-S_2\|_{L^p(A)}^\tau.
\end{align*}
(2) If $T'\in\cT_T^t$ and $\delta_1(T') > 2kt/\ct$,
then for any $v\in\cV_{T}$ we have
\begin{align*}
\|S_1-S_2\|_{L^\tau(B_v\cap T')}^\tau
&\le c(|B_v|/|T'|)\|S_1-S_2\|_{L^\tau(T')}^\tau
\le ct^2|T'|^{\tau s/2-1}\|S_1-S_2\|_{L^p(T')}^\tau\\
&\le ct^2\delta_1(T')^{\tau s/2-1}d(T')^{\tau s/2-1}\|S_1-S_2\|_{L^p(T')}^\tau\\
&\le ct^{1+\tau s/2}d(T')^{\tau s/2-1}\|S_1-S_2\|_{L^p(T')}^\tau,
\end{align*}
where we used that $\tau s/2 <1$, which is equivalent to $s<s+2/p$.

\noindent
(3) If $T'\in\cT_T^t$ and $\delta_1(T') \le 2kt/\ct$,
then for any $v\in\cV_{T}$
\begin{align*}
\|S_1-S_2\|_{L^\tau(B_v\cap T')}^\tau
&\le c(|B_v\cap T'|/|T'|)\|S_1-S_2\|_{L^\tau(T')}^\tau\\
&\le ct\delta_1(T')[\delta_1(T')d(T')]^{-1}\|S_1-S_2\|_{L^\tau(T')}^\tau\\
&\le ctd(T')^{-1}[\delta_1(T')d(T')]^{\tau s/2}\|S_1-S_2\|_{L^p(T')}^\tau\\
&\le ct^{1+\tau s/2}d(T')^{\tau s/2-1}\|S_1-S_2\|_{L^p(T')}^\tau.
\end{align*}
(4) If $A\in \fA_T^t$, then
\begin{align*}
\|S_1-S_2\|_{L^\tau(B_v\cap A)}^\tau
&\le \|S_1-S_2\|_{L^\tau(A)}^\tau
\le c|A|^{\tau s/2}\|S_1-S_2\|_{L^p(A)}^\tau\\
&\le cd(A)^{\tau s}\|S_1-S_2\|_{L^p(A)}^\tau.
\end{align*}
(5) If $T'\in \fTT_T^t$, then
\begin{align*}
\|S_1-S_2\|_{L^\tau(B_v\cap T')}^\tau
&\le \|S_1-S_2\|_{L^\tau(T')}^\tau
\le c|T'|^{\tau s/2}\|S_1-S_2\|_{L^p(T')}^\tau\\
&\le cd(T')^{\tau s}\|S_1-S_2\|_{L^p(T')}^\tau.
\end{align*}


{\em Case 1 (c).}
If $[x, x+kh]\not\subset \cup_{v\in\cV_{T}} B_v$
and $[x, x+kh]$ intersects $D_1$ or $D_2$, then
$\delta_1 > 2kt/\ct >2kt$ or $\delta_2 >2kt$ and hence $[x, x+kh]\subset \tri_{E_1}\cap\tri_{E_2}$,
which implies $\Delta_h^k S_1(x)=0$.
No estimate is needed.

\smallskip

{\em Case 1 (d).}
Let $\ITh\subset T$ be the quadrilateral bounded by the segments
$L_1$, $L_1-kh$, $D_1$ and
the line with equation $x=v_2+uh$, $u\in\R$,
where $v_2$ is the point of intersection of $L_2$ with $D_1$,
whenever this straight line intersects $L_1$.
If the line $x=v_2+uh$, $u\in\R$, does not intersect $L_1$,
then we replace it with the line $x=v_4+uh$, $u\in\R$.
Furthermore, we subtract $B_{v_1}$ and $B_{v_2}$ from $\ITh$.

Set $\JTh:= \ITh+[0, kh]$.

A simple geometric argument shows that $|\JTh| \le 2\delta_1 kt$.

In estimating $\|\Delta_h^k S_1\|_{L^\tau(\ITh)}^\tau$ there are two subcases to be considered.

If $\delta_1(T) \le 2kt/\ct$, we use (\ref{est-Delta-S1-triv}) to obtain
\begin{align*}
\|\Delta_h^k S_1\|_{L^\tau(\ITh)}^\tau
&\le \|\Delta_h^k S_2\|_{L^\tau(\ITh)}^\tau
+\|S_1-S_2\|_{L^\tau(\ITh)}^\tau
+\|S_1-S_2\|_{L^\tau(\JTh\cap A_1)}^\tau.
\end{align*}
We estimate the above norms quite like in {\em Case 1 ($b$)}, using Lemma~\ref{lem:Nikolski-Markov}.
We have
\begin{align*}
&\|S_1-S_2\|_{L^\tau(\ITh)}^\tau
\le c(|\ITh|/|T|)\|S_1-S_2\|_{L^\tau(T)}^\tau\\
&\le ct\delta_1(T)[\delta_1(T)d(T)]^{-1}\|S_1-S_2\|_{L^\tau(T)}^\tau
\le ct d(T)^{-1}|T|^{\tau s/2}\|S_1-S_2\|_{L^p(T)}^\tau\\
&\le ct d(T)^{-1}(\delta_1(T) d(T))^{\tau s/2}\|S_1-S_2\|_{L^p(T)}^\tau
\le ct^{1+\tau s/2} d(T)^{\tau s/2-1}\|S_1-S_2\|_{L^p(T)}^\tau.
\end{align*}
For the second norm we get
\begin{align*}
\|S_1-S_2\|_{L^\tau(\JTh\cap A_1)}^\tau
\le c|\JTh|\|S_1-S_2\|_{L^\infty(A_1)}^\tau
&\le ct^2|A_1|^{-\tau/p} \|S_1-S_2\|_{L^p(A_1)}^\tau\\
\le ct^2d(A_1)^{-2\tau/p} \|S_1-S_2\|_{L^p(A_1)}^\tau
&= ct^2d(A_1)^{\tau s-2} \|S_1-S_2\|_{L^p(A_1)}^\tau,
\end{align*}
where as before we used the fact that $2\tau/p=2-\tau s$.

From the above estimates we infer
\begin{align*}
\|\Delta_h^k S_1\|_{L^\tau(\ITh)}^\tau
\le \|\Delta_h^k S_2\|_{L^\tau(\ITh)}^\tau
&+ ct^{1+\tau s/2} d(T)^{\tau s/2-1}\|S_1-S_2\|_{L^p(T)}^\tau\\
&+ ct^2d(A_1)^{\tau s-2} \|S_1-S_2\|_{L^p(A_1)}^\tau.
\end{align*}


Let $\delta_1(T) > 2kt/\ct$.
We use (\ref{rep-Delta-k}) to obtain
\begin{align*} 
|\Delta_h^k S_1(x)|
&\le |\Delta_h^k S_2(x)|+|\Delta_h^k (S_1-S_2)(x)|\\
&\le |\Delta_h^k S_2(x)|+ ct^{k-1}\|D_\nn^{k-1}(S_1-S_2)\|_{L^\infty([x, x+kh])},
\end{align*}
implying
\begin{align*}
\|\Delta_h^k S_1\|_{L^\tau(\ITh)}^\tau
\le \|\Delta_h^k S_2\|_{L^\tau(\ITh)}^\tau
&+ c |\ITh|t^{\tau(k-1)}\|D_\nn^{k-1} (S_1-S_2)\|_{L^\infty(\ITh\cap T)}^\tau\\
&+ c |\ITh|t^{\tau(k-1)}\|D_\nn^{k-1} (S_1-S_2)\|_{L^\infty(A_1)}^\tau.
\end{align*}
Clearly,
\begin{align*}
&\|D_\nn^{k-1} (S_1-S_2)\|_{L^\infty(\ITh\cap T)}
\le c\delta_1(T)^{-(k-1)}\|S_1-S_2\|_{L^\infty(T)}\\
&\le c\delta_1(T)^{-(k-1)}|T|^{-1/p}\|S_1-S_2\|_{L^p(T)}
\le c\delta_1(T)^{-(k-1)-2/p}\|S_1-S_2\|_{L^p(T)},
\end{align*}
and
\begin{align}\label{deriv-A1}
\|D_\nn^{k-1} (S_1-S_2)\|_{L^\infty(A_1)}
&\le cd(A_1)^{-(k-1)}\|S_1-S_2\|_{L^\infty(A_1)}\\
&\le cd(A_1)^{-(k-1)-2/p}\|S_1-S_2\|_{L^p(A_1)}.\notag
\end{align}
Therefore,
\begin{align*}
\|\Delta_h^k S_1\|_{L^\tau(\ITh)}^\tau
&\le \|\Delta_h^k S_2\|_{L^\tau(\ITh)}^\tau
+ ct^{1+\tau(k-1)}\delta_1(T)^{1-\tau(k-1)-2\tau/p}\|S_1-S_2\|_{L^p(T)}^\tau\\
&+ ct^{1+\tau(k-1)}d(A_1)^{1-\tau(k-1)-2\tau/p}\|S_1-S_2\|_{L^p(A_1)}^\tau.
\end{align*}


{\em Case 1 (e) (Main).}
Let $T_h^\star \subset T_h$ be the set defined by
\begin{equation}\label{def-T-star}
T_h^\star:=\Big\{x\in T_h:
[x, x+kh]\cap L_1 \ne \emptyset
\quad\hbox{and}\quad
[x,x+ kh] \not\subset \ITh\bigcup_{v\in\cV_{T}} B_v\Big\}.
\end{equation}
We next estimate
$\|\Delta_h^k S_1\|_{L^\tau(T_h^\star)}^\tau$.

Let $x\in T_h^\star$.
Denote by $b_1$ and $b_2$ the points where the line through $x$ and $x+kh$
intersects $L_1$ and $L_2$.
Set $b=b(x):=b_2-b_1$.
We associate the segment $[x+b, x+b +kh]$ to $[x, x+kh]$ and
$\Delta^k_hS_2(x+b)$ to $\Delta^k_hS_1(x)$.

Since $S_1\in \Pi_k$ on $\tri_{E_1}$ we have
$D_\nn^{k-1}S_1(y) = \constant$ on $[b_1, x+b]$
and hence
\begin{equation}\label{S1-I3}
D_\nn^{k-1}S_1(b_1-u\nn)= D_\nn^{k-1}S_1(b_2-u\nn) \quad\hbox{for}\quad 0\le u\le |x-b_1|.
\end{equation}
Similarly, since $S_2\in \Pi_k$ on $\tri_{E_2}$ we have
$D_\nn^{k-1}S_2(y) = \constant$ on $[x+kh, b_2]$ and hence
\begin{equation}\label{S2-I3}
D_\nn^{k-1}S_2(b_1+u\nn)= D_\nn^{k-1}S_2(b_2+u\nn) \quad\hbox{for}\quad 0\le u\le |x+kh-b_1|.
\end{equation}
We use (\ref{rep-Delta-k}) and (\ref{S1-I3}) - (\ref{S2-I3}) to obtain
\begin{align*}
\Delta^k_hS_1(x)
& = |h|^{k-1}\int_{|b_1-x|}^{k|h|} D_\nn^{k-1} S_1(x+u\nn)M^*_k(u) du\\
&+ |h|^{k-1}\int_0^{|b_1-x|} D_\nn^{k-1} S_1(x+u\nn)M^*_k(u) du\\
& = |h|^{k-1}\int_{|b_1-x|}^{k|h|} D_\nn^{k-1} S_1(x+u\nn)M^*_k(u) du\\
&+ |h|^{k-1}\int_0^{|b_1-x|} D_\nn^{k-1} S_1(x+b+u\nn)M^*_k(u) du
\end{align*}
and
\begin{align*}
\Delta^k_hS_2(x+b)
& = |h|^{k-1}\int_{|b_1-x|}^{k|h|} D_\nn^{k-1} S_2(x+b+u\nn)M^*_k(u) du\\
&+ |h|^{k-1}\int_0^{|b_1-x|} D_\nn^{k-1} S_2(x+b+u\nn)M^*_k(u) du\\
& = |h|^{k-1}\int_{|b_1-x|}^{k|h|} D_\nn^{k-1} S_2(x+u\nn)M^*_k(u) du\\
&+ |h|^{k-1}\int_0^{|b_1-x|} D_\nn^{k-1} S_2(x+b+u\nn)M^*_k(u) du.
\end{align*}
Therefore,
\begin{align*}
\Delta^k_hS_1(x)
& = \Delta^k_hS_2(x+b) + \Delta^k_h (S_1-S_2)(x)  \notag\\
& = \Delta^k_hS_2(x+b)
+|h|^{k-1}\int_{|b_1-x|}^{k|h|} D_\nn^{k-1} [S_1-S_2]\big(x+u\nn\big)M^*_k(u) du\\
&+ |h|^{k-1}\int_0^{|b_1-x|} D_\nn^{k-1} [S_1-S_2]\big(x+b+u\nn\big)M^*_k(u) du  \notag
\end{align*}
and hence
\begin{align}\label{rep-DeltaS-k-1}
|\Delta^k_hS_1(x)|
& \le |\Delta^k_hS_2(x+b)|
+ct^{k-1}\|D_\nn^{k-1}(S_1-S_2)\|_{L^\infty([b_1, x+kh])}\\
&+ ct^{k-1}\|D_\nn^{k-1}(S_1-S_2)\|_{L^\infty([x+b, b_2])}  \notag
\end{align}
The key here is that $([b_1, x+kh]\cup [x+b, b_2])\cap T^\circ=\emptyset$.

Let $T_h^{\star\star}:=\{x+b(x): x\in T_h^{\star}\}$,
where $b(x)$ is defined above.
By (\ref{rep-DeltaS-k-1}) we get
\begin{align*}
\|\Delta_h^k S_1\|_{L^\tau(T_h^\star)}^\tau
&\le \|\Delta_h^k S_2\|_{L^\tau(T_h^{\star\star})}^\tau
+ c t d(A_1)t^{\tau(k-1)}\|D_\nn^{k-1} (S_1-S_2)\|_{L^\infty(A_1)}^\tau\\
&+ c t d(A_2)t^{\tau(k-1)}\|D_\nn^{k-1} (S_1-S_2)\|_{L^\infty(A_2)}^\tau.
\end{align*}
Just as (\ref{deriv-A1}) we have
\begin{align*}
\|D_\nn^{k-1} (S_1-S_2)\|_{L^\infty(A_1)}
&\le cd(A_1)^{-(k-1)}\|S_1-S_2\|_{L^\infty(A_1)}\\
&\le cd(A_1)^{-(k-1)-2/p}\|S_1-S_2\|_{L^p(A_1)},\notag
\end{align*}
and similar estimates hold with $A_1$ replaced by $A_2$.
We use these above to obtain
\begin{align*}
\|\Delta_h^k S_1\|_{L^\tau(T_h^\star)}^\tau
&\le \|\Delta_h^k S_2\|_{L^\tau(T_h^{\star\star})}^\tau
+ ct^{1+\tau(k-1)}d(A_1)^{1-\tau(k-1)-2\tau/p}\|S_1-S_2\|_{L^p(A_1)}^\tau\\
&+ ct^{1+\tau(k-1)}d(A_2)^{1-\tau(k-1)-2\tau/p}\|S_1-S_2\|_{L^p(A_2)}^\tau.
\end{align*}
It is an important observation that no part of
$\|\Delta_h^k S_2\|_{L^\tau(T_h^{\star\star})}^\tau$
has been used for estimation of quantities $\|\Delta_h^k S_1\|_{L^\tau(\cdot)}^\tau$
from previous cases.


Putting all of the above estimates together we arrive at
\begin{align}\label{est-case-1}
\|\Delta_h^k S_1\|_{L^\tau(T_h)}^\tau \le \|\Delta_h^k S_2\|_{L^\tau(T_h)}^\tau
+Y_1+Y_2+ Y_3 + Y_4,
\end{align}
where
\begin{align*}
Y_1 &:= \sum_{A\in \cA_T^t} ct^2d(A)^{\tau s-2}\|S_1-S_2\|_{L^p(A)}^\tau
 + \sum_{A\in \fA_T^t} cd(A)^{\tau s}\|S_1-S_2\|_{L^p(A)}^\tau,\\
\end{align*}
\begin{multline*}
Y_2 :=  ct^{1+\tau(k-1)}d(A_1)^{1-\tau(k-1)-2\tau/p}\|S_1-S_2\|_{L^p(A_1)}^\tau\\
+ ct^{1+\tau(k-1)}d(A_2)^{1-\tau(k-1)-2\tau/p}\|S_1-S_2\|_{L^p(A_2)}^\tau,
\end{multline*}
%
%
\begin{multline*}
Y_3 := \sum_{T'\in\cT_T^t}
ct^{1+\tau s/2}d(T')^{\tau s/2-1}\|S_1-S_2\|_{L^p(T')}^\tau\\
+ \sum_{T'\in\fTT_T^t} cd(T')^{\tau s}\|S_1-S_2\|_{L^p(T')}^\tau
 + ct^{1+\tau s/2} d(T)^{\tau s/2-1}\|S_1-S_2\|_{L^p(T)}^\tau,
\end{multline*}
%
and
\begin{align*}
Y_4 := ct^{1+\tau(k-1)}\delta_1(T)^{1-\tau(k-1)-2\tau/p}\|S_1-S_2\|_{L^p(T)}^\tau,
\quad\hbox{if}\;\;\delta_1(T)>2kt/\ct,
\end{align*}
otherwise $Y_4:=0$.

\begin{rem}\label{rem:main}
In all cases we considered above but {\em Case 1 (e)} we used the simple inequality
$|\Delta^k_hS_1(x)| \le |\Delta^k_hS_2(x)| + |\Delta^k_h(S_1-S_2)(x)|$
to estimate $\|\Delta_h^k S_1\|_{L^\tau(G)}^\tau$ for various sets $G$
and this works because these sets are of relatively small measure.
As Example~\ref{example-1} shows this approach in principle cannot be used in {\em Case 1 (e)}
and this is the main difficulty in this proof.
The gist of our approach in going around is to estimate $|\Delta^k_hS_1(x)|$
by using $|\Delta^k_hS_2(x+b)|$ with some shift $b$,
where $|\Delta^k_hS_2(x+b)|$ is not used to estimate other terms $|\Delta^k_hS_1(x')|$
(there is a one-to-one correspondence between these quantities).
\end{rem}

\subsection*{Case 2}

Let $\Omega_h^\star$ be the set of all $x\in \Omega$ such that
$[x, x+kh] \subset \Omega$,
$[x, x+kh]\cap A \ne \emptyset$ for some $A\in\cA$ with $d(A) > 2kt/\ct$,
and
$[x, x+kh] \cap T = \emptyset$ for all $T\in\cT$ with $d(T) \ge 2kt/\ct$.

Denote by $\cV_A$ the set of all vertices on $\partial A$ and
set $B_v:=B(v, 4kt/\ct)$, $v\in\cV_A$.

We next indicate how we estimate $|\Delta_h^k S_1(x)|$ in different cases.

{\em Case 2 (a).}
If $[x, x+kh] \subset A$, then
$
\Delta_h^kS_1(x)= \Delta_h^kS_2(x) =0
$
and no estimate is needed.

\smallskip

{\em Case 2 (b).}
If $[x, x+kh] \subset \cup_{v\in\cV_A} B(v, 2kt/\ct)$,
we estimate $|\Delta_h^k S_1(x)|$ trivially:
\begin{equation*} 
|\Delta_h^k S_1(x)|
\le |\Delta_h^k S_2(x)| + 2^k\sum_{\ell=0}^k|S_1(x+\ell h)-S_2(x+\ell h)|.
\end{equation*}

\smallskip

{\em Case 2 (c).}
Let $[x, x+kh]$ intersects the edge $E=:[w_1, w_2]$ from $\partial A$,
that is shared with $A'\in\cA$
and $[x, x+kh] \not\subset \cup_{v\in\cV_A} B_v$.
Let $y:= E\cap[x, x+kh]$.
Evidently, $|y-w_j| > kt/\ct$, $j=1,2$, and in light of Lemma~\ref{lem:property-2} we have
$[x, x+kh] \subset B(y, kt) \subset A\cup A'$.
In this case we use the inequality
\begin{align*} 
|\Delta_h^k S_1(x)|
&\le |\Delta_h^k S_2(x)|+|\Delta_h^k (S_1-S_2)(x)|\\
&\le |\Delta_h^k S_2(x)|+ ct^{k-1}\|D_\nn^{k-1}(S_1-S_2)\|_{L^\infty([x, x+kh])}\notag,
\end{align*}
which follows by (\ref{rep-Delta-k}).

The case when $[x, x+kh]$ intersects an edge from $\partial A$
that is shared with some $T\in\cT$ is covered in Case 1 above.

We proceed further similarly as in Case 1
and in the proof of Theorem~\ref{thm:smooth-Bernstein-1}
to obtain
\begin{align}\label{est-case-2}
\|\Delta_h^kS_1\|_{L^\tau(\Omega_t^\star)}^\tau
\le \|\Delta_h^kS_2\|_{L^\tau(\Omega_t^\star)}^\tau + Y_1+Y_2,
\end{align}
where
\begin{multline*}
Y_1 := \sum_{A\in\cA: d(A) \ge 2kt/\ct}t^{1+\tau(k-1)}
cd(A)^{1-\tau(k-1)-2\tau/p}\|S_1-S_2\|_{L^p(A)}^\tau \notag\\
+ \sum_{A\in\cA: d(A) \ge 2kt/\ct} ct^2 d(A)^{\tau s-2}\|S_1-S_2\|_{L^p(A)}^\tau
\end{multline*}
and
\begin{multline*}
Y_2
:= \sum_{A\in\cA: d(A) \le 2kt/\ct} cd(A)^{\tau s}\|S_1-S_2\|_{L^p(A)}^\tau\\
+ \sum_{T\in\cT: d(T) \le 2kt/\ct} cd(T)^{\tau s}\|S_1-S_2\|_{L^p(T)}^\tau.
\end{multline*}

\subsection*{Case 3}

Let $\Omega_h^{\star\star}$ be the set of all $x\in \Omega$ such that
$$
[x, x+kh] \subset \cup \{A\in\cA: d(A) \le 2kt/\ct\} \cup \{T\in\cT: d(T) \le 2kt/\ct\}.
$$
In this case we estimate $|\Delta_h^k S_1(x)|$ trivially
just as in (\ref{est-Delta-S1-triv}).
We obtain
\begin{align*}
\|\Delta_h^kS_1\|_{L^\tau(\Omega_h^{\star\star})}^\tau
&\le \|\Delta_h^kS_2\|_{L^\tau(\Omega_h^{\star\star})}^\tau
+ \sum_{A\in\cA: d(A) \le 2kt/\ct}c\|S_1-S_2\|_{L^\tau(A)}^\tau\\
&+ \sum_{T\in\cT: d(T) \le 2kt/\ct} c\|S_1-S_2\|_{L^\tau(T)}^\tau\\
&\le \|\Delta_h^kS_2\|_{L^\tau(\Omega_h^{\star\star})}^\tau
+ \sum_{A\in\cA: d(A) \le 2kt/\ct} cd(A)^{\tau s}\|S_1-S_2\|_{L^p(A)}^\tau\\
&+ \sum_{T\in\cT: d(T) \le 2kt/\ct} cd(T)^{\tau s}\|S_1-S_2\|_{L^p(T)}^\tau.
\end{align*}

Just as in the proof of Theorem~\ref{thm:Bernstein} it is important to note that
in the above estimates only finitely many norms may overlap at a time.
From above, (\ref{est-case-1}), and (\ref{est-case-2}) we obtain
\begin{align*}
\omega_k(S_1, t)_\tau^\tau
&\le \omega_k(S_2, t)_\tau^\tau + \bA_t + \bT_t,
\end{align*}
where
\begin{align*}
\bA_t
&:= \sum_{A\in\cA: d(A) > 2kt/\ct}t^{1+\tau(k-1)}
cd(A)^{1-\tau(k-1)-2\tau/p}\|S_1-S_2\|_{L^p(A)}^\tau \notag\\
&+ \sum_{A\in\cA: d(A) > 2kt/\ct} ct^2 d(A)^{\tau s-2}\|S_1-S_2\|_{L^p(A)}^\tau\\
&+ \sum_{A\in\cA: d(A) \le 2kt/\ct} cd(A)^{\tau s}\|S_1-S_2\|_{L^p(A)}^\tau.
\end{align*}
and
\begin{align*}
\bT_t
&:= \sum_{T\in\cT: \delta_1(T) > 2kt/\ct}
ct^{1+ \tau(k-1)}\delta_1(T)^{1-\tau(k-1)-2\tau/p}\|S_1-S_2\|_{L^p(T)}^\tau \notag\\
&+ \sum_{T\in\cT: \delta_2(T) > 2kt/\ct}
ct^{1+ \tau(k-1)}\delta_2(T)^{1-\tau(k-1)-2\tau/p}\|S_1-S_2\|_{L^p(T)}^\tau \notag\\
&+ \sum_{T\in\cT: d(T) > 2kt/\ct}
ct^{1+\tau s/2} d(T)^{\tau s/2-1}\|S_1-S_2\|_{L^p(T)}^\tau\\
&+ \sum_{T\in\cT: d(T) \le 2kt/\ct} cd(T)^{\tau s}\|S_1-S_2\|_{L^p(T)}^\tau.
\end{align*}

We insert this estimate in (\ref{def-smooth-Besov}) and
interchange the order of integration and summation to obtain
\begin{align*}
|S_1|_{B^{s,k}_\tau}^\tau \le |S_2|_{B^{s,k}_\tau}^\tau + Z_1 + Z_2,
\end{align*}
where
\begin{align*}
Z_1
&:= c \sum_{A\in\cA}  d(A)^{1-\tau(k-1)-2\tau/p}\|S_1-S_2\|_{L^p(A)}^\tau
\int_0^{\ct d(A)/2k} t^{-\tau s+\tau(k-1)} dt \\
& + c\sum_{A\in\cA} d(A)^{\tau s-2}\|S_1-S_2\|_{L^p(A)}^\tau \int_0^{\ct d(A)/2k} t^{-\tau s+1} dt\\
&+ c\sum_{A\in\cA} d(A)^{\tau s}\|S_1-S_2\|_{L^p(A)}^\tau \int_{\ct d(A)/2k}^\infty t^{-\tau s-1} dt
\end{align*}
and
\begin{align*}
Z_2
&:= c \sum_{T\in\cT}  \delta_1(T)^{1-\tau(k-1)-2\tau/p}\|S_1-S_2\|_{L^p(T)}^\tau
\int_0^{\ct \delta_1(T)/2k} t^{-\tau s+ \tau(k-1)} dt \\
&+ c \sum_{T\in\cT}  \delta_2(T)^{1-\tau(k-1)-2\tau/p}\|S_1-S_2\|_{L^p(T)}^\tau
\int_0^{\ct \delta_2(T)/2k} t^{-\tau s+ \tau(k-1)} dt \\
& + c\sum_{T\in\cT} d(T)^{\tau s/2-1}\|S_1-S_2\|_{L^p(T)}^\tau \int_0^{\ct d(T)/2k} t^{-\tau s/2} dt\\
&+ c\sum_{T\in\cT} d(T)^{s\tau}\|S_1-S_2\|_{L^p(T)}^\tau \int_{\ct d(T)/2k}^\infty t^{-\tau s-1} dt.
\end{align*}
Observe that $-\tau s + \tau(k-1) > -1$ is equivalent to $s/2<k-1+1/p$
which holds true by the hypothesis,
and $-\tau s/2>-1$ is equivalent to  $s<2/\tau = s+2/p$
which is obvious.
Therefore, all integrals above are convergent and taking into account that
$2-2\tau/p-\tau s = 2\tau(1/\tau-1/p-s/2) = 0$ we obtain
\begin{align*}
|S_1|_{B^{s,k}_\tau}^\tau
&\le |S_2|_{B^{s,k}_\tau}^\tau + c\sum_{A\in\cA\cup \cT} \|S_1-S_2\|_{L^p(A)}^\tau\\
&\le |S_2|_{B^{s,k}_\tau}^\tau + cn^{\tau(1/\tau-1/p)}\Big(\sum_{A\in\cA\cup\cT} \|S_1-S_2\|_{L^p(A)}^\tau\Big)^{\tau/p}\\
&=|S_2|_{B^{s,k}_\tau}^\tau + cn^{\tau s/2}\|S\|_{L^p(\Omega)}^\tau,
\end{align*}
where we used H\"{o}lder's inequality.
This completes the proof of Theorem~\ref{thm:smooth-Bernstein-2}.
\qed

\medskip

\noindent
{\bf Acknowledgment.}
We would like to give credit to Peter Petrov (Sofia University)
with whom the second author discussed the theme of this article some years ago.

\end{document}